\documentclass[11pt]{article}

\usepackage[margin=1in]{geometry}
\usepackage{amsmath,amssymb,amsthm,mathtools,bm}
\usepackage{mathrsfs}
\usepackage{array}
\usepackage{booktabs}
\usepackage{enumitem}
\usepackage{microtype}
\usepackage{xcolor}
\usepackage{graphicx}
\usepackage{subcaption}
\usepackage[round]{natbib}
\usepackage[colorlinks=true,linkcolor=blue,citecolor=blue,urlcolor=blue]{hyperref}
\usepackage[nameinlink,capitalize]{cleveref}

\hypersetup{
  pdftitle={Two Sample Sizes in Test-Time LLM Ensembling: Buffered ERM and Minimax Theory},
  pdfauthor={Zhi Zhang and Arash A. Amini}
}

\newtheorem{theorem}{Theorem}[section]
\newtheorem{proposition}[theorem]{Proposition}
\newtheorem{lemma}[theorem]{Lemma}
\newtheorem{corollary}[theorem]{Corollary}
\newtheorem{assumption}[theorem]{Assumption}

\newtheorem{conjecture}[theorem]{Conjecture}
\newtheorem{remark}[theorem]{Remark}

\crefname{assumption}{Assumption}{Assumptions}
\Crefname{assumption}{Assumption}{Assumptions}
\crefname{conjecture}{Conjecture}{Conjectures}
\Crefname{conjecture}{Conjecture}{Conjectures}

\newcommand\pr{\mathbb P}
\newcommand{\Pp}{\mathbb P}
\newcommand{\E}{\mathbb E}
\newcommand{\R}{\mathbb R}

\newcommand{\one}{\mathbf 1}
\newcommand{\sgn}{\operatorname{sgn}}
\newcommand{\TV}{\operatorname{TV}}
\newcommand{\KL}{\operatorname{KL}}
\newcommand{\Bin}{\operatorname{Bin}}
\newcommand{\Ber}{\operatorname{Ber}}
\newcommand{\Var}{\operatorname{Var}}
\newcommand{\calQ}{\mathcal Q}
\newcommand{\calA}{\mathcal A}

\newcommand{\calG}{\mathcal G}
\newcommand{\calF}{\mathcal F}

\newcommand{\calR}{\mathcal R}

\newcommand{\eps}{\varepsilon}

\newcommand{\abs}[1]{\left|#1\right|}
\DeclarePairedDelimiter{\norm}{\lVert}{\rVert}
\newcommand{\inner}[2]{\left\langle#1,#2\right\rangle}
\newcommand{\dd}{\,\mathrm d}
\newcolumntype{L}[1]{>{\raggedright\arraybackslash}p{#1}}

\DeclareMathOperator*{\argmin}{argmin}

\title{Two Sample Sizes in Test-Time LLM Ensembling:\\
Buffered ERM and Minimax Theory}
\author{Zhi Zhang\quad and\quad
Arash A. Amini\thanks{Emails: \href{mailto:zzh237@g.ucla.edu}{\texttt{zzh237@g.ucla.edu}} and \href{mailto:aaamini@ucla.edu}{\texttt{aaamini@ucla.edu}}.}\\[3pt]
{\normalsize Department of Statistics and Data Science, UCLA}}
\date{}

\begin{document}
\maketitle

\begin{abstract}
Test-time scaling can improve the performance of large language models by generating multiple answers to a question and aggregating them.
Weighted ensembles extend this idea by drawing answers from several models, with the weights determining how often each model is used.
Learning these weights from a training dataset involves two sample sizes: $Q$ labeled questions and $N$ generations per model and question.
We study their interplay through the clean risk, the population error of the ensemble's infinite-generation voting rule.
Under a margin condition with exponent $\beta$, we show that empirical risk minimization (ERM) with a calibrated margin buffer achieves excess risk of order $Q^{-1/2}+N^{-\beta/2}$, up to logarithmic factors, for fixed numbers of models and candidate answers.
To investigate the optimality of this dependence on $Q$ and $N$, we analyze a simplified submodel with two models and two possible answers per question.
We characterize its minimax excess risk up to constants by a Hellinger modulus for a binomial-mixture experiment.
Polynomial approximation and spectral analysis reveal three qualitatively different regimes.
When questions are relatively scarce, the $Q^{-1/2}$ term governs learning and the general buffered ERM is optimal.
As questions become more abundant, we reach an intermediate regime in which pooling generation counts across questions improves upon estimating each answer distribution separately, with only logarithmic gains from additional questions.
With sufficiently many questions, we enter the third regime where the minimax excess risk reaches an $N^{-\beta}$ identification floor that persists even with infinitely many questions.
The submodel's lower bounds also apply to general ensembles, matching the buffered-ERM upper bound up to logarithmic factors when questions are scarce and in parts of the intermediate regime.
Questions remain about the transitions between regimes and the construction of matching pooled estimators for general ensembles.

\end{abstract}

\section{Introduction}

Generating more than one response at inference time has become a standard way to improve the performance of large language models.  A natural way to combine repeated generations is by \emph{plurality voting}, that is, returning the final answer that appears most often. This voting rule is used in self-consistency~\citep{WangSelfConsistency}.  Alternatively, a verifier or reward model can score the candidates and select the highest-scoring one, as in best-of-$N$ selection \citep{CobbeVerifiers,HuangBestN,MukherjeeOT}.  Although these procedures use the samples differently, they share a statistical question: how much of the model's infinite-generation behavior can be recovered from finitely many generations?  One way to study this question is to consider test-time procedures with parameters learned from data; we introduce one such procedure below.

For such a learning problem, a first sample size is $N$, the number of generations used to estimate a model's answer distribution on each training question.  There is, however, a second sample size that is easy to overlook.  A test-time procedure is learned or evaluated on a finite benchmark, but its intended performance concerns a broader population of questions.  We model the benchmark questions as $Q$ draws from a population distribution $\calQ$.  Thus $Q$ controls generalization beyond the observed questions, whereas $N$ controls how well the answer distribution on each question is known.  The two sampling layers are nested: a benchmark score averages over questions, but each term in that average is itself computed from finitely many generations.

We study this interaction in the weighted multi-model setting of Best-of-$\infty$ \citep{KomiyamaBoInf}, which combines answers by voting without a verifier at deployment.  Models may favor different wrong answers to the same question.  A mixture can then make the correct answer the most probable even when it is not the most probable under any individual model (see Table~\ref{tab:aime-example} for an empirical illustration).

Suppose $K$ fixed LLMs can be queried on a question $q$.  The question has a finite candidate-answer set $\calA(q)$, whose size is uniformly bounded by $A$, and model $i$ has an infinite-generation answer distribution $p_i(\cdot\mid q)$ on $\calA(q)$.  Given weights $w\in\Delta_K$, the ensemble distribution is
\begin{equation}
\label{eq:intro-mixture}
p_w(a\mid q)
:=
\sum_{i=1}^K w_i p_i(a\mid q).
\end{equation}
Practically, this corresponds to choosing a model independently according to $w$ for each generation and pooling the resulting answers; equivalently, one may allocate asymptotic fractions $w_i$ of the generation budget to the models.  The single-model problem is recovered as a special case.

The weights also determine generation cost.  If a generation from model $i$ costs $c_i$, the mean cost per generation is $\bar c(w):=\sum_{i=1}^K w_i c_i$.  An ensemble that retains good population voting accuracy while shifting weight toward cheaper models reduces this mean cost.

Assume each question $q$ has a designated gold answer $a^\star(q)\in\calA(q)$.  For fixed weights $w$, the associated \emph{ideal voting rule} selects the most probable answer under the infinite-generation mixture distribution $p_w(\cdot\mid q)$.  The rule is declared correct when $a^\star(q)$ is the unique maximizer, counting ties as errors.  Define the \emph{clean margin}
\begin{equation}
\label{eq:intro-clean-margin}
M_q(w)
:=
p_w(a^\star(q)\mid q)
-
\max_{a\ne a^\star(q)}p_w(a\mid q).
\end{equation}
The error probability of this ideal rule over the population of questions is therefore
\begin{equation}
\label{eq:intro-clean-risk}
R(w)
:=
\Pp_{q\sim\calQ}\{M_q(w)\le0\}.
\end{equation}
We call \eqref{eq:intro-clean-risk} the \emph{clean} risk because it treats the answer distributions $p_i(\cdot\mid q)$ as known.  Our statistical problem is to learn weights with small clean risk from finite benchmark data.  Our target is a clean-risk minimizer
\begin{equation}
\label{eq:intro-oracle}
w^\star\in\argmin_{w\in\Delta_K}R(w).
\end{equation}

Neither layer of this risk is observed directly.  In practice, we observe labeled questions
\[
q_1,\ldots,q_Q\overset{\mathrm{iid}}{\sim}\calQ
\]
where ``labeled'' means that the benchmark supplies each $q_j$ together with its gold answer $a^\star(q_j)$.  We also observe $N$ sampled answers from every model on every question.  The raw generations for model $i$ on question $q_j$ can be summarized by the multinomial count vector
\[
X_{j,i}
:=
(X_{j,i,a}:a\in\calA(q_j))
\sim
\operatorname{Multinomial}\bigl(N,p_i(\cdot\mid q_j)\bigr),
\qquad
\widehat p_i(\cdot\mid q_j)=\frac{X_{j,i}}{N}.
\]
Here, $X_{j,i,a}$ is the number of times model $i$ generated answer $a$ on question $q_j$.

To make these ideas concrete, Table~\ref{tab:aime-example} gives an illustrative example from
a real experiment.
The selected question, Problem~6 of the 2024 American Invitational Mathematics Examination (AIME~I), asks for the number of monotone paths across an $8\times8$ grid with exactly four changes of direction; its gold answer is $294$.  In $N=25$ generations per model, Gemma~3~12B produced $12$ distinct answers and Llama~4~Maverick produced $3$.  Their empirical modes are different wrong answers, $1225$ and $810$, while their equal-weight empirical mixture has the gold answer as its unique mode.  The clean rankings remain unknown and must be inferred from such finite samples.

\begin{table}[t]
\centering
\small
\setlength{\tabcolsep}{3pt}
\begin{tabular}{@{}l*{13}{r}r@{}}
\toprule
& \multicolumn{13}{c}{Answer (empirical frequencies in \%)} & \\
\cmidrule(lr){2-14}
Distribution & 70 & 147 & 273 & $294^\star$ & 495 & 810 & 1225 & 1820 & 2730 & 4116 & 7350 & 12870 & 44100 & \raisebox{4pt}[0pt][0pt]{\shortstack[r]{Observed\\[1pt]support}} \\
\midrule
Gemma 3 12B & 4 & 12 & 0 & 20 & 8 & 4 & \textbf{24} & 4 & 4 & 4 & 4 & 8 & 4 & 12 \\
Llama 4 Maverick & 0 & 0 & 4 & 44 & 0 & \textbf{52} & 0 & 0 & 0 & 0 & 0 & 0 & 0 & 3 \\
Equal mixture & 2 & 6 & 2 & \textbf{32} & 4 & 28 & 12 & 2 & 2 & 2 & 2 & 4 & 2 & 13 \\
\bottomrule
\end{tabular}
\normalsize
\caption{Empirical answer frequencies for AIME~2024~I, Problem~6, with $N=25$ generations per model at temperature $0.8$.  Both models were asked to reason step by step and return a boxed final answer; all $50$ responses contained complete boxed answers and finished below their token caps.  Every observed answer value is shown.  The star marks the gold answer, zeros denote unobserved answers, and boldface marks each empirical mode.  The last row is the equal-weight mixture.}
\label{tab:aime-example}
\end{table}

If the clean distributions $p_i$ were known, ordinary empirical risk minimization over the sampled questions would minimize
\begin{equation}
\label{eq:intro-clean-erm}
\widehat R_Q^{\mathrm{clean}}(w)
:=
\frac1Q\sum_{j=1}^Q\one\{M_{q_j}(w)\le0\}.
\end{equation}
In practice, $p_i$ is replaced by $\widehat p_i$, giving the empirical margin $\widehat M_{q_j}(w)$ which replaces the clean margin in~\eqref{eq:intro-clean-erm}.
Finite-$N$ noise can then change whether a question is considered solved at a given $w$.  In particular, a favorable fluctuation may produce $\widehat M_{q_j}(w)>0$ even when $M_{q_j}(w)\le0$, causing zero-margin ERM to record an empirical success where the clean ensemble makes an error.  In other words, zero-margin ERM can exploit favorable fluctuations in the empirical margins and effectively overfit to the finite-$N$ generation noise.

This motivates the buffered estimator
\begin{equation}
\label{eq:intro-buffered-estimator}
\widehat w_\eps
\in
\argmin_{w\in\Delta_K}
\frac1Q\sum_{j=1}^Q
\one\{\widehat M_{q_j}(w)\le\eps\}.
\end{equation}
The choice $\eps=0$ recovers the Best-of-$\infty$ ERM introduced by \citet{KomiyamaBoInf}.  When $\eps>0$, a training question is counted as solved only if its empirical margin exceeds $\eps$; empirical margins in $(0,\eps]$ are treated as errors.  On the event
\begin{equation}
\label{eq:intro-uniform-event}
\sup_{j\le Q}\sup_{w\in\Delta_K}
\bigl|\widehat M_{q_j}(w)-M_{q_j}(w)\bigr|
\le\eps_N,
\end{equation}
we have, simultaneously for every training question and every weight vector,
\begin{equation}
\label{eq:intro-sandwich}
\one\{M_{q_j}(w)\le0\}
\le
\one\{\widehat M_{q_j}(w)\le\eps_N\}
\le
\one\{M_{q_j}(w)\le2\eps_N\}.
\end{equation}
The left inequality prevents the learned rule from claiming empirically certified successes that fail under the clean margin.  The right inequality shows that the price of this protection is confined to a narrow clean-margin band.  This buffered sandwich is the basic mechanism behind our general result.  We also record a corollary for deployment with finitely many generations (\cref{cor:finite-deployment}).

\subsection{Overview of the results}
\label{sec:overview}

Let us now give an overview of the main results, starting with a general bound on the excess risk of the buffered ERM; questions about the optimality of this bound lead to a scalar submodel and its minimax analysis.

\paragraph{Buffered ERM for general ensembles.}
Take
\[
\eps_N
\asymp
\sqrt{\frac{\log(KAQ/\delta)}{N}},
\]
where $A$ bounds the number of candidate answers, and $\delta > 0$ is some specified confidence level.  Under a one-sided margin condition at the clean-risk minimizer,
\[
\Pp\{0<M_q(w^\star)\le t\}
\le C_m t^\beta,
\]
the estimator in \eqref{eq:intro-buffered-estimator} (with $\eps=\eps_N$) satisfies
\begin{equation}
\label{eq:intro-main}
R(\widehat w_{\eps_N})-R(w^\star)
\lesssim
\sqrt{\frac{K\log(AQ)+\log(1/\delta)}{Q}}
+
C_m\left(\frac{\log(KAQ/\delta)}{N}\right)^{\beta/2}.
\end{equation}
The first term is generalization across questions.  The second is the cost of replacing population answer distributions by finite-generation estimates.  The proof combines \eqref{eq:intro-sandwich} with a hyperplane-arrangement bound for the clean margin classes.
Treating $K$, $A$, and the margin constants as fixed and suppressing logarithmic and confidence factors, the dependence on the two sample sizes is
\begin{equation}
\label{eq:intro-main-simplified}
R(\widehat w_{\eps_N})-R(w^\star)
=
\widetilde O\bigl(\, Q^{-1/2}+N^{-\beta/2}\, \bigr).
\end{equation}
This bound already exhibits a fast-rate effect in $N$.  Although each clean margin is estimated at the usual $N^{-1/2}$ scale, the margin condition converts this estimation error into an excess-risk contribution of order $N^{-\beta/2}$.  The exponent $\beta$ is fixed as $N$ and $Q$ grow, but it may be arbitrarily large; consequently, this term can decay faster than $N^{-1/2}$, or even $N^{-1}$.  This does not contradict standard parametric estimation limits: the excess risk depends only on questions near the zero-margin boundary, and the margin condition makes such questions increasingly rare.

The overall bound, of course, still contains the $Q^{-1/2}$ term.
This raises the key optimality question: to what extent is the dependence on $(Q,N)$ in \eqref{eq:intro-main-simplified} intrinsic, and in particular can the $N^{-\beta/2}$ term be improved by pooling information across questions?

\paragraph{A scalar model for the optimality question.}
To investigate this question, we isolate the simplest nontrivial submodel: querying two models producing two possible answers each.  Because the answer set may depend on $q$, we may relabel the gold answer on every question as ``$\mathrm{gold}$'' and the other answer as ``$\mathrm{other}$'' without loss of generality, that is, $\mathcal A(q) = \{\mathrm{gold}, \mathrm{other}\}$ for all $q$ and $a^\star(q) = \mathrm{gold}$.  Let $Z_q\in[-1,1]$ be a latent \emph{signed advantage} and suppose
\begin{equation}
\label{eq:intro-antisymmetric}
p_1(\mathrm{gold}\mid q)=\frac{1+Z_q}{2},
\qquad
p_2(\mathrm{gold}\mid q)=\frac{1-Z_q}{2}.
\end{equation}
For $w=(t,1-t)$, the clean margin is $(2t-1)Z_q$.  Thus every $t>1/2$ gives the same clean decisions: the ensemble is correct exactly when $Z_q>0$.  Every $t<1/2$ gives the opposite decisions and is correct exactly when $Z_q<0$, while $t=1/2$ ties on every question.  The learning problem therefore reduces to choosing between the two halves of the weight simplex, rather than estimating the magnitude of $t-1/2$.

Let $G$ denote the distribution of $Z_q$ when $q\sim\calQ$, that is, the distribution induced by the map $q\mapsto Z_q$.  Write $\sigma=+1$ for the action of choosing any $t>1/2$ and $\sigma=-1$ for choosing any $t<1/2$, and let $R_+(G)$ and $R_-(G)$ denote their respective population risks.  Their difference is the linear functional measuring the imbalance of population mass across zero:
\begin{equation}
\label{eq:intro-theta}
R_-(G)-R_+(G)
=
\theta(G)
:=
\int\sgn(z)\,\dd G(z).
\end{equation}
Hence the sign of $\theta(G)$ determines which half of the simplex is optimal, and $|\theta(G)|$ is the excess risk, or regret, incurred by choosing the wrong half.  Section~\ref{sec:scalar-model} gives the exact decision reduction; the minimax analysis below asks how well this sign can be learned from finite generations.

The general multinomial observation also simplifies.  If $X_{1q}$ and $X_{2q}$ count the gold answers from the two models, then
\[
S_q:=X_{1q}+N-X_{2q}
\]
is sufficient for $Z_q$.  For the sampled questions, write $Z_j:=Z_{q_j}$ and $S_j:=S_{q_j}$.  The reduced experiment can then be written as
\begin{equation}
\label{eq:intro-binomial-channel}
Z_1,\ldots,Z_Q\overset{\mathrm{iid}}{\sim}G,
\qquad
S_j\mid Z_j=z
\sim
\Bin\left(n,\frac{1+z}{2}\right),
\quad j=1,\ldots,Q,
\qquad n:=2N.
\end{equation}
Thus the raw generations first reduce to multinomial answer counts and then, under antisymmetry, to one binomial count per question.  We write $K_nG$ for the resulting marginal distribution of each $S_j$, where $K_n$ denotes the $n$-trial binomial channel.  The observed counts $S_1,\ldots,S_Q$ are iid from this nonparametric binomial-mixture distribution, and the ensemble problem has become indirect estimation of \eqref{eq:intro-theta} through the binomial channel.

The original learner also observes the question identities $q_1,\ldots,q_Q$.  We begin by studying the reduced \emph{count-only experiment}, in which only $S_1,\ldots,S_Q$ are retained.  This reduction is immediate for upper bounds, since an estimator may ignore the question identities, but it requires justification for lower bounds.  We later prove in \cref{thm:observed-transfer,cor:general-lower-transfer} that, over an unrestricted class of question-to-advantage maps, the count-only lower bounds remain valid when the questions are observed and therefore transfer to the general ensemble problem.

\paragraph{Exact scalar minimax characterization.}
On the two-sided margin class
\[
\calG_{\beta,L}
=
\{G: \; G([-t,t])\le Lt^\beta,\ 0<t\le1\},
\]
let $\calR^{\mathrm{dec}}_{Q,n}(\beta,L)$ denote the minimax expected regret for choosing between the two actions from $S_1,\ldots,S_Q$:
\[
\calR^{\mathrm{dec}}_{Q,n}(\beta,L)
:=
\inf_{\widehat\sigma}
\sup_{G\in\calG_{\beta,L}}
\E_G\!\left[
|\theta(G)|
\one\{\widehat\sigma\theta(G)<0\}
\right],
\qquad \widehat\sigma\in\{-1,+1\}.
\]
Writing $h$ for Hellinger distance, define
\[
\omega_{n,\beta,L}(\eps)
:=
\sup\left\{
|\theta(G)-\theta(H)|:
G,H\in\calG_{\beta,L},\
h(K_nG,K_nH)\le\eps
\right\}.
\]
This modulus measures the largest ambiguity in the functional $\theta$ among two mixing distributions whose observable count laws are within Hellinger distance $\eps$.  We prove the constant-factor characterization
\begin{equation}
\label{eq:intro-modulus}
\calR^{\mathrm{dec}}_{Q,n}(\beta,L)
\asymp
\omega_{n,\beta,L}(Q^{-1/2})
\end{equation}
for every pair $(Q,n)$.  Thus, up to constants, the minimax regret is exactly the ambiguity in $\theta$ remaining at Hellinger resolution $Q^{-1/2}$ after passage through the $n$-trial binomial channel.

Direct evaluations of the modulus~\eqref{eq:intro-modulus} help identify three regimes.  We begin with the two endpoints, in which one of $Q$ or $n$ remains fixed while the other tends to infinity.  At the finite-question endpoint, consider the hard-margin subclass of mixing distributions $G$ supported on $\{-1,+1\}$; this is equivalent to assuming that there are only ``two types'' of questions.  The count reveals the sign of $Z$ without error for every $n$, just as if infinitely many generations were available for each question.  Nevertheless, from only $Q$ sampled questions the learner must determine which of the two question types is more prevalent.  Evaluating the modulus on this subclass gives the unavoidable root-$Q$ lower bound uniformly in $(Q,n)$: $\calR^{\mathrm{dec}}_{Q,n}(\beta,L)\gtrsim Q^{-1/2}$.

At the opposite endpoint, fix $n$ and let $Q\to\infty$.  The observable count law $K_nG$ can then be learned arbitrarily accurately, so uncertainty from sampling questions disappears.  Yet $K_nG$ depends on $G$ only through its first $n$ moments.  Different mixing distributions can therefore induce exactly the same count law while having different values of $\theta$.  We show that the resulting zero-radius ambiguity is
\begin{equation}
\label{eq:intro-floor}
\omega_{n,\beta,L}(0)
\asymp
n^{-\beta}.
\end{equation}
Hence finitely many generations per question impose an $n^{-\beta}$ floor, no matter how large $Q$ becomes.  By monotonicity of the modulus and \eqref{eq:intro-modulus}, the same lower bound holds uniformly for all $(Q,n)$: $\calR^{\mathrm{dec}}_{Q,n}(\beta,L)\gtrsim n^{-\beta}$.

Between these two endpoints lies a subtler middle region.  Its behavior is governed by spectral properties of the Bernstein--Durrmeyer operator associated with the Bernstein polynomials $b_{k,n}(p)=\binom{n}{k}p^k(1-p)^{n-k}$ defining the binomial law.  We first approach this region from the upper-bound side, showing that the two endpoint rates $Q^{-1/2}$ and $n^{-\beta}$ continue to describe the minimax risk over nontrivial portions of the $(Q,n)$ plane.  We then use the same spectral structure to obtain matching lower bounds in a broad middle regime.

\paragraph{Pooling and the spectrum of the binomial channel.}
The endpoint bounds are attained by two different procedures.  A \emph{sign-then-aggregate} estimator replaces each count $S_j$ by $\sgn(S_j-n/2)$ and then chooses the majority side.  It satisfies
\[
\calR^{\mathrm{dec}}_{Q,n}(\beta,L)
\, \lesssim_{\beta,L}\,
Q^{-1/2}+n^{-\beta/2},
\]
the same dependence on $(Q,n)$ as buffered ERM in the general problem, up to logarithmic factors.  This procedure
compresses every question to one estimated sign before aggregating across questions.

With more questions, it can be better to retain the full count values.  A pooled polynomial-moment estimator approximates $\sgn$ by a degree-$d$ polynomial $P_d$ and uses each count $S_j$ to form an unbiased estimate of $P_d(Z_j)$, so their average estimates the functional $\int P_d\,\dd G$.  Its approximation bias is of order $d^{-\beta}$.  Its statistical variability is controlled by how strongly inversion of the binomial conditional-expectation operator $T_na:=\sum_{k=0}^n a_kb_{k,n}(\,\cdot\,)$ amplifies the polynomial modes of $P_d$; we call this inverse amplification the estimator's stability cost.  As shown in \cref{sec:durrmeyer}, the amplification factors are the inverse singular values of $T_n$, or equivalently the inverse square roots of the eigenvalues of $D_n=T_nT_n^*$, the Bernstein--Durrmeyer operator.  Writing $\kappa_{d,n}$ for the largest inverse singular value through degree $d$,
\[
\kappa_{d,n}
:=
\left[
\prod_{r=0}^{d-1}\frac{n+r+2}{n-r}
\right]^{1/2},
\]
we obtain the spectral upper bound
\begin{equation}
\label{eq:intro-spectral-upper}
\calR^{\mathrm{dec}}_{Q,n}(\beta,L)
\lesssim_{\beta,L}
\inf_{1\le d\le n}
\left\{
d^{-\beta}
+
\frac{\sqrt{n+1}\,\kappa_{d,n}}{\sqrt Q}
\right\}.
\end{equation}
For $d\le n/2$, $\log\kappa_{d,n}\asymp d^2/n$: higher-degree information is progressively more unstable, but it becomes estimable as the number of questions grows.

The same spectrum also produces matching lower bounds.  We construct pairs of mixing distributions whose first $d$ moments agree, whose targets differ by order $d^{-\beta}$, and whose remaining higher-order modes are strongly attenuated by the channel.  For $0<\beta\le1$, optimizing the cancellation degree gives the rigorous middle-regime result
\begin{equation}
\label{eq:intro-rigorous-middle}
\calR^{\mathrm{dec}}_{Q,n}(\beta,L)
\asymp_{\beta,L,\eta}
[n\log Q]^{-\beta/2},
\qquad
n^{\beta+1+\eta}\le Q\le e^{c_\eta n},
\end{equation}
for every fixed $\eta>0$.  The upper bound is attained by the pooled polynomial-moment estimator with degree $d\asymp\sqrt{n\log Q}$.  At the other endpoint, taking $d=n$ and using $\kappa_{n,n}\asymp2^nn^{-1/4}$ shows that the same estimator reaches the $n^{-\beta}$ floor once
\[
Q\gtrsim4^n n^{2\beta+1/2}.
\]

\paragraph{The rigorous rate picture and its scope.}
These results reveal three different sources of difficulty.  We call the range in which the unavoidable $Q^{-1/2}$ term dominates the \emph{question-limited regime}.  Once $Q$ is larger, pooling the count information can improve on per-question thresholding; we call this the \emph{spectral-pooling regime}.  Finally, when the count law is estimated accurately enough that finite-moment nonidentifiability becomes the bottleneck, the problem is \emph{identification-limited}.  The regions in which matching scalar rates are currently proved are summarized below; recall that $n=2N$.
\begin{center}
\small
\begin{tabular}{@{}L{0.20\linewidth}L{0.14\linewidth}L{0.28\linewidth}L{0.27\linewidth}@{}}
\toprule
Proved region & Minimax rate & Achieving procedure & Scope \\
\midrule
$Q\lesssim N^\beta$
& $Q^{-1/2}$
& Buffered ERM; scalar sign-then-aggregate
& Matching result for general ensembles, up to logarithmic factors \\
$N^{\beta+1+\eta}\lesssim Q\lesssim e^{c_\eta N}$
& $[N\log Q]^{-\beta/2}$
& Scalar pooled polynomial-moment estimator
& Scalar matching for $0<\beta\le1$; lower bound transfers generally \\
$Q\gtrsim4^{2N}\operatorname{poly}(N)$
& $N^{-\beta}$
& Scalar pooled polynomial-moment estimator
& Scalar matching; lower bound transfers generally \\
\bottomrule
\end{tabular}
\end{center}

The first row is a result for general ensembles: the buffered-ERM upper bound matches the transferred root-$Q$ lower bound, showing that buffered ERM is rate-optimal up to logarithmic factors in the question-limited regime.  In the other two rows, the matching upper procedures currently rely on the antisymmetric scalar structure.  There is nevertheless a useful general consequence in the polynomial-$Q$ portion of the middle regime when $0<\beta\le1$.  The transferred spectral lower bound is then of order $(N\log N)^{-\beta/2}$, while buffered ERM has excess risk of order $(\log N/N)^{\beta/2}$; the two differ by only a factor $(\log N)^\beta$.  Thus buffered ERM is minimax optimal up to logarithmic factors there.

This logarithmic near-optimality does not extend uniformly over the entire subexponential range.  For example, still for $0<\beta\le1$, if $Q=e^{N^a}$ with $0<a<1$, the transferred lower bound and the buffered-ERM upper bound scale respectively as $N^{-\beta(1+a)/2}$ and $N^{-\beta(1-a)/2}$.  Their ratio is $N^{a\beta}=(\log Q)^\beta$: it is polylogarithmic in $Q$, but polynomial in $N$.

The third row begins when $Q$ is exponential in $N$.  Indeed, for $Q=e^{cN}$, the finite-generation term in the buffered-ERM bound is of constant order, so the buffered-ERM theorem does not even guarantee a vanishing excess risk.  By contrast, in the antisymmetric scalar model, the pooled polynomial-moment estimator attains the $N^{-\beta}$ identification floor once the explicit exponential threshold in the table is met.  This rate does not follow from the buffered-ERM analysis, and constructing a matching pooled procedure for general ensembles remains open.

\paragraph{Conjectured completion of the scalar picture.}
The spectral upper and lower bounds suggest that the intrinsic stability cost at degree $d$ is $\kappa_{d,n}$ itself.  This leads to the conjecture
\[
\calR^{\mathrm{dec}}_{Q,n}(\beta,L)
\stackrel{?}{\asymp}_{\beta,L}
\inf_{1\le d\le n}
\left\{
d^{-\beta}
+
\frac{\kappa_{d,n}}{\sqrt Q}
\right\}.
\]
If true, it gives the following three-phase rate picture:
\begin{equation}
\label{eq:intro-three-phase}
\calR^{\mathrm{dec}}_{Q,n}(\beta,L)
\asymp_{\beta,L}
\begin{cases}
Q^{-1/2},
&Q\lesssim n^\beta,\\[0.8ex]
\bigl[n\log(eQ/n^\beta)\bigr]^{-\beta/2},
&n^\beta\lesssim Q\lesssim e^{cn},\\[0.8ex]
n^{-\beta},
&Q\gtrsim e^{cn}.
\end{cases}
\end{equation}
The notation $e^{cn}$ records only the exponential scale.  To be more precise, write
\[
Q=n^\beta e^{u_n},
\qquad
u_n=\log(Q/n^\beta).
\]
The question-limited regime corresponds to $u_n$ remaining bounded above, the genuinely spectral-pooling regime has $u_n\to\infty$ with $u_n=o(n)$, and the identification-limited regime has $u_n\gtrsim n$.  The conjectured middle-region rate is $(nu_n)^{-\beta/2}$; at the boundary $u_n\asymp n$, this becomes $n^{-\beta}$, matching the third-region rate.

The rigorous middle theorem currently begins at $u_n\ge(1+\eta)\log n$ for fixed $\eta>0$.  The conjecture predicts the middle rate for every $u_n\to\infty$, however slowly, and for all $\beta>0$.  Thus the main open questions for the scalar model are to resolve the transition immediately above $Q\asymp n^\beta$, to sharpen the spectral upper and lower bounds to the same scale, which includes removing the extra $\sqrt{n+1}$ factor from the upper bound in \eqref{eq:intro-spectral-upper}, and to extend the matching spectral lower theory beyond $\beta=1$.  For bounded $u_n$, the stabilized factor $1+u_n=\log(eQ/n^\beta)$ connects the middle expression continuously to the root-$Q$ phase.

\subsection{Related work}
\label{sec:related-work}

Our work follows the voting approach of self-consistency \citep{WangSelfConsistency}, which uses agreement among sampled final answers without requiring a verifier at deployment.  Experiments of \citet{BrownMonkeys} illustrate the gap between coverage, the probability that a sample collection contains a correct answer, and the accuracy of the selected answer.  Eventual coverage requires only positive probability on the correct answer; our ideal voting rule requires it to be the unique mode.  For a fixed model and dataset, \citet{WuInferenceScaling} analyze the limiting accuracy of voting and use concentration bounds to study its convergence as the number of generations grows.

The direct starting point for our work is Best-of-$\infty$ \citep{KomiyamaBoInf}.  That paper formulates the infinite-generation voting objective for mixtures of LLMs, expresses optimization over a benchmark as a mixed-integer linear program, and empirically studies learning and transferring ensemble weights.  It also develops a Bayesian stopping rule for approximating the modal answer with finitely many generations.
Whereas their study of learning and transferring ensemble weights is empirical, we develop a finite-sample theory of learning and generalization.  Our theory identifies the different statistical roles of the benchmark size $Q$ and generation depth $N$, and shows how their interplay governs the attainable learning rates.
The max-margin refinement in Best-of-$\infty$ can still exploit favorable fluctuations in empirical margins.  Our statistically calibrated buffer, on the other hand, controls their effect on learning, leading to clean excess-risk guarantees.  Our scalar minimax analysis uncovers phase transitions in the attainable learning rates as $Q$ and $N$ vary.  It reveals regimes in which pooling across questions improves learning, and an identification floor at finite $N$ that persists even with infinitely many questions.

Learned verifiers provide an alternative to voting by scoring candidate solutions \citep{CobbeVerifiers}.  With the generator and scoring model taken as given, \citet{HuangBestN} and \citet{MukherjeeOT} study how selection performance depends on the generator's coverage of high-quality responses and errors in their scores.  \citet{DiBestMajority} also use a reward model: they discard responses with low empirical frequency and return up to $k$ of the remaining responses with the highest predicted rewards.  For binary correctness rewards, their pass@$k$ criterion counts success whenever at least one returned response is correct.  Their minimax bounds concern this selection problem under coverage and reward-error assumptions; they do not address the joint $(Q,N)$ problem of learning ensemble weights across questions.

The two sampling layers also arise in distribution regression, where each covariate is a probability measure observed through a finite sample \citep{SzaboDistributionRegression}.  In ensemble learning, the underlying answer distributions determine a discontinuous loss through their ensemble rankings.  Our general-$K$ bound combines uniform control of their empirical estimates with a growth-function bound for the resulting polytopal decision rules.  Margin assumptions also play a central role in classification theory \citep{AudibertTsybakov,MassartNedelec}.  Here the one-sided condition at $w^\star$ limits the mass of questions on which the correct answer wins by a small amount; it controls the finite-$N$ term, while generalization across questions contributes the separate root-$Q$ term.

The scalar reduction leads to the binomial-mixture experiment studied by \citet{TianPopulations} and \citet{VinayakPopulations}, who establish optimal recovery rates under Wasserstein-$1$ loss.  \citet{BasuBinomial} construct exact pointwise confidence intervals for the mixing distribution's cumulative distribution function and quantiles.  Our target is the sign-imbalance functional $\theta(G)$, whose integrand is discontinuous at zero.  Wasserstein recovery guarantees do not directly provide our minimax rates, which depend on the mixing mass near this decision boundary.  We study estimation of this functional and the associated ensemble decision under a power-law bound on that mass.

The general framework for this scalar problem is indirect estimation of a linear functional over a convex model.  The modulus-via-testing theory of \citet{DonohoLiuII} relates minimax estimation to Hellinger continuity of the target.  \citet{JuditskyNemirovski} construct near-minimax estimators by convex programming for affine observation models, and \citet{PolyanskiyWu} characterize minimax quadratic risk through convex duality.  We give a self-contained finite-sample specialization for expected absolute loss in our categorical experiment, and use reflection symmetry to connect functional estimation to the ensemble decision.  Evaluating the resulting modulus as $Q$ and $N$ vary then reveals how sampling error interacts with the ambiguity left by observing only finitely many moments.

To obtain explicit rates, we relate this moment ambiguity to weighted polynomial approximation of the sign function.  The matching lower approximation bound uses the inverse theorem of \citet{Kopotun} for doubling weights with finitely many zeros and singularities.  For finite $Q$, we use the classical Bernstein--Durrmeyer spectrum \citep{DerriennicDurrmeyer,DurrmeyerSpectral} to quantify how estimating higher polynomial moments amplifies sampling error.  Combining this spectral information with the margin-dependent approximation error gives the pooling bounds and the rigorous intermediate regime described above.

\section{Buffered ERM for general ensembles}
\label{sec:general-model}
\label{sec:buffered}

We now return to the general $K$-model ensemble and formalize the observations used to learn its weights.  The $Q$ sampled questions and the $N$ generations from each model on each question produce two
different
sources of randomness.  We then define buffered ERM and establish its excess-risk guarantee.  Whether its dependence on $(Q,N)$ is optimal leads to the scalar submodel studied in the next section.

\subsection{Model and observations}

We retain the notation $p_i$, $p_w$, $M_q$, and $w^\star$ from the introduction.  Write $[m]=\{1,\dots,m\}$ and $\Delta_K:=\{w\in\R_+^K:\sum_{i=1}^K w_i=1\}$.  We assume throughout that $a^\star(q)\in\calA(q)$ and $2\le |\calA(q)|\le A<\infty$.  For later reference, define the per-question clean loss
\begin{equation}
\label{eq:clean-loss-risk}
\ell_q(w)=\one\{M_q(w)\le0\},
\qquad
R(w)=\E_{q\sim\calQ}\ell_q(w)
=\Pp_{q\sim\calQ}\{M_q(w)\le0\}.
\end{equation}

The observed data consist of the labeled questions and the raw generations
\begin{equation}
\label{eq:training-data}
\mathcal S_{Q,N}
:=
\left\{
\left(q_j,a^\star(q_j),
(Y_{j,i,r})_{(i,r)\in[K]\times[N]}
\right):j\in[Q]
\right\}.
\end{equation}
Here $q_1,\dots,q_Q\overset{\mathrm{iid}}{\sim}\calQ$, and, conditional on the questions, the variables
\[
Y_{j,i,r}\sim p_i(\cdot\mid q_j),
\qquad
(j,i,r)\in[Q]\times[K]\times[N],
\]
are independent.  For each question and model, reduce these raw generations to the answer counts
\[
X_{j,i,a}
:=
\sum_{r=1}^N\one\{Y_{j,i,r}=a\},
\qquad
X_{j,i}:=(X_{j,i,a}:a\in\calA(q_j)).
\]
Then
\begin{equation}
\label{eq:count-model}
X_{j,i}\mid q_j\sim
\operatorname{Multinomial}\bigl(N,p_i(\cdot\mid q_j)\bigr),
\qquad (j,i)\in[Q]\times[K],
\end{equation}
and $\widehat p_i(\cdot\mid q_j)=X_{j,i}/N$.  The counts are sufficient for the answer distributions, so all subsequent procedures may be expressed in terms of $X_{j,i}$ rather than the ordered raw generations.  Define the empirical ensemble distribution and margin by
\[
\widehat p_w(a\mid q_j)
:=
\sum_{i=1}^K w_i\widehat p_i(a\mid q_j),
\qquad
\widehat M_{q_j}(w)
:=
\widehat p_w(a^\star(q_j)\mid q_j)
-\max_{a\ne a^\star(q_j)}\widehat p_w(a\mid q_j).
\]
The empirical answer distributions $\widehat p_i(\cdot\mid q_j)$ are observable and therefore make $\widehat M_{q_j}(w)$ computable for any $w$; the corresponding clean margin $M_{q_j}(w)$ is not observed.  The target $R(w)$ is defined by these clean margins and the population distribution of questions.  Thus the analysis must account both for generalization across the $Q$ questions and for estimation of each question's margin from $N$ generations.

\subsection{Estimator and guarantee}

Raw ERM minimizes the number of training questions with nonpositive empirical margin.  It can exploit weights for which a question appears solved only because of a favorable finite-$N$ fluctuation.  We instead define, for a buffer $\eps>0$,
\begin{equation}
\label{eq:buffered-erm}
\widehat R_\eps(w)
:=
\frac1Q\sum_{j=1}^Q\one\{\widehat M_{q_j}(w)\le\eps\},
\qquad
\widehat w_\eps
\in
\argmin_{w\in\Delta_K}\widehat R_\eps(w).
\end{equation}
Because $\widehat R_\eps$ takes values in the finite set $\{0,Q^{-1},\dots,1\}$, its minimum is attained.

For a confidence level $\delta\in(0,1)$, choose the buffer to be
\begin{equation}
\label{eq:epsilon-N}
\eps_N
:=
\sqrt{\frac{2\log(8KAQ/\delta)}{N}}.
\end{equation}
The logarithm comes from requiring uniform control over all training questions, models, and candidate answers.  The key is the $N^{-1/2}$ scale.

The buffer prevents favorable finite-$N$ fluctuations from creating empirical successes, but it may also count as errors questions that $w^\star$ solves with only a small winning margin.  We control the population mass of such questions through the following one-sided condition.

\begin{assumption}[One-sided margin at $w^\star$]
\label{ass:one-sided-margin}
There are constants $C_m>0$, $\beta>0$, and $t_0>0$ such that
\begin{equation}
\label{eq:one-sided-margin}
\Pp\{0<M_q(w^\star)\le t\}
\le C_m t^\beta,
\qquad 0<t\le t_0.
\end{equation}
\end{assumption}

This condition is imposed only at $w^\star$; no margin condition is required at the data-dependent estimator.  Generalization across the $Q$ sampled questions contributes a separate error scale, which we denote by
\begin{equation}
\label{eq:vq}
\mathfrak v_Q(\delta)
:=
C_0
\sqrt{
\frac{(K-1)\log(eAQ)+\log(8/\delta)}{Q}
},
\end{equation}
where $C_0$ is a universal constant.

\begin{theorem}[Buffered ERM]
\label{thm:buffered-main}
Assume \cref{ass:one-sided-margin}, let $\eps_N$ be given by \eqref{eq:epsilon-N}, and suppose $2\eps_N\le t_0$.  Let $\widehat w_{\eps_N}$ be any minimizer in \eqref{eq:buffered-erm}.  Then, with probability at least $1-\delta$ over both sources of randomness,
\begin{equation}
\label{eq:buffered-bound-exact}
R(\widehat w_{\eps_N})-R(w^\star)
\le
2\mathfrak v_Q(\delta)
+C_m(2\eps_N)^\beta.
\end{equation}
Substituting \eqref{eq:epsilon-N} and \eqref{eq:vq} gives the explicit rate stated in \eqref{eq:intro-main}.
\end{theorem}

\begin{proof}
For $t\ge0$, define the clean population and empirical threshold risks
\[
R_t(w):=\Pp\{M_q(w)\le t\},
\qquad
\widetilde R_t(w):=
\frac1Q\sum_{j=1}^Q\one\{M_{q_j}(w)\le t\}.
\]
In particular, $\widetilde R_0$ is the ordinary empirical risk that would be available if the clean answer distributions were known.  It contains the randomness from sampling $Q$ questions, but not the finite-$N$ generation noise present in $\widehat R_{\eps_N}$.  Write $\mathfrak v_Q=\mathfrak v_Q(\delta)$ for brevity.
The result follows from a chain of six inequalities.  By \cref{lem:clean-risk-deviation}, with probability at least $1-\delta/4$ over the sampled questions,
$|\widetilde R_t(w)-R_t(w)|\le\mathfrak v_Q$ simultaneously for all $w\in\Delta_K$ and $t\in\{0,2\eps_N\}$.  In particular,
\[
R_0(\widehat w_{\eps_N}) \stackrel{(a)}{\le} \widetilde R_0(\widehat w_{\eps_N})+\mathfrak v_Q, \quad \widetilde R_{2\eps_N}(w^\star) \stackrel{(e)}{\le} R_{2\eps_N}(w^\star)+\mathfrak v_Q.
\]

Separately, \cref{lem:margin-accuracy} shows, conditionally on the sampled questions, that event \eqref{eq:intro-uniform-event} holds with probability at least $1-\delta/4$ over the generations.  On this event, the buffered sandwich \eqref{eq:intro-sandwich} holds for every $w\in\Delta_K$ and $j\in[Q]$.  Summing over $j$ and dividing by $Q$ gives
\[
\widetilde R_0(\widehat w_{\eps_N}) \stackrel{(b)}{\le} \widehat R_{\eps_N}(\widehat w_{\eps_N}) \stackrel{(c)}{\le}  \widehat R_{\eps_N}(w^\star) \stackrel{(d)}{\le} \widetilde R_{2\eps_N}(w^\star).
\]
The middle inequality (c) follows from the minimality of $\widehat w_{\eps_N}$ for the buffered empirical risk.  Finally, \cref{ass:one-sided-margin} gives
\[
R_{2\eps_N}(w^\star)-R_0(w^\star)
=
\Pp\{0<M_q(w^\star)\le2\eps_N\}
\stackrel{(f)}{\le} C_m(2\eps_N)^\beta.
\]
By a union bound, the two preceding events hold simultaneously with probability at least $1-\delta$.  Combining (a)--(f) in alphabetic order and using $R=R_0$ proves \eqref{eq:buffered-bound-exact}.
\end{proof}

The two terms in \eqref{eq:intro-main} correspond to generalization across questions and estimation of their clean margins from finite generations.  The core of the proof is the chain (b)--(d), which compares the clean empirical risk of the data-dependent $\widehat w_{\eps_N}$ directly with the $2\eps_N$-thresholded clean empirical risk of $w^\star$.  The outer inequalities follow from the two sides of the buffered sandwich, while the middle inequality is the empirical optimality of $\widehat w_{\eps_N}$.  Consequently, no population margin condition is required at $\widehat w_{\eps_N}$, or uniformly over $w$.  The only margin-band cost occurs at the fixed $w^\star$, where the identity preceding (f) identifies it with the probability of $0<M_q(w^\star)\le2\eps_N$, which \cref{ass:one-sided-margin} bounds.  For the zero-buffer estimator, the analogue of inequality (b) would fail because a favorable fluctuation can make a question with clean margin in $(-\eps_N,0]$ appear solved.  A direct analysis would then have to control this negative-margin band at the data-dependent selected weight, typically through a population margin condition uniform over candidate weights or additional structure on the selection step.

\begin{remark}[Computation]
For fixed $\eps$, the buffered objective is piecewise constant on the arrangement of the empirical hyperplanes
\[
\widehat p_w(a^\star(q_j)\mid q_j)
-\widehat p_w(a\mid q_j)
=\eps,
\qquad a\ne a^\star(q_j).
\]
It can be optimized by the same mixed-integer formulation used for weighted Best-of-$\infty$, with the right-hand side shifted by $\eps$.  If a numerical solver returns $\widetilde w_\eps$ satisfying
\[
\widehat R_\eps(\widetilde w_\eps)
\le
\inf_{w\in\Delta_K}\widehat R_\eps(w)+\eta,
\]
the same proof adds only $\eta$ to \eqref{eq:buffered-bound-exact}.  Thus a computational tolerance no larger than the statistical bound in \eqref{eq:intro-main} does not change the rate.
\end{remark}

\subsection{Deployment with a finite generation budget}
\label{sec:finite-deployment}

The clean risk concerns the population voting rule.  If instead one is interested in performance at deployment under a limited generation budget, an additional source of sampling error comes into play.
With finitely many
generations on a new question, the magnitude of a positive (clean) margin also
affects the reliability of the vote.  We can account for this additional
sampling error by increasing the training buffer.

For a fixed weight $w$ and a fresh question $q\sim\calQ$, draw
$Y_1,\ldots,Y_T$ independently from $p_w(\cdot\mid q)$, independently of
the training experiment.  This is the independent-mixture implementation:
each generation first selects a model according to $w$, and then samples from it. Form the empirical answer frequencies (or counts) based on $Y_{1:T}$ and consider the plurality vote based on them. Define
\begin{equation}
\label{eq:deployment-risk}
R_T^{\mathrm{dep}}(w)
:=
\Pp_{q,Y_{1:T}}\!\left\{
a^\star(q)\text{ is not the unique empirical plurality}
\right\}.
\end{equation}
Thus $T$ is the total number of deployment generations, whereas $N$
counts training generations from each model on each training question.
As in the clean loss, ties count as errors.

\begin{corollary}[Finite deployment of buffered ERM]
\label{cor:finite-deployment}
Assume \cref{ass:one-sided-margin}, and let $\eps_N$ be given by
\eqref{eq:epsilon-N}.  Fix an integer $T\ge1$ and a reserve $r>0$ before
observing the training data, and suppose $2\eps_N+r\le t_0$.
Let $\widehat w_{\eps_N+r}$ be any minimizer in
\eqref{eq:buffered-erm} with buffer $\eps_N+r$.  Then, with probability
at least $1-\delta$ over the training experiment,
\begin{equation}
\label{eq:finite-deployment-bound}
R_T^{\mathrm{dep}}(\widehat w_{\eps_N+r})-R(w^\star)
\le
2\mathfrak v_Q(\delta)
+C_m(2\eps_N+r)^\beta
+(A-1)e^{-Tr^2/2}.
\end{equation}
\end{corollary}

The additional reserve controls questions whose positive margins are too
small for the deployment vote to be reliable.  The margin assumption is
still needed only at $w^\star$; the proof is given in
\cref{app:finite-deployment-proof}.
Taking $r_T=\sqrt{2/T}\log(eAT)$ makes the deployment tail decay faster
than any fixed power of $T^{-1}$, without requiring knowledge of $\beta$.
With $K$, $A$, the margin constants, and the confidence level fixed, the
corollary gives
\begin{equation}
\label{eq:finite-deployment-rate}
R_T^{\mathrm{dep}}(\widehat w_{\eps_N+r_T})
\le
R(w^\star)
+\widetilde O\!\left(Q^{-1/2}+N^{-\beta/2}+T^{-\beta/2}\right),
\end{equation}
whenever $2\eps_N+r_T\le t_0$.  The comparison is with optimal clean
risk.  Learning weights that optimize finite-deployment accuracy or cost
is a further problem.

\section{A scalar antisymmetric two-model core}
\label{sec:scalar-model}

Buffered ERM thresholds the empirical margin of each training question before averaging across questions.  To ask whether weaker finite-generation evidence can instead be pooled across questions, we now study a submodel in which the global statistical target becomes one-dimensional and can be analyzed sharply.

\subsection{Model and exact decision reduction}

There are two models and two answers $\{\mathrm{gold}, \mathrm{other}\}$.  Each question $q$ has a latent signed advantage $Z_q\in[-1,1]$ and
\begin{equation}
\label{eq:antisymmetric}
p_1(\mathrm{gold}\mid q)=\frac{1+Z_q}{2},
\qquad
p_2(\mathrm{gold}\mid q)=\frac{1-Z_q}{2}.
\end{equation}
Let $G$ denote the distribution of $Z_q$ when $q\sim\calQ$.  It is a probability measure supported on $[-1,1]$.
An ensemble weight can be parameterized by a single parameter as $w=(t,1-t)$ with $0\le t\le1$. In the two-answer case,
the (clean) margin of the ensemble  is the same weighted average of the models' individual margins.
Since the margin of model 1 is $Z_q$ and that of model 2 is $-Z_q$, the ensemble margin is
\begin{equation}
\label{eq:scalar-margin}
M_q(w) = t(Z_q) + (1-t)(-Z_q)
= (2t-1)Z_q.
\end{equation}

Recalling the definition of the clean risk in \eqref{eq:clean-loss-risk}, we have
\[
R(w)=\int\one\{(2t-1)z\le0\}\,\dd G(z).
\]
Here and throughout, $\sgn(z)=\one\{z>0\}-\one\{z<0\}$, so that $\sgn(0)=0$.  The following lemma formalizes that there are essentially two decision rules to consider.

\begin{lemma}[Only the side of one half matters]
\label{lem:side}
Define
\begin{equation}
\label{eq:theta}
\theta(G)
:=
\int_{-1}^1\sgn(z)\,\dd G(z)
=
G((0,1])-G([-1,0)).
\end{equation}
Every $t>1/2$ has risk $G([-1,0])$, every $t<1/2$ has risk $G([0,1])$, and $t=1/2$ has risk one.  Hence the optimal side is the sign of $\theta(G)$ when $\theta(G)\ne0$; when $\theta(G)=0$, the two sides have equal risk.  Choosing the wrong side incurs excess risk $|\theta(G)|$.
\end{lemma}

\begin{proof}
The stated risks for the three decisions $t>1/2$ (the $+$ decision), $t<1/2$ (the $-$ decision), and $t=1/2$ follow from~\eqref{eq:scalar-margin}.  The decision $t=1/2$ is weakly dominated.  Moreover,
\[
R(-)-R(+)
=G([0,1])-G([-1,0])
=G((0,1])-G([-1,0))
=\theta(G),
\]
where the mass at zero cancels.  Thus, when $\theta(G)>0$, the $+$ decision is optimal and choosing $-$ incurs excess risk $\theta(G)=|\theta(G)|$.  The case $\theta(G)<0$ is similar, while the two decisions have equal risk when $\theta(G)=0$.

\end{proof}

Thus, after excluding the weakly dominated choice $t=1/2$, the ensemble problem is to choose a side $\widehat\sigma\in\{-1,+1\}$.  The $+$ side is optimal when $\theta(G)>0$, the $-$ side is optimal when $\theta(G)<0$, and either side is optimal when $\theta(G)=0$.

The loss producing the regret (i.e. excess-risk) of a decision rule $\widehat\sigma$ is
\begin{equation}
\label{eq:scalar-decision-loss}
\mathcal L_G(\widehat\sigma)
:=
|\theta(G)|\,
\one\{\widehat\sigma\,\theta(G)<0\}.
\end{equation}
In particular, the loss is zero when $\theta(G)=0$, because the two sides then have equal risk.

Although the original ensemble problem only requires choosing a side, it will be useful to consider the companion problem of estimating $\theta(G)$ under absolute-error loss.  Given any real-valued estimator $\widehat\theta$, let its induced decision be $\widehat\sigma(\widehat\theta)=+1$ when $\widehat\theta\ge0$ and $-1$ otherwise.
Then, for every $G$,
\begin{equation}
\label{eq:pointwise-decision-reduction}
\mathcal L_G\bigl(\widehat\sigma(\widehat\theta)\bigr)
\le
|\widehat\theta-\theta(G)|.
\end{equation}
Indeed, the left-hand side is zero when $\theta(G)=0$; otherwise, whenever the induced decision chooses the wrong side, $\widehat\theta$ lies on the opposite side of zero from $\theta(G)$.  Thus any absolute-error bound for $\widehat\theta$ is also an upper bound on the regret of the induced decision rule.

\subsection{What the generations reveal}

We now formalize the reduced binomial-mixture experiment previewed in \eqref{eq:intro-binomial-channel}.  Let $X_{1q}$ and $X_{2q}$ be the numbers of gold answers among the $N$ generations of models 1 and 2 on question $q$.  Conditional on $Z_q=z$,
\[
X_{1q}\sim\Bin\left(N,\frac{1+z}{2}\right),
\qquad
X_{2q}\sim\Bin\left(N,\frac{1-z}{2}\right),
\quad\text{independently}.
\]
Define
\begin{equation}
\label{eq:S-stat}
S_q:=X_{1q}+N-X_{2q}.
\end{equation}
The joint likelihood of $(X_{1q},X_{2q})$ depends on $z$ only through $S_q$, so $S_q$ is sufficient for $z$ in this antisymmetric experiment.  For the sampled questions, write $Z_j:=Z_{q_j}$ and $S_j:=S_{q_j}$, and set $n:=2N$.  The complete hierarchical model is
\begin{equation}
\label{eq:scalar-reduced-experiment}
Z_1,\dots,Z_Q\overset{\mathrm{iid}}{\sim}G,
\qquad
S_j\mid Z_j=z
\sim
\Bin\left(n,\frac{1+z}{2}\right),
\quad j=1,\dots,Q,
\end{equation}
By sufficiency, the raw pair of model counts can be reduced to $S_j$ conditional on the observed question $q_j$, so the full antisymmetric data reduce to $((q_j,S_j))_{j=1}^Q$.  Passing from these pairs to $S_1,\dots,S_Q$ alone is an additional observation coarsening.  We analyze the resulting count-only experiment first and prove in \cref{sec:observed-transfer} that its minimax lower bounds transfer, up to universal constants, to a fully nonparametric observed-question model.

For $k=0,\dots,n$, let $b_{k,n}(p):=\binom nkp^k(1-p)^{n-k}$ denote the $k$th Bernstein basis polynomial of degree $n$.  Let $K_nG$ denote the marginal law of $S_j$ after mixing over $Z_j\sim G$.  As usual on a finite sample space, we identify this distribution with its pmf, so that
\begin{equation}
\label{eq:KnG}
(K_nG)(k)
=
\int_{-1}^1
b_{k,n}\left(\frac{1+z}{2}\right)
\dd G(z),
\qquad k=0,\dots,n.
\end{equation}
This is the nonparametric $n$-trial binomial-mixture distribution with mixing law $G$.  Equivalently, the marginal observation model is
\begin{equation}
\label{eq:scalar-marginal-experiment}
S_1,\dots,S_Q\overset{\mathrm{iid}}{\sim}K_nG.
\end{equation}
From these observations, the goal is either to estimate $\theta(G)$ in absolute error or decide its sign.

More precisely, $K_n$ is the Markov kernel from $[-1,1]$ to $\{0,\dots,n\}$ that assigns to each $z$ the distribution $\Bin(n,(1+z)/2)$; we refer to it as the $n$-trial binomial channel.  We use the same notation for the induced map on probability measures, so that $K_nG$ is the output distribution when the input $Z$ has distribution $G$, as displayed in~\eqref{eq:KnG}.  Thus~\eqref{eq:scalar-marginal-experiment} can be viewed as observing the infinite-dimensional mixing distribution $G$ through the binomial channel.

To make the channel and its inferential targets concrete, consider the smooth family $G_\eta$ with density
\begin{equation}
\label{eq:scalar-example-density}
g_\eta(z)
:=
\frac{105}{16}z^2(1-z^2)^2(1+\eta z),
\qquad z\in[-1,1],\quad |\eta|\le1.
\end{equation}
The odd tilt leaves the normalizing constant unchanged, and a direct calculation gives
$G_\eta([-t,t])\le(35/8)t^3$ and $\theta(G_\eta)=(35/64)\eta$.
\Cref{fig:scalar-channel} shows how the channel transforms $G_\eta$ into the less informative observation law $K_nG_\eta$, while the two inferential goals depend only on the imbalance between its colored lobes.

\begin{figure}[t]
\centering
\includegraphics[width=\linewidth]{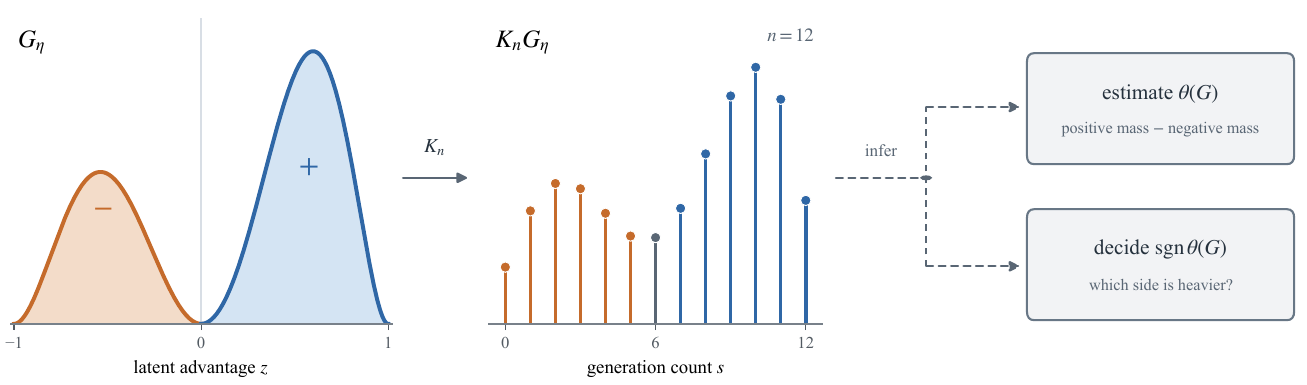}
\caption{The scalar binomial channel and its two inferential targets, shown for the family in \eqref{eq:scalar-example-density} at $\eta=1/2$ and with $n=12$.  Orange and blue distinguish the negative and positive latent mass of $G_\eta$, whose signed difference is $\theta(G_\eta)$.  The channel maps the continuous mixing law to the discrete count law $K_nG_\eta$.  The two statistical tasks are to estimate this imbalance or decide which side is larger; neither requires reconstructing $G_\eta$ itself.}
\label{fig:scalar-channel}
\end{figure}

The following lemma describes an inherent limitation of this channel: it reveals only the first $n$ moments of $G$.

\begin{lemma}\label{lem:moment-equivalence}
For probability measures $G,H$ on $[-1,1]$, the following are equivalent:
\begin{enumerate}[label=(\alph*)]
\item $K_nG=K_nH$;
\item $\int P\,\dd G=\int P\,\dd H$ for every algebraic polynomial $P$ of degree at most $n$;
\item $\int z^j\,\dd G(z)=\int z^j\,\dd H(z)$ for $j=0,\dots,n$.
\end{enumerate}
\end{lemma}

\begin{proof}
Equality of $K_nG$ and $K_nH$ is equivalent to equality of the integrals of $b_{k,n}((1+z)/2)$ for $k=0,\dots,n$.  The Bernstein polynomials form a basis for the polynomials of degree at most $n$ in $p$, and composition with the invertible affine map $p=(1+z)/2$ gives a basis for the polynomials of degree at most $n$ in $z$. Monomials $z^j, j=0,\dots,n$ form another basis.
\end{proof}

In the count-only experiment, letting $Q\to\infty$ identifies only $K_nG$, equivalently the first $n$ moments of $G$.  Consequently, the count-only observations need not identify $G$, or even $\theta(G)$.  This is the source of the finite-generation identification floor quantified below.  The observed-question experiment requires a separate minimax transfer argument, developed in \cref{sec:observed-transfer}.

\section{Scalar minimax theory}
\label{sec:scalar-minimax}

The preceding section reduced the antisymmetric two-model problem to observing independent draws $S_1,\dots,S_Q$
from the binomial-mixture
$K_nG$, with the goal of determining the sign of $\theta(G)$.  We now study the minimax difficulty of this decision problem under a two-sided margin condition on $G$.

We begin by showing that the Hellinger modulus of the binomial channel characterizes the minimax decision risk up to universal constants.  We then evaluate its zero-radius limit to obtain the finite-generation identification floor, derive finite-$Q$ obstructions, and show that these scalar lower bounds transfer when the questions themselves are observed and to the general ensemble problem.  The next sections develop explicit estimators and use the Durrmeyer spectrum to prove matching rates not only at the two endpoints but also in a broad intermediate regime for $0<\beta\le1$; the precise rate immediately above the first transition remains open.

\subsection{Margin class and minimax risks}

Fix $\beta>0$ and $L\ge1$.  Define the two-sided margin class
\begin{equation}
\label{eq:Gclass}
\calG_{\beta,L}
:=
\left\{
G\in\mathcal P([-1,1]):
G([-t,t])\le Lt^\beta
\text{ for every }0<t\le1
\right\}.
\end{equation}
The condition implies $G(\{0\})=0$.  It is stronger than the one-sided condition at $w^\star$ used in \cref{thm:buffered-main}; the symmetry is convenient for the minimax identification problem.
The illustrative family in \cref{fig:scalar-channel} belongs to $\calG_{3,35/8}$ for every $|\eta|\le1$.

For decision rules $\widehat\sigma\in\{-1,+1\}$ based on $S_1,\dots,S_Q$, define the minimax decision risk
\begin{equation}
\label{eq:minimax-decision}
\calR^{\mathrm{dec}}_{Q,n}(\beta,L)
:=
\inf_{\widehat\sigma}
\sup_{G\in\calG_{\beta,L}}
\E_G\mathcal L_G(\widehat\sigma).
\end{equation}
For estimators $\widehat\theta$ based on the same observations, define the minimax absolute-error risk
\[
\calR^{\mathrm{abs}}_{Q,n}(\beta,L)
:=
\inf_{\widehat\theta}
\sup_{G\in\calG_{\beta,L}}
\E_G|\widehat\theta-\theta(G)|.
\]
The pointwise reduction \eqref{eq:pointwise-decision-reduction} shows that
\begin{equation}
\label{eq:dec-le-abs}
\calR^{\mathrm{dec}}_{Q,n}(\beta,L)
\le
\calR^{\mathrm{abs}}_{Q,n}(\beta,L).
\end{equation}

\subsection{Exact characterization by the Hellinger modulus}

For two probability vectors $p,q$ on $\{0,\dots,n\}$, let
$h^2(p,q)=\sum_{k=0}^n(\sqrt{p_k}-\sqrt{q_k})^2$
denote their squared Hellinger distance.
Define the indirect Hellinger modulus
\begin{equation}
\label{eq:modulus}
\omega_{n,\beta,L}(\eps)
:=
\sup\bigl\{
|\theta(G)-\theta(H)|:\;
G,H\in\calG_{\beta,L},
\ h(K_nG,K_nH)\le\eps
\bigr\}.
\end{equation}
This is the largest ambiguity in the target among latent distributions whose observable count laws are $\eps$-indistinguishable.

The following theorem characterizes both scalar minimax risks, up to
universal constants, through this modulus.

\begin{theorem}[Scalar minimax characterization]
\label{thm:scalar-modulus}
There are universal constants $0<c<C<\infty$ such that, for every $Q,n\ge1$,
\begin{equation}
\label{eq:scalar-modulus-chain}
c\,\omega_{n,\beta,L}(Q^{-1/2})
\le
\calR^{\mathrm{dec}}_{Q,n}(\beta,L)
\le
\calR^{\mathrm{abs}}_{Q,n}(\beta,L)
\le
C\,\omega_{n,\beta,L}(Q^{-1/2}).
\end{equation}
\end{theorem}

The upper bound follows by applying the abstract result in
\cref{thm:modulus-upper} to the
image of $\calG_{\beta,L}$ under the affine map
$G\mapsto(K_nG,\theta(G))$, which is a compact convex subset of
$\Delta_{n+1}\times[-1,1]$.  The proof of that
abstract result follows the modulus-via-testing approach of
\citet{DonohoLiuII}: robust tests separate target level sets, and a
multiscale interval search turns a sequence of such tests into an
estimator.  The result does not require the observation map
$G\mapsto K_nG$ to be injective, so any residual nonidentifiability is
retained through $\omega_{n,\beta,L}(0)$.  The lower decision bound
requires an additional step: an arbitrary pair in the modulus may have
target values with the same sign, and hence the same optimal decision.
Reflection symmetrization produces an equally indistinguishable pair with
opposite optimal decisions, after which a two-point testing argument
applies.  Both arguments initially evaluate the modulus at fixed numerical
multiples of $Q^{-1/2}$; the scaling property in
\cref{lem:modulus-scaling} absorbs these multiples into the constants in
\eqref{eq:scalar-modulus-chain}.  Full details are given in
\cref{app:scalar-modulus}.

Theorem~\ref{thm:scalar-modulus} characterizes the minimax decision
regret $\calR^{\mathrm{dec}}_{Q,n}(\beta,L)$ in the scalar model,
uniformly over $Q,n\ge1$, up to universal constant factors.  Determining
the rate as the two sample sizes vary therefore reduces to evaluating the
modulus at resolution $Q^{-1/2}$.  Because the modulus is itself a
constrained supremum over pairs of mixing distributions, its scaling with
$Q$ and $n$ is not immediately clear.  We next derive explicit bounds and identify
regimes in which they match.

\subsection{The \texorpdfstring{$Q=\infty$}{Q=infinity} identification floor}

We first consider the implication of \cref{thm:scalar-modulus} in the limit $Q\to\infty$.  As the statistical resolution $Q^{-1/2}$ vanishes, the relevant quantity becomes the zero-radius modulus $\omega_{n,\beta,L}(0)$.  By \cref{lem:moment-equivalence}, equality of the observable laws $K_nG=K_nH$ is equivalent to equality of the moments through degree $n$.  Thus $\omega_{n,\beta,L}(0)$ measures the remaining ambiguity in $\theta(G)$ when those moments are known exactly.  The next theorem evaluates this identification width sharply.

\begin{theorem}[Sharp identification floor]
\label{thm:identification-floor}
For every fixed $\beta>0$ and $L\ge1$, there are constants $0<c_{\beta,L}<C_{\beta,L}<\infty$ such that
\begin{equation}
\label{eq:identification-floor}
 c_{\beta,L}n^{-\beta}
\le
\omega_{n,\beta,L}(0)
\le
C_{\beta,L}n^{-\beta}
\qquad\text{for every }n\ge1.
\end{equation}
Consequently,
\begin{equation}
\label{eq:N-only-lower}
\calR^{\mathrm{dec}}_{Q,n}(\beta,L)
\ge c_{\beta,L}n^{-\beta}
\qquad\text{for every }Q.
\end{equation}
Since $n=2N$, the finite-generation identification floor is $\Theta(N^{-\beta})$.
\end{theorem}

\begin{figure}[t]
\centering
\begin{subfigure}[t]{0.49\linewidth}
\centering
\includegraphics[width=\linewidth]{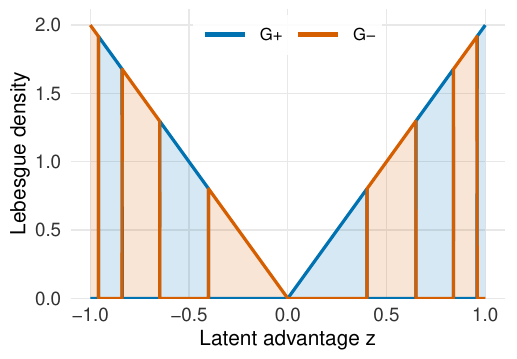}
\subcaption{A nonidentifiable latent pair.}
\end{subfigure}\hfill
\begin{subfigure}[t]{0.49\linewidth}
\centering
\includegraphics[width=\linewidth]{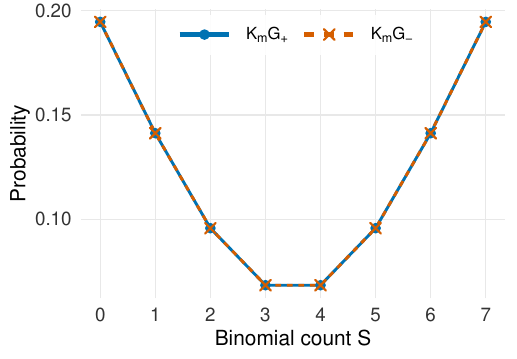}
\subcaption{Their identical count laws.}
\end{subfigure}
\caption{A numerical realization of the dual pair in \eqref{eq:dual-pair-main} for $\beta=2$ and $m=n=7$, corresponding to margin constant $L=2$.  Since the dual witness is bang--bang, the Lebesgue densities of $G_{7,+}$ and $G_{7,-}$ occupy complementary shells beneath the common envelope $2|z|$.  The measures agree on every moment through degree seven, so their seven-trial binomial-mixture laws coincide even though their sign-imbalance functionals differ.}
\label{fig:dual-moment-pair}
\end{figure}

The minimax consequence~\eqref{eq:N-only-lower} follows directly from \cref{thm:scalar-modulus} and monotonicity of the modulus: $\calR^{\mathrm{dec}}_{Q,n}(\beta,L)\gtrsim\omega_{n,\beta,L}(Q^{-1/2})\ge\omega_{n,\beta,L}(0)$.  It therefore remains to establish the two bounds in \eqref{eq:identification-floor}.

For the upper bound, a localized Jackson-kernel construction gives a degree-$n$ polynomial $P_n$, which we call the \emph{localized Jackson polynomial}, such that
\[
|P_n(z)-\sgn(z)|
\lesssim_s
\min\{1,(n|z|)^{-s}\},
\qquad \norm{P_n}_\infty\le1,
\]
where $s>\beta$.  Integrating this pointwise bound under the margin condition gives $\int|\sgn-P_n|\,\dd G\lesssim_{\beta,L}n^{-\beta}$ uniformly over $G\in\calG_{\beta,L}$.  If $K_nG=K_nH$, then $G$ and $H$ have identical moments through degree $n$, and hence
\[
\int P_n\,\dd G=\int P_n\,\dd H,
\qquad
|\theta(G)-\theta(H)|
\le
\int|\sgn-P_n|\,\dd G
+
\int|\sgn-P_n|\,\dd H
\lesssim_{\beta,L}n^{-\beta}.
\]
Taking the supremum over all such pairs proves the upper bound on the identification width.

The lower bound uses a related but different approximation statement.  The argument above requires a single polynomial whose error integrates uniformly over every $G\in\calG_{\beta,L}$.  For the lower bound, it suffices to work under the canonical measure
\[
\mu_\beta(\dd z):=(\beta/2)|z|^{\beta-1}\one_{[-1,1]}(z)\,\dd z
\]
and determine the best weighted approximation error.  Writing $\Pi_n$ for the real algebraic polynomials of degree at most $n$, \cref{lem:weighted-sign} shows that
\begin{equation}
\label{eq:weighted-L1}
E_n(\beta)
:=
\inf_{P\in\Pi_n}
\int|\sgn(z)-P(z)|\,\dd\mu_\beta(z)
\asymp_\beta n^{-\beta}.
\end{equation}
The upper inequality here follows from the localized Jackson construction; the reverse inequality uses the weighted inverse theorem of \citet[Theorem~9.1]{Kopotun}.

The $L_1$ analogue of linear-programming duality, applied to the optimization problem defining $E_n(\beta)$, provides an odd function $h_n$ with $\norm{h_n}_\infty\le1$ that annihilates $\Pi_n$ and attains the approximation error $E_n(\beta)$.  The measures
\begin{equation}
\label{eq:dual-pair-main}
\dd G_{n,+}=(1+h_n)\dd\mu_\beta,
\qquad
\dd G_{n,-}=(1-h_n)\dd\mu_\beta
\end{equation}
form a reflected pair of probability measures with identical moments through degree $n$, hence inducing the same binomial-mixture law, while their targets are separated by $2E_n(\beta)\asymp n^{-\beta}$.  \Cref{fig:dual-moment-pair} illustrates this exact nonidentifiability mechanism for one numerically constructed pair.  Mixing both measures with the same symmetric distribution on $\{-1,+1\}$ places them in $\calG_{\beta,L}$ while preserving moment matching and the order of the target separation.  Full proofs of the localized Jackson construction, the weighted inverse argument, and the dual moment-matching construction are given in \cref{app:identification}.

The dual construction in \cref{lem:duality-moment-pair} is formulated for an arbitrary cancellation degree $m$.  The choice $m=n$ made here cancels every moment visible through the $n$-trial binomial channel and hence gives exact equality of the count laws.  The same family reappears in \cref{thm:spectral-modulus-lower} with $m=d<n$: only the first $d$ moments cancel, and the Durrmeyer spectrum controls the information that leaks through the remaining modes.

\subsection{Finite-\texorpdfstring{$Q$}{Q} lower bounds}

The modulus characterization also isolates two distinct finite-sample obstructions.  The first is ordinary population uncertainty: even when every latent advantage is separated from zero, $Q$ questions reveal their population imbalance only at the root-$Q$ scale.

\begin{proposition}[Hard-margin root-$Q$ lower bound]
\label{thm:rootQ-scalar}
For every $n\ge1$, $\beta>0$, $L\ge1$, and $0<\eps\le1$,
\begin{equation}
\label{eq:rootQ-modulus}
\omega_{n,\beta,L}(\eps)\ge\eps.
\end{equation}
Consequently, for every $Q,n\ge1$,
\begin{equation}
\label{eq:rootQ-scalar}
\calR^{\mathrm{dec}}_{Q,n}(\beta,L)
\ge cQ^{-1/2}
\end{equation}
for a universal constant $c>0$.  Both claims hold on a two-atom subclass supported on $Z\in\{-1,+1\}$ and therefore under a hard margin.
\end{proposition}

Indeed, under
\[
G_{\sigma,\eta}
=
\frac{1+\sigma\eta}{2}\delta_{+1}
+
\frac{1-\sigma\eta}{2}\delta_{-1},
\qquad \sigma\in\{-1,+1\},
\]
every observed count is deterministically $0$ or $n$.  The two output laws differ only through the population imbalance, and their Hellinger distance is of order $\eta$, while their targets differ by $2\eta$.  Thus finite-generation noise plays no role.  The exact modulus calculation is given in \cref{app:finite-lower}.

For comparison, an elementary construction based on a rare population of near-boundary questions with a common orientation gives the uniform lower bound $\calR^{\mathrm{dec}}_{Q,n}(\beta,L)\gtrsim_{\beta,L}(Qn)^{-\beta/(\beta+2)}$.  Although it illustrates why pooling weak count evidence can help, it is dominated by the spectral lower bound throughout the asymptotic region in which it improves on both endpoint bounds.  We give the construction and comparison in \cref{app:rare-boundary}.

\subsection{Observed questions and minimax transfer}
\label{sec:observed-transfer}

The count-only experiment in \eqref{eq:scalar-marginal-experiment} deliberately discards the question identities.  This is harmless for upper bounds, because a procedure in the original experiment can always ignore the questions.  It is not harmless for lower bounds: an estimator that observes $q_j$ may try to infer the latent advantage $Z(q_j)$ from the question itself.  We now show that, in a fully nonparametric observed-question model, this extra information does not improve the minimax rate.  The proof embeds an arbitrary scalar hard pair into deterministic question-to-advantage maps and randomizes only the assignment of their latent advantages to otherwise uninformative question cells.

We model the question population as a nonatomic standard probability
space; see \cref{rem:nonatomic} for the role of this assumption and a
countable alternative.  Every such space is measure-isomorphic modulo null
sets to $([0,1],\mathcal B,\lambda)$, which we therefore use as a canonical
representation.  No regularity or structural relation between $q$ and
$Z(q)$ is imposed.  Let
\[
\mathfrak Z_{\beta,L}
:=
\{Z:[0,1]\to[-1,1]\text{ measurable}:Z_\#\lambda\in\calG_{\beta,L}\},
\]
where $Z_\#\lambda$ is the pushforward law of $Z(q)$ for $q\sim\lambda$.  For $Z\in\mathfrak Z_{\beta,L}$, the observed-question antisymmetric experiment is
\begin{equation}
\label{eq:full-observed-scalar}
q_1,\dots,q_Q\overset{\mathrm{iid}}{\sim}\lambda,
\qquad
S_j\mid q_j
\sim
\Bin\!\left(n,\frac{1+Z(q_j)}2\right),
\quad j=1,\dots,Q.
\end{equation}
The estimator sees the full pairs $(q_j,S_j)$.  Define
\[
\theta_Z
:=
\int_0^1\sgn(Z(q))\,\dd q
=
\theta(Z_\#\lambda)
\]
and
\begin{equation}
\label{eq:minimax-observed}
\calR^{\mathrm{obs}}_{Q,n}(\beta,L)
:=
\inf_{\widehat\sigma}
\sup_{Z\in\mathfrak Z_{\beta,L}}
\E_Z\!\left[
|\theta_Z|\,
\one\{\widehat\sigma\theta_Z<0\}
\right],
\end{equation}
where the infimum ranges over all, possibly randomized, rules based on $((q_j,S_j))_{j=1}^Q$.  The raw pair of model counts may be used instead of $S_j$ without changing the experiment, because $S_j$ is sufficient conditional on $q_j$ in the antisymmetric family.

Every count-only rule is admissible in \eqref{eq:minimax-observed}.  Moreover, every $G\in\calG_{\beta,L}$ can be realized as $Z_\#\lambda$ by a quantile map.  Hence
\begin{equation}
\label{eq:obs-le-count}
\calR^{\mathrm{obs}}_{Q,n}(\beta,L)
\le
\calR^{\mathrm{dec}}_{Q,n}(\beta,L).
\end{equation}
The reverse comparison is the nontrivial direction.

The reverse comparison uses two auxiliary constructions.  First, an
arbitrary scalar mixing distribution is approximated by an equal-weight
discrete measure using outward quantiles.  Each mass block is moved to its
endpoint farther from zero, preserving the margin bound, while the target
and induced count law converge to their original values.  Equal weights
then allow the $M$ atoms to be assigned exactly to $M$ equal-mass question
cells.  We average over uniform random permutations of this assignment.
After conditioning away collisions among both the sampled question cells
and the assigned atom indices, the resulting augmented experiment agrees
with its count-only counterpart; the total discrepancy is $O(Q^2/M)$.
Applying this construction to a reflected pair from the Hellinger modulus
and letting $M$ grow gives the lower bound.  Details are provided in
\cref{app:observed-transfer}.

\begin{theorem}[Transfer from the scalar count experiment to observed questions]
\label{thm:observed-transfer}
There are universal constants $c_0,c_1>0$ such that, for every $Q,n\ge1$, $\beta>0$, and $L\ge1$,
\begin{equation}
\label{eq:observed-transfer}
 c_1\,\omega_{n,\beta,L}(c_0Q^{-1/2})
\le
\calR^{\mathrm{obs}}_{Q,n}(\beta,L)
\le
\calR^{\mathrm{dec}}_{Q,n}(\beta,L).
\end{equation}
One may take $c_0=1/64$ and $c_1=1/16$.
\end{theorem}

Combining Theorems~\ref{thm:observed-transfer} and~\ref{thm:scalar-modulus} with Lemma~\ref{lem:modulus-scaling} yields the desired equivalence.

\begin{corollary}[Count-only and observed-question minimax risks are equivalent]
\label{cor:observed-modulus}
For fixed $\beta,L$ and every $Q,n\ge1$,
\begin{equation}
\label{eq:observed-modulus-equivalence}
\calR^{\mathrm{obs}}_{Q,n}(\beta,L)
\asymp
\calR^{\mathrm{dec}}_{Q,n}(\beta,L)
\asymp
\omega_{n,\beta,L}(Q^{-1/2}),
\end{equation}
where the implicit constants are universal.
\end{corollary}

\begin{corollary}[Transfer to general ensembles]
\label{cor:general-lower-transfer}
Fix $K\ge2$ and $A\ge2$.  For each $Z\in\mathfrak Z_{\beta,L}$,
consider the binary-answer $K$-model environment on $([0,1],\lambda)$ in
which models 1 and 2 obey \eqref{eq:antisymmetric}, while every model
$i\ge3$ returns either answer with probability $1/2$.  Then
\begin{equation}
\label{eq:general-lower-transfer}
\inf_{\widehat w}
\sup_{Z\in\mathfrak Z_{\beta,L}}
\E_Z\!
\left[
R_Z(\widehat w)-\inf_{w\in\Delta_K}R_Z(w)
\right]
\ge
\calR^{\mathrm{obs}}_{Q,2N}(\beta,L)
\gtrsim
\omega_{2N,\beta,L}(Q^{-1/2}).
\end{equation}
\end{corollary}

Consequently, the root-$Q$, identification-floor, and spectral lower bounds
obtained in the scalar model hold for any larger ensemble class containing
this subclass.  The proof verifies the zero-margin case, the sufficiency
reduction, and the one-sided margin condition; see
\cref{app:observed-transfer}.

\begin{remark}[Fresh questions and finite strings]
\label{rem:nonatomic}
The nonatomic formulation keeps the two sample sizes conceptually separate: $Q$ counts fresh questions, while $N$ is the generation budget attached to each question.
If questions are instead represented literally as finite strings, let $\mathcal X=\Sigma^\star$ with the discrete sigma-field $2^{\Sigma^\star}$.  The same constant-factor transfer rate holds for the distribution-free class
\[
\mathfrak P^{\mathrm{obs}}_{\beta,L}
:=
\left\{
(\calQ,Z):
\calQ\in\mathcal P(\mathcal X),\;
Z:\mathcal X\to[-1,1],\;
Z_\#\calQ\in\calG_{\beta,L}
\right\}.
\]
Indeed, we can use the same construction, taking $\calQ$ to be uniform
on $M$ distinct strings, randomly assigning the atoms of each equal-weight
approximation to those strings, and then choosing $M\gg Q^2$.  The proof otherwise repeats the
discretization, permutation coupling, and collision argument above and is omitted.
\end{remark}

\section{Scalar estimators and spectral upper bounds}
\label{sec:scalar-estimators}

We now turn from the minimax characterization and lower bounds to explicit procedures.  The two estimators developed below use the per-question generation counts in different ways.  The plug-in estimator first converts each count into an estimated sign and then averages these decisions across questions.  The affine moment estimator instead retains the count values and pools their information through a polynomial approximation to $\sgn$.

The plug-in estimator matches the root-$Q$ lower bound in the question-limited regime.  For the affine moment estimator, we derive a spectral upper bound that reaches the $n^{-\beta}$ identification floor when $Q$ is sufficiently large and also yields sharper guarantees at intermediate degrees.  The matching spectral lower bound and the resulting middle regime are developed in \cref{sec:spectral-middle}.

\subsection{Per-question plug-in estimator}

For each question, define the majority-side score and its average by
\[
\widehat\xi_j
:=
\one\{S_j>n/2\}-\one\{S_j<n/2\},
\qquad
\widehat\theta_{\mathrm{plug}}
:=
\frac1Q\sum_{j=1}^Q\widehat\xi_j.
\]
Thus $\widehat\xi_j$ estimates $\sgn(Z_j)$ and equals zero at an exact tie.  Set $\widehat\sigma_{\mathrm{plug}}=+1$ when $\widehat\theta_{\mathrm{plug}}\ge0$ and $\widehat\sigma_{\mathrm{plug}}=-1$ otherwise.

\begin{theorem}[Plug-in estimator]
\label{thm:plugin-scalar}
For a constant $C=C(\beta,L)$,
\begin{equation}
\label{eq:plugin-scalar}
\sup_{G\in\calG_{\beta,L}}
\E_G|\widehat\theta_{\mathrm{plug}}-\theta(G)|
\le
C\left(Q^{-1/2}+n^{-\beta/2}\right).
\end{equation}
The same expression upper-bounds the decision regret.
\end{theorem}
The proof, a direct conditional Hoeffding bound integrated under the margin condition, is deferred to \cref{app:plugin-proof}.

This estimator makes a hard sign decision separately for each question.  It is optimal in the question-limited regime but pays the larger finite-generation scale $n^{-\beta/2}$.

\begin{corollary}[Question-limited regime]
\label{cor:question-limited}
If $Q\le c_0 n^\beta$ for a fixed constant $c_0$, then
\begin{equation}
\label{eq:question-limited}
\calR^{\mathrm{dec}}_{Q,n}(\beta,L)
\asymp_{\beta,L,c_0}
Q^{-1/2}.
\end{equation}
\end{corollary}
Indeed, the plug-in upper bound is $O(Q^{-1/2})$ in this regime, while \cref{thm:rootQ-scalar} gives the matching lower bound.

At the opposite endpoint, the pooled polynomial estimator developed next reaches the $n^{-\beta}$ identification floor once $Q$ is sufficiently large.

\subsection{The affine moment-estimation idea}

A different estimator avoids per-question thresholding.  Recall that $b_{k,n}(p):=\binom nkp^k(1-p)^{n-k}$.  Given $a=(a_0,\dots,a_n)\in\R^{n+1}$, define the statistic and its associated polynomial together by
\begin{equation}
\label{eq:affine-estimator}
\widehat\theta_a
:=
\frac1Q\sum_{j=1}^Q a_{S_j},
\qquad
P_a(z)
:=
\sum_{k=0}^n a_k
 b_{k,n}\!\left(\frac{1+z}{2}\right).
\end{equation}
Since $\Pp(S_j=k\mid Z_j=z)=b_{k,n}((1+z)/2)$,
\begin{equation}
\label{eq:bernstein-conditional-identity}
\E[a_{S_j}\mid Z_j=z]=P_a(z),
\qquad
\E_G\widehat\theta_a=\int P_a(z)\,\dd G(z).
\end{equation}
If the polynomial $P_a$ approximates $\sgn$, $\widehat\theta$ has a small bias; if the coefficients $a_k$ are stable, in the sense of having a small spread, it has small variance.
More precisely, define the coefficient spread
\[
A(a):=\max_k a_k-\min_k a_k.
\]
The variable $a_{S_j}$ is supported on an interval of length $A(a)$, so a small spread controls its fluctuation uniformly over $G$. Designing a good estimator $\widehat \theta_a$ is thus the problem of finding a \emph{stable polynomial approximation} to a discontinuous function.

For analysis it is more convenient to reverse this parametrization and design an estimator by choosing its polynomial.

\begin{lemma}[Bernstein parametrization]
\label{lem:bernstein-identity}
The map $a\mapsto P_a$ is a linear bijection from $\R^{n+1}$ onto $\Pi_n$.  Equivalently, every $P\in\Pi_n$ has unique coefficients $a_0(P),\dots,a_n(P)$ such that
\begin{equation}
\label{eq:bernstein-representation}
P(z)
=
\sum_{k=0}^n a_k(P)\,
 b_{k,n}\!\left(\frac{1+z}{2}\right).
\end{equation}
\end{lemma}
\begin{proof}
The functions $b_{k,n}((1+z)/2)$, $k=0,\dots,n$, form a basis of $\Pi_n$: the Bernstein polynomials form a basis in $p$, and $p=(1+z)/2$ is an invertible affine change of variable.
\end{proof}

For $P\in\Pi_n$, define
\begin{equation}
\label{eq:polynomial-estimator}
\widehat\theta_P
:=
\widehat\theta_{a(P)}
=
\frac1Q\sum_{j=1}^Q a_{S_j}(P),
\qquad
A(P):=A(a(P)).
\end{equation}
Thus $\widehat\theta_a$ and $\widehat\theta_P$ are the same estimator family under two parametrizations: the coefficient parametrization makes the statistic transparent, while the polynomial parametrization is better suited to estimator design.  Their performance
is captured by the tradeoff between  approximation bias and coefficient stability.

\begin{proposition}[Bias--stability bound]
\label{prop:bias-stability}
For every $P\in\Pi_n$ and $G\in\calG_{\beta,L}$,
\begin{equation}
\label{eq:bias-stability}
\E_G|\widehat\theta_P-\theta(G)|
\le
\int|P(z)-\sgn(z)|\,\dd G(z)
+
\frac{A(P)}{2\sqrt Q}.
\end{equation}
\end{proposition}

\begin{proof}
The conditional identity \eqref{eq:bernstein-conditional-identity} gives
$
\E_G\widehat\theta_P
=
\int P(z)\,\dd G(z),
$
and hence
\[
|\E_G\widehat\theta_P-\theta(G)|
\le
\int|P(z)-\sgn(z)|\,\dd G(z).
\]
Moreover, $a_{S_j}(P)$ is supported on an interval of length $A(P)$, so $\Var_G(a_{S_j}(P))\le A(P)^2/4$.  Independence across questions and Jensen's inequality give
\[
\E_G|\widehat\theta_P-\E_G\widehat\theta_P|
\le
\sqrt{\Var_G(\widehat\theta_P)}
\le
\frac{A(P)}{2\sqrt Q}.
\]
Combining the bias and fluctuation bounds proves \eqref{eq:bias-stability}.
\end{proof}

By \eqref{eq:pointwise-decision-reduction}, the same bound applies to the regret of the decision induced by $\widehat\theta_P$.  Estimator design has therefore become a bias--stability tradeoff: choose $P\in\Pi_n$ that approximates $\sgn$ over the margin class while keeping the spread of its degree-$n$ Bernstein coefficients small.

A self-contained Chebyshev-to-Bernstein calculation in \cref{app:coefficients} gives the following conservative coefficient bound.

\begin{lemma}[Explicit Bernstein coefficient bound]
\label{lem:B0}
There are universal constants $C_0<\infty$ and $B_0<8$ such that, if $P$ has degree at most $d$ and $\norm{P}_{L_\infty[-1,1]}\le1$, then every degree-$n$ Bernstein coefficient of $P(2p-1)$, for every $n\ge d$, satisfies
\[
|a_k(P)|\le C_0B_0^d.
\]
One may take $B_0=7.5$.
\end{lemma}

Combining this lemma with the localized Jackson polynomial gives
\begin{equation}
\label{eq:B0-upper}
\calR^{\mathrm{dec}}_{Q,n}(\beta,L)
\lesssim
\inf_{1\le d\le n}
\left\{
 d^{-\beta}+\frac{B_0^d}{\sqrt Q}
\right\}.
\end{equation}
This already proves that the $n^{-\beta}$ identification floor is constructively attainable once $Q$ is exponential in $n$.

\subsection{The Bernstein--Durrmeyer spectrum and estimator stability}
\label{sec:durrmeyer}

The preceding coefficient bound $A(P)\lesssim B_0^d$ is uniform over all trial budgets $n\ge d$; thus $n$ enters only by restricting the admissible degree to $d\le n$.  A sharper bound on $A(P)$, depending jointly on $(d,n)$, follows from the spectral properties of the operator implicit in~\eqref{eq:bernstein-conditional-identity}.  More precisely, introduce the auxiliary reference experiment
\[
U\sim\operatorname{Unif}[0,1],
\qquad
S\mid U=u\sim\Bin(n,u).
\]
The marginal law of $S$ is uniform on $\{0,\dots,n\}$, since
$\Pp(S=k)=\int_0^1b_{k,n}(u)\,\dd u=1/(n+1)$.  The uniform law of $U$ is a reference measure for the operator calculation, not an assumption on the unknown mixing distribution $G$.
Accordingly, consider $\mathbb R^{n+1}$ with normalized inner product, i.e.,
\[
\mathsf H_n
:=
L_2(\{0,\dots,n\},\nu_n),
\qquad
\inner{a}{b}_{\nu_n}
:=
\frac1{n+1}\sum_{k=0}^na_kb_k,
\]
where $\nu_n$ is the uniform distribution on  $\{0,\dots,n\}$.  The conditional-expectation operator
\begin{equation}
\label{eq:Tn-conditional}
T_n:\mathsf H_n\longrightarrow L_2[0,1],
\qquad
(T_na)(u)
=
\E[a_S\mid U=u]
=
\sum_{k=0}^na_kb_{k,n}(u)
\end{equation}
has range $\Pi_n$.  Its adjoint reverses the conditioning:
\begin{equation}
\label{eq:Tn-adjoint-conditional}
T_n^*:L_2[0,1]\longrightarrow\mathsf H_n,
\qquad
(T_n^*f)(k)
=
\E[f(U)\mid S=k]
=
(n+1)\int_0^1f(u)b_{k,n}(u)\,\dd u,
\end{equation}
where $U\mid S=k\sim\operatorname{Beta}(k+1,n-k+1)$.  Thus
\[
D_n := T_n T_n^*
\]
is the classical Bernstein--Durrmeyer operator, with
$
(D_nf)(u)
=
\E\!\left[\E[f(U)\mid S]\mid U=u\right].
$
Equivalently, $D_n$ is the Markov smoothing operator obtained by drawing $S\sim\Bin(n,u)$ and then $U'\sim\operatorname{Beta}(S+1,n-S+1)$; its value at $u$ is $\E[f(U')\mid U=u]$.  It is a positive, self-adjoint operator of rank $n+1$, with range $\Pi_n$.

Let $L_0,L_1,\dots$ be the orthonormal shifted Legendre polynomials, obtained by orthogonalizing $1,u,u^2,\dots$ in $L_2[0,1]$; they turn out to be the eigenfunctions of $D_n$ \citep{DerriennicDurrmeyer}.
More precisely:
\begin{proposition}[Durrmeyer spectral decomposition]
\label{prop:durrmeyer}
For $0\le j\le n$,
\begin{equation}
\label{eq:durrmeyer-eigen}
D_nL_j=\lambda_{j,n}L_j,
\qquad
\lambda_{j,n}
=
\frac{n!(n+1)!}{(n-j)!(n+j+1)!}
=
\prod_{r=0}^{j-1}\frac{n-r}{n+r+2}.
\end{equation}
Define $v_j:=\lambda_{j,n}^{-1/2}T_n^*L_j$.  Then $v_0,\dots,v_n$ form an orthonormal basis of $\mathsf H_n$, and
\begin{equation}
\label{eq:durrmeyer-svd}
T_na
=
\sum_{j=0}^n
\sqrt{\lambda_{j,n}}\,
\inner{a}{v_j}_{\nu_n}L_j.
\end{equation}
\end{proposition}
The eigenvalue formula is classical; \cref{app:durrmeyer} gives the short proof, which relies on degree preservation---$\Pi_j$ is an invariant subspace of $D_n$---and orthogonality of $L_j$ to $\Pi_{j-1}$.  Equation~\eqref{eq:durrmeyer-svd} is the singular value decomposition (SVD) of $T_n$.
The singular values $\sqrt{\lambda_{j,n}}$ quantify how strongly the binomial channel attenuates polynomial mode $j$, while $\kappa_{j,n}=\lambda_{j,n}^{-1/2}$ is the corresponding inverse amplification.

The SVD of $T_n$ can be used to control $a(P)$.  We continue to write $\Pi_d$ for the vector space of polynomials of degree at most $d$; the ambient interval will be clear from the argument.  The affine change of variables below identifies its realizations on $[-1,1]$ and $[0,1]$.  For $P$ on $[-1,1]$, define
\begin{equation}
\label{eq:polynomial-affine-transform}
\widetilde P(u):=P(2u-1),
\end{equation}
which is a polynomial on $[0,1]$ of the same degree.
Then the degree-$n$ Bernstein coefficient vector introduced in \eqref{eq:bernstein-representation} is characterized by
\[
T_na(P)=\widetilde P.
\]
Every such polynomial $\widetilde P$ of degree at most $d$ has a unique expansion in $L_0,\dots,L_d$.

\begin{proposition}\label{prop:coefficient-spectral}
Let $P\in\Pi_d$, $d\le n$, and write its shifted-Legendre expansion as
$\widetilde P=\sum_{j=0}^d c_jL_j$.
Then its degree-$n$ Bernstein coefficient vector satisfies
\begin{equation}
\label{eq:coefficient-spectral}
\frac1{n+1}\sum_{k=0}^n a_k(P)^2
=
\sum_{j=0}^d\frac{c_j^2}{\lambda_{j,n}}.
\end{equation}
In particular, since $\kappa_{d,n}:=\lambda_{d,n}^{-1/2}$, one has
\begin{equation}
\label{eq:range-spectral}
A(P)
\le
2\sqrt{n+1}\,\kappa_{d,n}\norm{\widetilde P}_{L_2[0,1]}.
\end{equation}
\end{proposition}
Identity~\eqref{eq:coefficient-spectral} is Parseval's identity in $\mathsf H_n$: although a degree-$d$ polynomial may have all $n+1$ degree-$n$ Bernstein coefficients nonzero, its coefficient vector has only $d+1$ nonzero coordinates in the singular basis $v_0,\dots,v_n$.

The inverse singular values have two useful asymptotic regimes.  For $d\le n/2$,
\begin{equation}
\label{eq:kappa-asymp}
 c\frac{d(d+1)}{n}
\le
\log\kappa_{d,n}
\le
C\frac{d(d+1)}{n},
\end{equation}
where $c,C>0$ are numerical constants.  Thus
\[
\kappa_{d,n}
=
\exp\{\Theta(d^2/n)\}
\qquad(d\ll n),
\]
while $\kappa_{n,n}\asymp2^nn^{-1/4}$.  The elementary calculation is given in \cref{app:durrmeyer}.

Applying \eqref{eq:range-spectral} to the bounded Jackson
polynomial
gives
the
refinement
\begin{equation}
\label{eq:spectral-upper}
\calR^{\mathrm{dec}}_{Q,n}(\beta,L)
\lesssim
\inf_{1\le d\le n}
\left\{
 d^{-\beta}
+
\frac{\min\{B_0^d,\sqrt{n+1}\,\kappa_{d,n}\}}{\sqrt Q}
\right\}.
\end{equation}

The improvement is present even at full degree: the endpoint asymptotic above gives the spectral stability factor $\sqrt{n+1}\,\kappa_{n,n}\asymp2^nn^{1/4}$, rather than $B_0^n$.  For smaller $d$, \eqref{eq:kappa-asymp} captures how stability varies jointly with $d$ and $n$.  Taking $d=n$, the moment estimator reaches risk $O(n^{-\beta})$ whenever
\begin{equation}
\label{eq:large-Q-threshold}
Q
\gtrsim
4^n n^{2\beta+1/2}.
\end{equation}
Combined with \cref{thm:identification-floor}, this gives a second
endpoint regime.

\begin{corollary}[Identification-limited regime]
\label{cor:identification-limited}
There is $C=C(\beta,L)$ such that, if \eqref{eq:large-Q-threshold} holds with a sufficiently large constant, then
\[
\calR^{\mathrm{dec}}_{Q,n}(\beta,L)
\asymp_{\beta,L}
n^{-\beta}.
\]
\end{corollary}

For intermediate degrees, \eqref{eq:spectral-upper} yields rates depending jointly on $Q$ and $n$.  We make the degree choice explicit in the next section, where the resulting upper bound can be compared directly with the spectral moment-cancellation lower bound.

\section{Spectral lower bounds and the middle regime}
\label{sec:spectral-middle}

The preceding section used the Durrmeyer spectrum to control the stability of the polynomial-moment estimator.  We now use its singular decomposition to quantify how difficult it is for the binomial channel to distinguish moment-matched alternatives.  The resulting lower bound can be compared directly with the spectral upper bound, giving explicit rates and a matching middle regime.

\subsection{Spectral lower bound}

The construction reuses the dual family from \cref{thm:identification-floor,lem:duality-moment-pair}.  At cancellation degree $d$, the two alternatives have the same moments up to degree $d$,
while their values of $\theta$ differ by order $d^{-\beta}$.  In the shifted-Legendre basis, this cancellation removes modes $0,\ldots,d$ from the observable perturbation.  Every surviving mode $j\ge d+1$ is multiplied by a singular value no larger than $\sqrt{\lambda_{d+1,n}}=\kappa_{d+1,n}^{-1}$.  When $d=n$, no observable mode remains and the count laws coincide, recovering the identification floor.  For $d<n$, the following theorem quantifies the residual distinguishability.

\begin{theorem}[Spectral moment-cancellation lower bound]
\label{thm:spectral-modulus-lower}
Fix $0<\beta\le1$ and $L\ge1$.  There are constants $0<c_{\beta,L}<C_{\beta,L}<\infty$ such that, for every $n\ge2$ and every integer $1\le d\le n-1$,
\begin{equation}
\label{eq:spectral-modulus-lower}
\omega_{n,\beta,L}\!\left(
C_{\beta,L}\frac{n+1}{\kappa_{d+1,n}}
\right)
\ge
c_{\beta,L}d^{-\beta}.
\end{equation}
If $1/2<\beta\le1$, the following stronger bound holds:
\begin{equation}
\label{eq:spectral-modulus-lower-L2}
\omega_{n,\beta,L}\!\left(
\frac{C_{\beta,L}}{\kappa_{d+1,n}}
\right)
\ge
c_{\beta,L}d^{-\beta}.
\end{equation}
\end{theorem}

The factor $n+1$ in \eqref{eq:spectral-modulus-lower} comes from controlling the spectral tail of an integrable, potentially non-$L_2$, perturbation using supremum norms of the shifted Legendre polynomials.  When $\beta>1/2$, the perturbation is square-integrable and Bessel's inequality removes this factor, giving \eqref{eq:spectral-modulus-lower-L2}.  A complete proof is given in \cref{app:spectral-lower}.

\subsection{Explicit rates and the middle regime}

The upper and lower spectral results use \eqref{eq:kappa-asymp} in opposite directions.  To display this comparison, define
\begin{equation}
\label{eq:spectral-log-scales}
x_{Q,n}
:=
\log\!\left(\frac{Q}{n^{\beta+1}}\right),
\qquad
y_{Q,n}
:=
\log\!\bigl(e(n+1)\sqrt Q\bigr),
\qquad
y_Q^{(2)}
:=
\log(e\sqrt Q).
\end{equation}
The quantity $x_{Q,n}$ measures the logarithmic question budget remaining in the upper bound after accounting for its polynomial prefactor.  The quantity $y_{Q,n}$ is the spectral attenuation required to place the lower-bound alternatives within Hellinger resolution $Q^{-1/2}$; $y_Q^{(2)}$ is its square-integrable refinement.

\begin{proposition}[Explicit spectral rate bounds]
\label{prop:explicit-spectral-rates}
Fix $\beta>0$ and $L\ge1$.  There are constants $c_0,c,C>0$ and $n_0\ge2$, depending only on $(\beta,L)$, such that, for every $n\ge n_0$, the following statements hold.
\begin{enumerate}[label=(\alph*)]
\item If $0<\beta\le1$ and $y_{Q,n}\le c_0n$, then
\begin{equation}
\label{eq:spectral-explicit-lower}
\calR^{\mathrm{dec}}_{Q,n}(\beta,L)
\ge
c\,(ny_{Q,n})^{-\beta/2}.
\end{equation}
If $1/2<\beta\le1$, the same assertion holds with $y_Q^{(2)}$ in place of $y_{Q,n}$, both in the condition and in the bound.

\item If $1\le x_{Q,n}\le c_0n$, then
\begin{equation}
\label{eq:intermediate-upper}
\calR^{\mathrm{dec}}_{Q,n}(\beta,L)
\le
C\,(nx_{Q,n})^{-\beta/2}.
\end{equation}
The bound is achieved by the polynomial-moment estimator with degree $d\asymp\sqrt{nx_{Q,n}}$.
\end{enumerate}
\end{proposition}

The degree choices make the comparison transparent.  For the upper bound, take $d_{\mathrm{up}}\asymp\sqrt{nx_{Q,n}}$ with a sufficiently small proportionality constant.  The upper half of \eqref{eq:kappa-asymp} then gives $\kappa_{d_{\mathrm{up}},n}\le e^{x_{Q,n}/4}$, and the stochastic term in \eqref{eq:spectral-upper} satisfies
\[
\frac{\sqrt{n+1}\,\kappa_{d_{\mathrm{up}},n}}{\sqrt Q}
\lesssim
n^{-\beta/2}e^{-x_{Q,n}/4}
\lesssim_\beta
(nx_{Q,n})^{-\beta/2}.
\]
For the lower bound, take $d_{\mathrm{low}}\asymp\sqrt{ny_{Q,n}}$ with a sufficiently large proportionality constant.  The lower half of \eqref{eq:kappa-asymp} then makes $\kappa_{d_{\mathrm{low}}+1,n}$ at least a fixed multiple of $(n+1)\sqrt Q$.  Thus \eqref{eq:spectral-modulus-lower} places alternatives separated by order $d_{\mathrm{low}}^{-\beta}$ within Hellinger radius $Q^{-1/2}$.  The version using $y_Q^{(2)}$ follows identically from \eqref{eq:spectral-modulus-lower-L2}.  The complete calculations, including integer rounding and constants, are given in \cref{app:spectral-lower}.

When the two logarithmic scales $x_{Q,n}$ and $y_{Q,n}$ are comparable, the preceding upper and lower bounds match.  This gives the following rigorous middle regime.

\begin{theorem}[A rigorous middle regime]
\label{thm:rigorous-middle}
Fix $0<\beta\le1$, $L\ge1$, and $\eta>0$.  There are constants $0<c_{\beta,L,\eta}<C_{\beta,L,\eta}<\infty$, $c_\eta>0$, and $n_0$ such that, for every $n\ge n_0$ and
\begin{equation}
\label{eq:rigorous-middle-range}
n^{\beta+1+\eta}
\le Q
\le e^{c_\eta n},
\end{equation}
\begin{equation}
\label{eq:rigorous-middle}
c_{\beta,L,\eta}
[n\log Q]^{-\beta/2}
\le
\calR^{\mathrm{dec}}_{Q,n}(\beta,L)
\le
C_{\beta,L,\eta}
[n\log Q]^{-\beta/2}.
\end{equation}
The upper bound is attained by the pooled polynomial-moment estimator with degree $d\asymp\sqrt{n\log Q}$.
\end{theorem}

\begin{proof}
Under the lower inequality in \eqref{eq:rigorous-middle-range},
\[
x_{Q,n}
\asymp_{\beta,\eta}
y_{Q,n}
\asymp_{\beta,\eta}
\log Q.
\]
The upper inequality in \eqref{eq:rigorous-middle-range}, with $c_\eta$ sufficiently small, places both logarithmic scales in the ranges required by \cref{prop:explicit-spectral-rates}.  Combining its lower and upper assertions proves \eqref{eq:rigorous-middle}.
\end{proof}

\subsection{First transition and spectral conjecture}

The middle theorem deliberately stays away from the transition at the end of the question-limited regime.  Above that endpoint, the currently proved bounds include
\begin{equation}
\label{eq:first-transition-strip}
\max\left\{
Q^{-1/2},
(ny_{Q,n})^{-\beta/2}
\right\}
\lesssim_{\beta,L}
\calR^{\mathrm{dec}}_{Q,n}(\beta,L)
\lesssim_{\beta,L}
Q^{-1/2}+n^{-\beta/2},
\end{equation}
provided $0<\beta\le1$ and $y_{Q,n}\le c_0n$, with $c_0$ as in \cref{prop:explicit-spectral-rates}.  When $1/2<\beta\le1$, $y_{Q,n}$ may be replaced by $y_Q^{(2)}$.  For polynomial $Q$ with $n^\beta\ll Q\lesssim n^{\beta+1}$, this leaves at most a logarithmic gap between $(n\log n)^{-\beta/2}$ and $n^{-\beta/2}$ once the spectral term dominates the root-$Q$ term.

The exact spectrum suggests that the intrinsic inverse difficulty at degree $d$ is $\kappa_{d,n}$ itself.

\begin{conjecture}[Spectral minimax formula]
\label{conj:Psi}
For fixed $\beta,L$, the scalar minimax risk satisfies
\begin{equation}
\label{eq:Psi}
\calR^{\mathrm{dec}}_{Q,n}(\beta,L)
\asymp_{\beta,L}
\Psi_\beta(Q,n),
\qquad
\Psi_\beta(Q,n)
:=
\inf_{1\le d\le n}
\left\{
 d^{-\beta}
+
\frac{\kappa_{d,n}}{\sqrt Q}
\right\}.
\end{equation}
\end{conjecture}

Optimizing \eqref{eq:Psi} using \eqref{eq:kappa-asymp} gives the three-phase rate picture in \eqref{eq:intro-three-phase}; its interpretation and relation to the rigorous results are discussed in the introduction.

To see what the conjecture requires of the modulus, write $\eps=Q^{-1/2}$.  By \cref{thm:scalar-modulus,lem:modulus-scaling}, \eqref{eq:Psi} predicts
\[
\omega_{n,\beta,L}(\eps)
\asymp_{\beta,L}
\inf_{1\le d\le n}
\left\{d^{-\beta}+\eps\kappa_{d,n}\right\}.
\]
At an interior optimizing degree, the two terms balance, so $\eps\asymp d^{-\beta}/\kappa_{d,n}$ and the infimum is of order $d^{-\beta}$.  Thus the conjecture predicts
\begin{equation}
\label{eq:conjectured-modulus-pair}
\omega_{n,\beta,L}\!\left(\frac{d^{-\beta}}{\kappa_{d,n}}\right)
\asymp_{\beta,L}
d^{-\beta},
\end{equation}
where
multiplicative constants in the argument of the modulus
are absorbed into $\asymp_{\beta, L}$ by \cref{lem:modulus-scaling}.  By the definition of the modulus, the lower bound in \eqref{eq:conjectured-modulus-pair} requires alternatives whose target values differ by order $d^{-\beta}$ while their count laws are within Hellinger distance of order $d^{-\beta}/\kappa_{d,n}$.  The current moment-cancellation construction has the required target separation, but \cref{thm:spectral-modulus-lower} places its count laws only within radius $(n+1)/\kappa_{d+1,n}$, improved to $1/\kappa_{d+1,n}$ when $\beta>1/2$.

The upper side has a separate stability gap.  For a polynomial $P$, define the robust stability norm
\[
\mathsf S_n(P)
:=
\sup_{G\in\calG_{\beta,L}}
\sqrt{\Var_{S\sim K_nG}(a_S(P))}.
\]
An exact upper proof would follow from degree-$d$ approximants satisfying
\[
\sup_G\int|P_d-\sgn|\,\dd G\lesssim d^{-\beta},
\qquad
\mathsf S_n(P_d)\lesssim\kappa_{d,n}.
\]
The present coefficient-range argument instead gives the larger stability factor $\sqrt{n+1}\,\kappa_{d,n}$.  Closing these upper and lower gaps would establish the conjectured formula.

\section{Discussion}
\label{sec:back-general}

We studied the problem of learning LLM ensemble weights from labeled questions and sampled answers, with performance measured by the clean risk.
The interplay between $Q$ and $N$ affects how the observations should be used.  Our general theory shows that estimating each question's answer distribution separately
is enough to attain the optimal rate, up to logarithmic factors,
when questions are relatively scarce.  Our scalar analysis shows that pooling generation counts
can improve the rate
once sufficiently many questions are available.  The root-$Q$ obstruction persists even under a hard margin; improving it uniformly requires additional assumptions, such as a Bernstein condition linking loss disagreement to excess risk \citep{MassartNedelec}.

We study what sampled answers reveal without assuming an exploitable relation between question features and model answer distributions.  That is, the LLM itself processes the question text, but we do not assume that its answer distribution can be predicted from that text alone, without generating answers.
This provides a baseline for learning ensembles from sampled answers.  Additional assumptions describing the predictive value of question texts
 could permit further gains.

We considered the scalar antisymmetric model mainly as a device to study the fundamental limits of learning from sampled answers,
and it already supplies lower bounds for general ensembles (\cref{cor:general-lower-transfer}). However, it would be of interest to study two models with two possible answers per question without imposing antisymmetry, and determine whether matching upper bounds can be obtained.
In this setting, a question has two latent signed advantages $(Z_1,Z_2)$ and the clean margin of the ensemble with weights $(t,1-t)$ on question $q$ is
$
M_q(t)=tZ_1+(1-t)Z_2.
$
The learner must estimate and optimize the risk curve, $t\mapsto \Pp\{tZ_1+(1-t)Z_2\le0\}$, over $t\in[0,1]$
and, for general $K$ and multiple wrong answers, a family of polytopal decision risks.  One possible approach is to approximate the loss indicators by multivariate polynomials and minimize the estimated surrogate risk over $w$.  Our current scalar construction
balances approximation error
versus
the variability induced by the polynomial's Bernstein coefficients (\cref{prop:bias-stability,prop:coefficient-spectral}).  Extending this argument to multivariate surrogate losses would require
similar control uniformly over the candidate weights. For general $K$, this could provide the matching upper bound when $Q$ is exponential in $N$ or larger, addressing an open problem outlined in Section~\ref{sec:overview}.

Choosing the approximation degree also raises an adaptation problem.  The general buffer $\eps_N$ does not require knowledge of $\beta$, while the pooling analysis is tuned to a margin class with known $\beta$.  In the general theorem, $\beta$ controls the frequency of small positive margins at $w^\star$; in the scalar class, it controls latent mass on both sides of zero.  Inferring this behaviour from empirical margin histograms requires accounting for generation noise and, when the weights are fitted, their dependence on the training data.  Adapting pooled estimators to unknown margin behaviour could also inform how a training budget is divided between more questions and more generations per question.

Deployment introduces an additional sampling problem.  \Cref{cor:finite-deployment}
gives a glimpse of what is possible,
extending the buffered learning guarantee to $T$ deployment generations by reserving an additional margin.
The result still compares with the optimal clean risk (i.e., infinite-generation performance).
Ensembles with similar clean risk may differ in both their mean generation cost $\bar c(w)$ and the number of generations needed for reliable voting under deployment constraints.
Establishing savings at a prescribed accuracy therefore requires accounting for both.  To explore this trade-off, one could compare the deployment performance of the learned ensemble with that of the best ensemble under the same fixed generation budget.  Another possibility is to allow the number of generations to vary across questions, subject to an average budget, and study adaptive stopping in that setting.  We leave these questions for future work.

\paragraph*{Acknowledgements.}
This work was supported in part by NSF grant DMS-1945667.
We made use of large language models from OpenAI and Anthropic in developing this work. Their assistance included literature searches, estimator design, proof development and checking, and drafting and revising the manuscript. Both the mathematical arguments and their presentation were refined through sustained, iterative exchanges between the authors and these models. The authors guided this process and take full responsibility for the results and the final manuscript.

\appendix

\section{Technical lemmas for buffered ERM}
\label{app:buffered-proof}

\subsection{Uniform accuracy of the empirical margins}

For the proof, express the margins as affine comparisons in $w$.  Set
\[
p_\cdot(a\mid q)
:=
\bigl(p_1(a\mid q),\dots,p_K(a\mid q)\bigr),
\qquad
\widehat p_\cdot(a\mid q_j)
:=
\bigl(\widehat p_1(a\mid q_j),\dots,
\widehat p_K(a\mid q_j)\bigr),
\]
and define
\begin{equation}
\label{eq:gap-vectors}
\Delta_{q,a}
:=
p_\cdot(a^\star(q)\mid q)-p_\cdot(a\mid q),
\qquad
\widehat\Delta_{q_j,a}
:=
\widehat p_\cdot(a^\star(q_j)\mid q_j)
-\widehat p_\cdot(a\mid q_j).
\end{equation}
Then
\begin{equation}
\label{eq:margin-linear-form}
M_q(w)=\min_{a\ne a^\star(q)}\inner{w}{\Delta_{q,a}},
\qquad
\widehat M_{q_j}(w)
=\min_{a\ne a^\star(q_j)}\inner{w}{\widehat\Delta_{q_j,a}}.
\end{equation}

\begin{lemma}[Uniform margin accuracy]
\label{lem:margin-accuracy}
With probability at least $1-\delta/4$ over the generations, conditional on $q_1,\dots,q_Q$,
\begin{equation}
\label{eq:margin-accuracy}
\sup_{j\le Q}\sup_{w\in\Delta_K}
\abs{\widehat M_{q_j}(w)-M_{q_j}(w)}
\le\eps_N.
\end{equation}
Consequently, on this event, for every $j$ and $w$,
\begin{equation}
\label{eq:sandwich}
\one\{M_{q_j}(w)\le0\}
\le
\one\{\widehat M_{q_j}(w)\le\eps_N\}
\le
\one\{M_{q_j}(w)\le2\eps_N\}.
\end{equation}
\end{lemma}

\begin{proof}
For each $(j,i,a)$, Hoeffding's inequality gives
\[
\Pp\left(
\abs{\widehat p_i(a\mid q_j)-p_i(a\mid q_j)}>\frac{\eps_N}{2}
\right)
\le
2\exp\left(-\frac{N\eps_N^2}{2}\right).
\]
A union bound over at most $KAQ$ triples makes the probability of any violation at most $\delta/4$.  On the complementary event,
\[
\abs{\inner{w}{\widehat\Delta_{q_j,a}-\Delta_{q_j,a}}}
\le\eps_N
\]
for every $w\in\Delta_K$.  Taking minima over wrong answers preserves this bound and gives \eqref{eq:margin-accuracy}; the indicator sandwich follows.
\end{proof}

\subsection{Uniform convergence across questions}

\begin{lemma}[Uniform convergence of clean threshold risks]
\label{lem:clean-risk-deviation}
Fix two deterministic thresholds $s_1,s_2\ge0$.
With probability at least $1-\delta/4$ over the questions,
\begin{equation}
\label{eq:uniform-margin-deviation}
\sup_{w\in\Delta_K,\,t\in\{s_1,s_2\}}
\left|
\frac1Q\sum_{j=1}^Q\one\{M_{q_j}(w)\le t\}
-
\Pp\{M_q(w)\le t\}
\right|
\le\mathfrak v_Q(\delta).
\end{equation}
\end{lemma}

We prove the lemma by first identifying the empirical process that
appears in \eqref{eq:uniform-margin-deviation}.  For a binary function
$f$ on the question space, write
\[
Pf:=\E_{q\sim\calQ}f(q),
\qquad
P_Qf:=\frac1Q\sum_{j=1}^Qf(q_j).
\]
For a class $\mathcal H$ of binary functions, its growth function is
\begin{equation}
\label{eq:growth-function-definition}
\Pi_{\mathcal H}(m)
:=
\sup_{q_1,\ldots,q_m}
\left|
\left\{
(h(q_1),\ldots,h(q_m)):h\in\mathcal H
\right\}
\right|.
\end{equation}
For $t\ge0$, define
\begin{equation}
\label{eq:margin-class}
f_{w,t}(q):=\one\{M_q(w)\le t\},
\qquad
\calF_t:=\{f_{w,t}:w\in\Delta_K\},
\end{equation}
and let $\calF_{\mathrm{buf}}:=\calF_{s_1}\cup\calF_{s_2}$.  With this
notation, the left side of \eqref{eq:uniform-margin-deviation} is exactly
\begin{equation}
\label{eq:clean-risk-empirical-process}
\sup_{f\in\calF_{\mathrm{buf}}}|P_Qf-Pf|.
\end{equation}
We next bound the growth function of this class.  The following classical
face-count bound for hyperplane arrangements goes back to
\citet{ZaslavskyArrangements}.

\begin{lemma}[Weak sign patterns of affine functions]
\label{lem:weak-sign-patterns}
Fix affine functions $g_1,\dots,g_H:\R^p\to\R$.  Their weak sign patterns
are the vectors
\[
\bigl(\sgn g_1(x),\dots,\sgn g_H(x)\bigr)
\in\{-1,0,+1\}^H,
\qquad x\in\R^p.
\]
The number of distinct weak sign patterns is at most
\begin{equation}
\label{eq:faces-bound}
B(H,p):=\sum_{r=0}^{p}2^r\binom Hr.
\end{equation}
If $H\ge p\ge1$, then $B(H,p)\le(2eH/p)^p$.
If $H<p$, then $B(H,p)=3^H$.  In either case,
$\log B(H,p)\le Cp\log(eH)$ for a universal constant $C$.
\end{lemma}

\begin{proof}
Constant affine functions may be discarded because they do not increase
the number of patterns.  The hyperplanes $\{g_h=0\}$ partition $\R^p$
into relatively open faces of all dimensions.  The weak sign vector is
constant on each face, so it suffices to bound the total number of faces.

Let $F(H,p)$ be the largest possible number of such faces.  We prove the
bound by induction on $H$.  After arranging the first $H-1$ hyperplanes,
add the $H$th.  The preceding hyperplanes induce an arrangement of at
most $H-1$ hyperplanes on the new hyperplane, and hence at most
$F(H-1,p-1)$ faces there.  Each induced face can split one old face into
two while itself becoming a new face, producing a net increase of at
most two.  Therefore
\[
F(H,p)\le F(H-1,p)+2F(H-1,p-1),
\]
with $F(0,p)=F(H,0)=1$.  The quantity $B(H,p)$ obeys the same recurrence
by Pascal's identity, so induction proves $F(H,p)\le B(H,p)$.  Finally,
when $H\ge p$,
\[
B(H,p)
\le\sum_{r=0}^{p}\binom{2H}{r}
\le\left(\frac{2eH}{p}\right)^p.
\]
If $H<p$, then $\binom Hr=0$ for $r>H$, and the binomial theorem
gives $B(H,p)=\sum_{r=0}^H2^r\binom Hr=3^H$.  The stated logarithmic
bound follows from these two cases.
\end{proof}

\begin{lemma}[Growth of the clean margin classes]
\label{lem:capacity}
Let $p=K-1$.  For every $m\ge1$ and $t\ge0$,
\begin{equation}
\label{eq:growth-bound}
\Pi_{\calF_t}(m)
\le
B\bigl(m(A-1),p\bigr).
\end{equation}
Consequently,
\begin{equation}
\label{eq:buffered-growth-bound}
\Pi_{\calF_{\mathrm{buf}}}(m)
\le
2B\bigl(m(A-1),p\bigr),
\qquad
\log\Pi_{\calF_{\mathrm{buf}}}(m)
\le C\bigl(1+p\log(eAm)\bigr),
\end{equation}
where $C$ is universal.
\end{lemma}

\begin{proof}
Fix $q_1,\dots,q_m$ and $t$.  For every wrong answer
$a\ne a^\star(q_j)$, consider the affine function
\[
g_{j,a}(w):=\inner{w}{\Delta_{q_j,a}}-t
\]
on the $(K-1)$-dimensional affine hull of $\Delta_K$.  Since
\[
f_{w,t}(q_j)=1
\quad\Longleftrightarrow\quad
g_{j,a}(w)\le0
\text{ for at least one }a\ne a^\star(q_j),
\]
the label vector $(f_{w,t}(q_1),\dots,f_{w,t}(q_m))$ is determined by
the weak sign pattern of at most $m(A-1)$ affine functions.
\Cref{lem:weak-sign-patterns} proves \eqref{eq:growth-bound}.

The restrictions realized by a union of two classes are contained in
the union of their two sets of restrictions.  Hence
$\Pi_{\calF_{\mathrm{buf}}}(m)\le
\Pi_{\calF_{s_1}}(m)+\Pi_{\calF_{s_2}}(m)$, which gives the first bound
in \eqref{eq:buffered-growth-bound}.  The second follows from
\cref{lem:weak-sign-patterns}, with the additive $\log 2$ accounted
for by the constant term.  This also covers $p=0$, when $B(H,0)=1$.
\end{proof}

We use the following standard consequence of symmetrization.

\begin{lemma}[Growth-function deviation inequality]
\label{lem:growth-deviation}
Let $\mathcal H$ be a class of binary functions and let $P_Q$ be the
empirical measure of $Q$ independent observations from $P$.  For every
$u>0$,
\begin{equation}
\label{eq:vc-ineq}
\Pp\left\{
\sup_{h\in\mathcal H}|P_Qh-Ph|>u
\right\}
\le
2\Pi_{\mathcal H}(2Q)
\exp\left(-\frac{Qu^2}{2}\right).
\end{equation}
\end{lemma}

\begin{proof}
Write $Z:=\sup_{h\in\mathcal H}|P_Qh-Ph|$.  Introduce an
independent ghost sample with empirical measure $P'_Q = \frac1Q \sum_{j=1}^Q \delta_{q'_j}$, and write
$Z':=\sup_{h\in\mathcal H}|P_Qh-P'_Qh|$.  Conditional on the original
sample, $P_Qh-Ph=\E[P_Qh-P'_Qh\mid q_1,\dots,q_Q]$ for every $h$.
Convexity of the supremum norm and of $x\mapsto e^{\lambda x}$ therefore
gives $\E e^{\lambda Z}\le\E e^{\lambda Z'}$ for every $\lambda>0$.

By exchangeability of each original--ghost pair, $Z'$ has the same
distribution after independent random swaps.  Encode the swaps by
Rademacher signs $\sigma_1,\dots,\sigma_Q$.  Conditional on the combined
sample, a fixed restriction of $h$ gives coefficients
$d_j=h(q_j)-h(q'_j)\in\{-1,0,+1\}$, and the randomized difference is
$Q^{-1}\sum_{j=1}^Q\sigma_jd_j$.  There are at most
$\Pi_{\mathcal H}(2Q)$ such restrictions.  Hoeffding's exponential-moment
bound, together with $e^{\lambda|x|}\le e^{\lambda x}+e^{-\lambda x}$,
thus gives
\[
\E e^{\lambda Z}
\le
2\Pi_{\mathcal H}(2Q)
\exp\left(\frac{\lambda^2}{2Q}\right).
\]
Markov's inequality with $\lambda=Qu$ now proves
\eqref{eq:vc-ineq}.
\end{proof}

\begin{proof}[Proof of \cref{lem:clean-risk-deviation}]
By \eqref{eq:clean-risk-empirical-process}, the supremum in
\eqref{eq:uniform-margin-deviation} is
$\sup_{f\in\calF_{\mathrm{buf}}}|P_Qf-Pf|$.  Apply
\cref{lem:growth-deviation} to $\calF_{\mathrm{buf}}$, take $u=\mathfrak v_Q(\delta)$ and use
\eqref{eq:buffered-growth-bound} with $m=2Q$.  Taking the universal
constant $C_0$ in \eqref{eq:vq} sufficiently large makes the resulting
probability at most $\delta/4$.
\end{proof}

\subsection{Proof of the finite-deployment corollary}
\label{app:finite-deployment-proof}

\begin{proof}[Proof of \cref{cor:finite-deployment}]
We first bound the error of the deployment vote on questions with margin
greater than $r$.  The enlarged training buffer will then control the
population mass of the remaining questions.

Fix $w$ and $q$ with $M_q(w)>r$.  For each wrong answer
$a\ne a^\star(q)$, define the vote differences
\[
D_t(a):=\one\{Y_t=a^\star(q)\}-\one\{Y_t=a\},
\qquad t\in[T].
\]
Conditional on $q$, these variables are independent, take values in
$[-1,1]$, and have mean
$p_w(a^\star(q)\mid q)-p_w(a\mid q)\ge M_q(w)>r$.
Hoeffding's inequality gives
\[
\Pp\!\left\{\sum_{t=1}^T D_t(a)\le0\,\middle|\,q\right\}
\le e^{-T M_q(w)^2/2}
\le e^{-Tr^2/2}.
\]
The gold answer fails to be the unique empirical plurality only if this
event occurs for some wrong answer.  A union bound over at most $A-1$
wrong answers therefore gives the conditional error bound
$(A-1)e^{-Tr^2/2}$.  On questions with $M_q(w)\le r$, we bound the
conditional error probability by one.  Averaging over questions gives
\begin{equation}
\label{eq:deployment-vote-bound}
R_T^{\mathrm{dep}}(w)
\le R_r(w)+(A-1)e^{-Tr^2/2},
\qquad R_r(w):=\Pp\{M_q(w)\le r\}.
\end{equation}
This bound holds for every $w$.  Conditional on the training data, the
learned weight is fixed, and the fresh deployment question and generations
have the law used above.

The rest of the argument follows the proof of \cref{thm:buffered-main}, and we provide a brief sketch below.
Let $\widetilde R_s$ be the clean empirical threshold risk defined there.  Apply
\cref{lem:clean-risk-deviation} at the fixed thresholds $r$ and
$2\eps_N+r$.  With probability at least $1-\delta/4$ over the questions,
$|\widetilde R_s(w)-R_s(w)|\le\mathfrak v_Q(\delta)$ simultaneously
for all $w$ and both thresholds.  Separately,
\cref{lem:margin-accuracy} bounds the uniform training-margin error by
$\eps_N$ with probability at least $1-\delta/4$, conditional on the
questions.  On this latter event, the buffered sandwich
\eqref{eq:intro-sandwich}, with thresholds shifted by $r$, holds for every
$j\in[Q]$ and $w\in\Delta_K$:
\[
\one\{M_{q_j}(w)\le r\}
\le \one\{\widehat M_{q_j}(w)\le\eps_N+r\}
\le \one\{M_{q_j}(w)\le2\eps_N+r\}.
\]
Averaging over the training questions gives
\begin{equation}
\label{eq:deployment-buffer-sandwich}
\widetilde R_r(w)
\le \widehat R_{\eps_N+r}(w)
\le \widetilde R_{2\eps_N+r}(w),
\qquad w\in\Delta_K.
\end{equation}
The two events hold jointly with probability at least $1-\delta$.
The same comparison argument now applies at these shifted thresholds:
the sandwich and empirical
optimality compare the clean empirical risks at
$\widehat w_{\eps_N+r}$ and $w^\star$, while uniform convergence adds
$\mathfrak v_Q(\delta)$ at each end.  Consequently,
\[
R_r(\widehat w_{\eps_N+r})
\le R_{2\eps_N+r}(w^\star)+2\mathfrak v_Q(\delta).
\]
Since $2\eps_N+r\le t_0$, the one-sided margin condition gives
\[
R_{2\eps_N+r}(w^\star)-R(w^\star)
=\Pp\{0<M_q(w^\star)\le2\eps_N+r\}
\le C_m(2\eps_N+r)^\beta.
\]
Combining this inequality with \eqref{eq:deployment-vote-bound} proves
\eqref{eq:finite-deployment-bound}.

For \eqref{eq:finite-deployment-rate}, take
$r_T=\sqrt{2/T}\log(eAT)$.  The last term in
\eqref{eq:finite-deployment-bound} becomes
$(A-1)\exp\{-\log^2(eAT)\}$, which, for fixed $A$, decreases faster than
any fixed power of $T^{-1}$.  Also,
\[
(2\eps_N+r_T)^\beta
\le C_\beta(\eps_N^\beta+r_T^\beta),
\qquad
r_T^\beta=(2/T)^{\beta/2}\log^\beta(eAT).
\]
Here one can take $C_\beta=3^\beta$ for every $\beta>0$, since
$2\eps_N+r_T\le3\max\{\eps_N,r_T\}$.
Substitution of \eqref{eq:epsilon-N} and \eqref{eq:vq} now gives
\eqref{eq:finite-deployment-rate}.
\end{proof}

\section{Technical details for the scalar modulus}
\label{app:scalar-modulus}

\subsection{Geometry of the margin class}

\begin{lemma}[Basic properties of the margin class]
\label{lem:margin-class-properties}
The class $\calG_{\beta,L}$ is convex, invariant under reflection
$G\mapsto G^\dagger$, where $G^\dagger(B)=G(-B)$, and compact for weak convergence.  Every $G\in\calG_{\beta,L}$ satisfies $G(\{0\})=0$.  The functional $G\mapsto\theta(G)$ and the channel map $G\mapsto K_nG$ are affine and continuous on $\calG_{\beta,L}$.  Moreover,
\begin{equation}
\label{eq:reflection-properties}
\theta(G^\dagger)=-\theta(G),
\qquad
(K_nG^\dagger)(k)=(K_nG)(n-k),
\quad k=0,\dots,n.
\end{equation}
\end{lemma}

\begin{proof}
Convexity and reflection invariance are immediate from the definition of the class.  Since $[-1,1]$ is compact, the set of all probability measures on it is weakly compact.  It remains to show that the margin class is weakly closed.  Suppose $G_m\Rightarrow G$.  Fix $0<t<1$ and choose $\eta>0$ such that $t+\eta<1$.  By the Portmanteau theorem applied to the open interval $(-t-\eta,t+\eta)$,
\[
G([-t,t])
\le
G((-t-\eta,t+\eta))
\le
\liminf_{m\to\infty}G_m((-t-\eta,t+\eta))
\le
L(t+\eta)^\beta.
\]
Letting $\eta\downarrow0$ proves the margin inequality for $t<1$; at $t=1$ it is automatic because $G([-1,1])=1\le L$.  Thus $\calG_{\beta,L}$ is weakly closed and hence compact.  Taking $t\downarrow0$ shows that $G(\{0\})=0$.

For continuity of $\theta$, let $\psi_\delta$ be the continuous clipped-sign function that equals $-1$ on $[-1,-\delta]$, equals $z/\delta$ on $[-\delta,\delta]$, and equals $1$ on $[\delta,1]$.  Uniformly over the margin class,
\[
\left|\theta(G)-\int\psi_\delta\,\dd G\right|
\le
2G([-\delta,\delta])
\le
2L\delta^\beta.
\]
The integral of $\psi_\delta$ is weakly continuous.  First letting $m\to\infty$ and then $\delta\downarrow0$ shows that $\theta(G_m)\to\theta(G)$.

Each coordinate of $K_nG$ is the integral of a bounded continuous polynomial in $z$, so the channel map is weakly continuous; both the channel and target maps are affine by linearity of integration.  Finally, reflection changes the sign of the target, and the Bernstein identity
\[
b_{k,n}\!\left(\frac{1-z}{2}\right)
=
b_{n-k,n}\!\left(\frac{1+z}{2}\right)
\]
gives the channel identity in \eqref{eq:reflection-properties}.
\end{proof}

\subsection{A finite-sample Hellinger-modulus upper bound}

The scalar characterization requires only an upper bound for absolute-error estimation.  We give the needed result directly for compact convex categorical models.  The proof follows the classical modulus-via-testing method of \citet{DonohoLiuII}: robust tests separate target level sets, and a multiscale interval search turns those tests into an estimator.  Splitting the observations across scales makes the finite-sample expected-loss calculation straightforward.

Let $\mathcal X$ be a nonempty compact convex subset of $\Delta_m\times\R$, where $\Delta_m$ denotes the probability simplex on a finite alphabet of size $m$.  Write $x=(p,\tau)$ for a generic element and define
\begin{equation}
\label{eq:abstract-modulus}
\omega_{\mathcal X}(r)
:=
\sup\left\{
|\tau-\tau'|:
(p,\tau),(p',\tau')\in\mathcal X,
\ h(p,p')\le r
\right\}.
\end{equation}

\begin{theorem}[Finite-sample Hellinger-modulus upper bound]
\label{thm:modulus-upper}
Under $(p,\tau)\in\mathcal X$, observe $Y_1,\dots,Y_Q\overset{\mathrm{iid}}{\sim}p$.  Then
\begin{equation}
\label{eq:modulus-upper}
\inf_{\widehat\tau}
\sup_{(p,\tau)\in\mathcal X}
\E_p|\widehat\tau-\tau|
\le
\frac72\,
\omega_{\mathcal X}(32Q^{-1/2}).
\end{equation}
The infimum is over all possibly randomized estimators $\widehat\tau$ based on these observations.
The projection $(p,\tau)\mapsto p$ need not be injective, so the result allows $\omega_{\mathcal X}(0)>0$.
\end{theorem}

Two elementary ingredients precede the proof.  The first controls how rapidly the modulus can change with its radius.  It is a consequence of convexity, rather than an additional regularity assumption.

\begin{lemma}[Scaling of the convex Hellinger modulus]
\label{lem:modulus-scaling}
For $0<\eps_1\le\eps_2$,
\begin{equation}
\label{eq:modulus-scaling}
\omega_{\mathcal X}(\eps_1)
\le
\omega_{\mathcal X}(\eps_2)
\le
\left(\frac{\eps_2}{\eps_1}\right)^2
\omega_{\mathcal X}(\eps_1).
\end{equation}
Consequently, replacing an argument of the modulus by any fixed positive multiple changes its value by at most a fixed multiplicative factor.
\end{lemma}

\begin{proof}
The first inequality follows because the set of pairs admissible at radius $\eps_1$ is contained in the set admissible at radius $\eps_2$.  For the second, take any $x_i=(p_i,\tau_i)\in\mathcal X$, $i=0,1$, satisfying $h(p_0,p_1)\le\eps_2$.  Let $\overline x=(x_0+x_1)/2$, and set
\[
\lambda=\left(\frac{\eps_1}{\eps_2}\right)^2,
\qquad
x_i^\lambda=(1-\lambda)\overline x+\lambda x_i,
\quad i=0,1.
\]
Convexity keeps $x_0^\lambda,x_1^\lambda$ in $\mathcal X$.  Their target separation is $\lambda|\tau_1-\tau_0|$, while affinity of the observation coordinate and joint convexity of squared Hellinger distance give
\[
h^2(p_0^\lambda,p_1^\lambda)
\le
\lambda h^2(p_0,p_1)
\le\eps_1^2.
\]
Therefore $\lambda|\tau_1-\tau_0|\le\omega_{\mathcal X}(\eps_1)$.  Dividing by $\lambda$ and taking the supremum over pairs admissible at radius $\eps_2$ proves \eqref{eq:modulus-scaling}.
\end{proof}

The second ingredient constructs one test that is uniformly valid over two convex sets of categorical laws.  This is stronger than testing only the pair of laws that are closest in Hellinger distance.

\begin{lemma}[Robust affinity test]
\label{lem:robust-affinity-test}
Let $\mathcal P_0,\mathcal P_1\subset\Delta_m$ be nonempty compact convex sets, and put
\[
r_*
:=
\inf\{h(p,q):p\in\mathcal P_0,\ q\in\mathcal P_1\}.
\]
For every integer $k\ge1$, there is a possibly randomized test $\phi$ based on observations $Y_1,\dots,Y_k$ such that
\begin{equation}
\label{eq:robust-test-error}
\sup_{p\in\mathcal P_0}\E_{p^{\otimes k}}\phi
\ \vee\
\sup_{q\in\mathcal P_1}\E_{q^{\otimes k}}(1-\phi)
\le
\left(1-\frac{r_*^2}{2}\right)^k
\le
\exp\left(-\frac{k r_*^2}{2}\right),
\end{equation}
where $\phi=1$ means choosing $\mathcal P_1$.
\end{lemma}

\begin{proof}
First suppose that every distribution in $\mathcal P_0\cup\mathcal P_1$ has strictly positive coordinates.  Let $(p_0,q_1)$ maximize the Hellinger affinity
\[
\rho(p,q):=\sum_{i=1}^m\sqrt{p_iq_i}
\]
over $\mathcal P_0\times\mathcal P_1$.  For fixed $q_1$, the map $p\mapsto\rho(p,q_1)$ is concave.  First-order optimality at $p_0$ along every segment from $p_0$ to $p\in\mathcal P_0$ gives
\begin{equation}
\label{eq:affinity-first-order-p}
\sum_{i=1}^m p_i\sqrt{\frac{q_{1i}}{p_{0i}}}
\le
\sum_{i=1}^m p_{0i}\sqrt{\frac{q_{1i}}{p_{0i}}}
=\rho(p_0,q_1).
\end{equation}
Similarly, for every $q\in\mathcal P_1$,
\begin{equation}
\label{eq:affinity-first-order-q}
\sum_{i=1}^m q_i\sqrt{\frac{p_{0i}}{q_{1i}}}
\le\rho(p_0,q_1).
\end{equation}
Use the likelihood-ratio test between $p_0^{\otimes k}$ and $q_1^{\otimes k}$, namely,
$\phi = \one \left\{\prod_{j=1}^k q_1(Y_j)/p_0(Y_j)\ge1 \right\}$.  For $p\in\mathcal P_0$, Markov's inequality and \eqref{eq:affinity-first-order-p} yield
\begin{align*}
\E_{p^{\otimes k}}\phi
\;\le\;
\E_{p^{\otimes k}}
\prod_{j=1}^k
\left(\frac{q_1(Y_j)}{p_0(Y_j)}\right)^{1/2}
=
\left[
\sum_{i=1}^m p_i\sqrt{\frac{q_{1i}}{p_{0i}}}
\right]^k
\le\rho(p_0,q_1)^k.
\end{align*}
The inverse likelihood ratio and \eqref{eq:affinity-first-order-q} give the same bound for the other error.  Since maximizing affinity is equivalent to minimizing Hellinger distance,
\[
\rho(p_0,q_1)=1-\frac12r_*^2,
\]
which proves \eqref{eq:robust-test-error} under strict positivity.

To remove that assumption, let $u$ be the uniform law and, for $\eta\in(0,1)$, set
\[
\mathcal P_j^\eta
:=
\{(1-\eta)p+\eta u:p\in\mathcal P_j\},
\qquad j=0,1.
\]
Apply the preceding construction to obtain a randomized test $\phi_\eta$, viewed as a vector in $[0,1]^{m^k}$.  Along a sequence $\eta_\ell\downarrow0$, compactness gives coordinatewise convergence to a randomized test $\phi$.  Uniform continuity of Hellinger affinity shows that the maximal affinities of the smoothed sets converge to the maximal affinity of the original sets.  For each fixed $p\in\mathcal P_0$ and $q\in\mathcal P_1$, the corresponding smoothed product laws converge in total variation to $p^{\otimes k}$ and $q^{\otimes k}$.  Together with coordinatewise convergence of the tests on the finite sample space, this implies convergence of their testing risks.  Passing to the limit proves the result on the boundary of the simplex.
\end{proof}

\begin{proof}[Proof of \cref{thm:modulus-upper}]
Let
\[
\tau_-:=\min_{(p,\tau)\in\mathcal X}\tau,
\qquad
\tau_+:=\max_{(p,\tau)\in\mathcal X}\tau,
\qquad
D:=\tau_+-\tau_-.
\]
Because all categorical laws are within Hellinger distance $\sqrt2$, we have $D=\omega_{\mathcal X}(\sqrt2)$.  Set
\[
A:=32,
\qquad
\eps_0:=\frac{A}{\sqrt Q},
\qquad
\delta:=\omega_{\mathcal X}(\eps_0).
\]
Our goal is to show an estimator that has expected absolute error at most $7\delta/2$, uniformly over $\mathcal X$.
If $D=0$, the target is constant, and the corresponding constant estimator achieves zero error.  If $\eps_0\ge\sqrt2$, then $\delta=D$, and the midpoint of $[\tau_-,\tau_+]$ has error at most
$D/2 = \delta/2$.
We may therefore assume $\eps_0<\sqrt2$.  The scaling lemma gives
\begin{equation}
\label{eq:abstract-diameter-delta}
D
\le
\frac{2\delta}{\eps_0^2}.
\end{equation}
In particular, $\delta=0$ would imply $D=0$.  If $\delta\ge D/6$, the midpoint again has error at most $D/2\le3\delta$.  It remains to consider
\begin{equation}
\label{eq:nontrivial-search-case}
0<\delta<D/6,
\end{equation}
which, together with \eqref{eq:abstract-diameter-delta}, implies $Q>3A^2$.

We now construct a sample-split interval-search estimator.  It proceeds through stages $s=0,\dots,S-1$ and maintains a nested sequence of intervals $I_s$ intended to localize $\tau$.  Start from $I_0=[\tau_-,\tau_+]$.  The observations are divided into $S$ disjoint blocks, one per stage, with block sizes specified below.  At stage $s$, write $I_s=[a_s,b_s]$ and $\ell_s=b_s-a_s$.  Using only block $s$, test the lower third of $I_s$ against the upper third through the composite hypotheses
\[
H_{s}^-:\ \tau\le a_s+\frac{\ell_s}{3}
\qquad\text{and}\qquad
H_{s}^+:\ \tau\ge a_s+\frac{2\ell_s}{3}.
\]
If the test chooses $H_{s}^-$, retain the lower two thirds, $I_{s+1}=[a_s,a_s+2\ell_s/3]$; if it chooses $H_{s}^+$, retain the upper two thirds, $I_{s+1}=[a_s+\ell_s/3,b_s]$.  The two candidate intervals overlap in the middle third of $I_s$.  Thus, when $\tau$ lies in that middle third, either test outcome retains it.

The interval lengths are deterministic and satisfy
\[
\ell_s=\left(\frac23\right)^sD.
\]
Stop at the first integer $S$ for which $\ell_S\le6\delta$, and return the midpoint $\widehat\tau$ of $I_S$.  It remains to choose the stagewise tests and block sizes.  For $s<S$, define
\[
g_s:=\frac{\ell_s}{3},
\qquad
r_s^2:=\frac{\eps_0^2\ell_s}{6\delta},
\qquad
u_s:=\frac{\ell_s}{\delta}>6.
\]
Here $g_s$ is the separation between the two outer thirds, while $u_s$ measures the current interval length in units of the target accuracy $\delta$.  Since $\ell_s>6\delta$, we have $r_s\ge\eps_0$, while \eqref{eq:abstract-diameter-delta} gives $r_s^2\le1/3<2$.  Thus $\eps_0\le r_s<\sqrt2$, and \cref{lem:modulus-scaling}, applied at radii $\eps_0$ and $r_s$, gives
\begin{equation}
\label{eq:search-separation}
\omega_{\mathcal X}(r_s)
\le
\left(\frac{r_s}{\eps_0}\right)^2\delta
=
\frac{\ell_s}{6}
=\frac{g_s}{2}
<g_s.
\end{equation}

For every possible interval $I_s$, the observation laws $p$ generated under $H_s^-$ and $H_s^+$ form two compact convex sets.  Indeed, each is the image, under the coordinate projection $(p,\tau)\mapsto p$, of the intersection of $\mathcal X$ with the corresponding closed halfspace.
By \eqref{eq:search-separation}, every pair across these sets is separated by more than $r_s$ in Hellinger distance.  If either set is empty, use the corresponding deterministic decision.  Otherwise, apply the test supplied by \cref{lem:robust-affinity-test} using a fresh block of
\begin{equation}
\label{eq:search-allocation}
k_s
:=
\left\lceil
\frac{4}{r_s^2}\log u_s
\right\rceil
\end{equation}
observations.

The use of a fresh block makes the testing guarantee valid conditional on the preceding outcomes.  Let $\mathcal F_s$ be the sigma-field generated by blocks $0,\dots,s-1$ and any auxiliary randomization used through stage $s-1$, so that $I_s$ is $\mathcal F_s$-measurable.  Call the update at stage $s$ \emph{unsafe} if $I_{s+1}$ does not contain $\tau$, and let $A_s$ be the event that all preceding updates are safe.  Since $I_0$ contains $\tau$ and the intervals are nested, $A_s=\{\tau\in I_s\}$.  On $A_s$, either $\tau$ lies in the middle third, in which case an unsafe update is impossible, or the true law $p$ belongs to one of the two tested composite hypotheses.  Since block $s$ is independent of $\mathcal F_s$ and has law $p^{\otimes k_s}$ conditional on that history, the robust test bound gives
\begin{equation}
\label{eq:search-stage-error}
\Pp_p(\tau\notin I_{s+1}\mid\mathcal F_s)
\le
\exp\left(-\frac{k_s r_s^2}{2}\right)
\le
u_s^{-2}
=
\left(\frac{\delta}{\ell_s}\right)^2
\qquad\text{on }A_s.
\end{equation}

We now bound the risk of this estimator.  Let $E_s:=A_s\cap\{\tau\notin I_{s+1}\}$ be the event that the first unsafe update occurs at stage $s$.  If no update is unsafe, then $\tau\in I_S$ and $|\widehat\tau-\tau|\le\ell_S/2\le3\delta$.  On $E_s$, all preceding updates were safe, so $\tau\in I_s$ immediately before the unsafe update.  Every later interval, and hence $\widehat\tau$, remains inside $I_s$; therefore the loss on $E_s$ is at most $\ell_s$.  Moreover, \eqref{eq:search-stage-error} gives
\[
\Pp_p(E_s)
=
\E_p\left[
\one_{A_s}
\Pp_p\{\tau\notin I_{s+1}\mid\mathcal F_s\}
\right]
\le
\left(\frac{\delta}{\ell_s}\right)^2.
\]
The events $E_s$ are disjoint, and consequently
\begin{align*}
\sup_{(p,\tau)\in\mathcal X}
\E_p|\widehat\tau-\tau|
&\le
3\delta
+
\sum_{s=0}^{S-1}
\ell_s\left(\frac{\delta}{\ell_s}\right)^2.
\end{align*}
Finally, $\ell_{S-1}>6\delta$ and the interval lengths form a geometric sequence, so
\[
\sum_{s=0}^{S-1}\frac1{\ell_s}
\le
\frac1{\ell_{S-1}}
\sum_{j=0}^\infty\left(\frac23\right)^j
<\frac1{2\delta}.
\]
Thus the expected loss is at most $7\delta/2$.

It remains to verify that the planned independent blocks fit within the original sample.  From \eqref{eq:search-allocation},
\begin{equation}
\label{eq:allocation-sum-start}
\sum_{s=0}^{S-1}k_s
\le
S+\frac{24Q}{A^2}
\sum_{s=0}^{S-1}\frac{\log u_s}{u_s}.
\end{equation}
Minimality of $S$ gives $6<u_{S-1}\le9$.  Writing $j=S-1-s$, we have $u_s=u_{S-1}(3/2)^j$, and hence
\begin{align}
\sum_{s=0}^{S-1}\frac{\log u_s}{u_s}
 \le
\frac16\sum_{j=0}^\infty
\left(\frac23\right)^j
\left[\log9+j\log\left(\frac32\right)\right]
=
\log\left(\frac92\right).
\label{eq:allocation-geometric-sum}
\end{align}
Also,
\[
S
\le
1+\log_{3/2}\left(\frac{D}{6\delta}\right)
\le
1+\frac{D}{2\delta}
\le
1+\frac{Q}{A^2},
\]
where the second inequality uses $\log_{3/2}x\le3x$ for $x\ge1$.  Combining this with \cref{eq:allocation-sum-start,eq:allocation-geometric-sum} gives
\[
\sum_{s=0}^{S-1}k_s
\le
1+\frac{Q}{A^2}
+\frac{24Q}{A^2}\log\left(\frac92\right)
<Q,
\]
using $A=32$ and $Q>3A^2$.  Any unused observations are ignored.  The estimator is therefore feasible with $Q$ observations, which completes the proof of \eqref{eq:modulus-upper}.
\end{proof}

\subsection{Proof of the scalar minimax characterization}

\begin{proof}[Proof of \cref{thm:scalar-modulus}]
By \cref{lem:margin-class-properties}, the set
\[
\mathcal X_n
:=
\{(K_nG,\theta(G)):G\in\calG_{\beta,L}\}
\subset\Delta_{n+1}\times[-1,1]
\]
is compact and convex, and both coordinates are continuous affine functions of $G$.  Its indirect modulus is exactly \eqref{eq:modulus}.  Applying \cref{thm:modulus-upper} with $\mathcal X=\mathcal X_n$, followed by \cref{lem:modulus-scaling}, gives
\[
\calR^{\mathrm{abs}}_{Q,n}(\beta,L)
\le
C\,\omega_{n,\beta,L}(Q^{-1/2})
\]
for a universal constant $C$.  This proves the rightmost inequality in \eqref{eq:scalar-modulus-chain}; the pointwise reduction \eqref{eq:dec-le-abs} gives the middle inequality.

It remains to prove the lower bound for decision regret.  Fix  $\eps=(2\sqrt Q)^{-1}$
and take any $G,H\in\calG_{\beta,L}$ satisfying $h(K_nG,K_nH)\le\eps$.  After interchanging $G$ and $H$ if necessary, write
\[
\delta:=\theta(G)-\theta(H)\ge0.
\]
The original targets may have the same sign, in which case the same side is optimal under both distributions and the pair need not be difficult for the decision problem.  Reflection symmetrization produces a pair with opposite optimal decisions.  Specifically, define
\[
G_+:=\frac12(G+H^\dagger),
\qquad
G_-:=\frac12(H+G^\dagger).
\]
By \cref{lem:margin-class-properties}, $G_+,G_-\in\calG_{\beta,L}$ and
\[
\theta(G_+)=\frac\delta2,
\qquad
\theta(G_-)=-\frac\delta2.
\]
Writing $P_+=K_nG_+$ and $P_-=K_nG_-$, joint convexity of squared Hellinger distance and invariance under reversal of the count labels give
\begin{align}
h^2(P_+,P_-)
&\le
\frac12h^2(K_nG,K_nH)
+
\frac12h^2(K_nH^\dagger,K_nG^\dagger)
\notag\\
&=
h^2(K_nG,K_nH)
\le\eps^2.
\label{eq:symmetrized-hellinger}
\end{align}

For probability measures $P$ and $R$, write
\[
\rho(P,R):=\sum_x\sqrt{P(x)R(x)}
=1-\frac12h^2(P,R)
\]
for their Hellinger affinity.  Affinity factorizes over product measures, so
\begin{align}
h^2(P^{\otimes Q},R^{\otimes Q})
&=
2\{1-\rho(P,R)^Q\}\notag\\
&=
2\left\{1-\left(1-\frac12h^2(P,R)\right)^Q\right\}
\le
Qh^2(P,R).
\label{eq:product-hellinger}
\end{align}
The last inequality is Bernoulli's inequality.  Under the present Hellinger convention, Cauchy--Schwarz also gives
\begin{align}
\TV(P,R)
&=
\frac12\sum_x
|\sqrt{P(x)}-\sqrt{R(x)}|
(\sqrt{P(x)}+\sqrt{R(x)})\notag\\
&\le
h(P,R)\sqrt{1-\frac14h^2(P,R)}
\le h(P,R).
\label{eq:TV-le-hellinger}
\end{align}
Combining \eqref{eq:symmetrized-hellinger}--\eqref{eq:TV-le-hellinger} with $\eps=(2\sqrt Q)^{-1}$ yields
\begin{equation}
\label{eq:product-TV-small}
\TV(P_+^{\otimes Q},P_-^{\otimes Q})
\le\frac12.
\end{equation}

Now fix any, possibly randomized, decision rule $\widehat\sigma$.  Because a wrong decision costs $\delta/2$ under either $G_+$ or $G_-$ (recall \eqref{eq:scalar-decision-loss}),
\begin{align}
\sup_{F\in\calG_{\beta,L}}
\E_F\mathcal L_F(\widehat\sigma)
&\ge
\frac12\left\{
\E_{G_+}\mathcal L_{G_+}(\widehat\sigma)
+
\E_{G_-}\mathcal L_{G_-}(\widehat\sigma)
\right\}\notag\\
&=
\frac\delta4\left\{
P_+^{\otimes Q}(\widehat\sigma=-1)
+
P_-^{\otimes Q}(\widehat\sigma=+1)
\right\}\notag\\
&\ge
\frac\delta4
\left\{1-\TV(P_+^{\otimes Q},P_-^{\otimes Q})\right\}
\ge
\frac\delta8.
\label{eq:decision-le-cam}
\end{align}
The penultimate inequality is Le Cam's two-point inequality.  Since \eqref{eq:decision-le-cam} holds for every admissible pair $G,H$ and every decision rule, taking first the infimum over $\widehat\sigma$ and then the supremum over all pairs in the definition of the modulus gives
\[
\calR^{\mathrm{dec}}_{Q,n}(\beta,L)
\ge
\frac18\,
\omega_{n,\beta,L}\!\left(\frac{1}{2\sqrt Q}\right).
\]
Finally, \cref{lem:modulus-scaling} gives
\[
\omega_{n,\beta,L}\!\left(\frac{1}{2\sqrt Q}\right)
\ge
\frac14\omega_{n,\beta,L}(Q^{-1/2}),
\]
which proves the lower bound in \eqref{eq:scalar-modulus-chain}.
\end{proof}

\section{Finite-sample scalar bounds}
\label{app:finite-lower}

\begin{proof}[Proof of \cref{thm:rootQ-scalar}]
Take the two distributions $G_{\sigma,\eta}$ displayed after the proposition.  Conditional on $Z=+1$, the count is deterministically $S=n$, while conditional on $Z=-1$, it is $S=0$.  Hence $K_nG_{+1,\eta}$ assigns masses $(1-\eta)/2$ and $(1+\eta)/2$ to $0$ and $n$, respectively, while $K_nG_{-1,\eta}$ interchanges these masses.  Therefore
\[
h^2(K_nG_{+1,\eta},K_nG_{-1,\eta})
=
2-2\sqrt{1-\eta^2}
\le
2\eta^2.
\]
For $0<\eps\le1$, take $\eta=\eps/2$.  The two channel laws are then within Hellinger distance $\eps$, whereas
\[
|\theta(G_{+1,\eta})-\theta(G_{-1,\eta})|
=2\eta=\eps.
\]
Both measures belong to $\calG_{\beta,L}$ for every $\beta>0$ and $L\ge1$, proving \eqref{eq:rootQ-modulus}.  The minimax conclusion \eqref{eq:rootQ-scalar} follows from the lower bound in \cref{thm:scalar-modulus}.
\end{proof}

\subsection{An elementary rare-boundary construction}
\label{app:rare-boundary}

The following construction places a small fraction of the population near the decision boundary, with a common orientation that can be detected only by pooling its weak count evidence.

\begin{theorem}[Pooled rare-boundary lower bound]
\label{thm:pooled-lower}
For every fixed $\beta>0$ and $L\ge1$, there are constants $c_0,c_1>0$ such that, for every $n\ge1$ and $0<\eps\le c_0$,
\begin{equation}
\label{eq:pooled-modulus-lower}
\omega_{n,\beta,L}(\eps)
\ge
c_1\left(\frac{\eps^2}{n}\right)^{\beta/(\beta+2)}.
\end{equation}
Consequently, there is a constant $c_{\beta,L}>0$ such that
\begin{equation}
\label{eq:pooled-lower}
\calR^{\mathrm{dec}}_{Q,n}(\beta,L)
\ge
c_{\beta,L}(Qn)^{-\beta/(\beta+2)}
\end{equation}
for every $Q,n\ge1$.
\end{theorem}

\begin{proof}[Proof of \cref{thm:pooled-lower}]
Consider the fixed family
\[
\mathfrak G_\beta
:=
\{G_{\sigma,m}:\sigma\in\{-1,+1\},\ 0<m\le1/2\},
\qquad
G_{\sigma,m}
=
\frac{1-m^\beta}{2}(\delta_{-1}+\delta_{+1})
+m^\beta\delta_{\sigma m}.
\]
For $0<t<m$ there is no mass in $[-t,t]$, while for $m\le t<1$ the mass is $m^\beta\le t^\beta$.  Hence the family lies in $\calG_{\beta,L}$, independently of $Q,n$, and $\theta(G_{\sigma,m})=\sigma m^\beta$.

Let $P_{\sigma,m}=K_nG_{\sigma,m}$.  Augmenting the observation by revealing whether the latent draw came from the central component and then applying data processing gives
\[
\KL(P_{+1,m}\|P_{-1,m})
\le
m^\beta
\KL\!\left(
\Bin\left(n,\frac{1+m}{2}\right)
\middle\|
\Bin\left(n,\frac{1-m}{2}\right)
\right)
\le 4nm^{\beta+2},
\]
where the last inequality uses
\[
\KL\!\left(
\Ber\left(\frac{1+m}{2}\right)
\middle\|
\Ber\left(\frac{1-m}{2}\right)
\right)
=
m\log\frac{1+m}{1-m}
\le4m^2
\]
for $m\le1/2$.  Since squared Hellinger distance is bounded by Kullback--Leibler divergence, taking
\[
m
=
\left(\frac{\eps^2}{16n}\right)^{1/(\beta+2)}
\]
for $\eps\le c_0(\beta)$ ensures both $m\le1/2$ and $h(P_{+1,m},P_{-1,m})\le\eps$.  Meanwhile,
\[
|\theta(G_{+1,m})-\theta(G_{-1,m})|
=2m^\beta
\asymp_\beta
\left(\frac{\eps^2}{n}\right)^{\beta/(\beta+2)}.
\]
This proves \eqref{eq:pooled-modulus-lower}.  Its minimax consequence follows from \cref{thm:scalar-modulus}, monotonicity of the modulus, and a fixed reduction of the constant multiplying $Q^{-1/2}$ when necessary.
\end{proof}

The lower bound describes a shared weak orientation among a rare set of near-boundary questions.  Unlike sign-then-aggregate, the optimal test for this pair pools the weak count information.  Combined with the two endpoint lower bounds, it gives
\begin{equation}
\label{eq:explicit-lower-envelope}
\calR^{\mathrm{dec}}_{Q,n}(\beta,L)
\gtrsim_{\beta,L}
\max\left\{
Q^{-1/2},\,
n^{-\beta},\,
(Qn)^{-\beta/(\beta+2)}
\right\}
\end{equation}
for every $Q,n\ge1$.

The rare-boundary term can exceed both endpoint terms only when $0<\beta<1$, in the intermediate range
\[
n^{2\beta/(2-\beta)}
\lesssim Q
\lesssim n^{\beta+1}.
\]
To compare it with the spectral obstruction, let $r_{\mathrm{pool}}=(Qn)^{-\beta/(\beta+2)}$ and $r_{\mathrm{sp}}=(ny_{Q,n})^{-\beta/2}$.  Direct calculation gives $r_{\mathrm{sp}}\ge r_{\mathrm{pool}}$ whenever
\[
Q
\gtrsim
n^{\beta/2}y_{Q,n}^{(\beta+2)/2}.
\]
For $0<\beta<1$, this condition holds throughout the asymptotic region in which the rare-boundary term can exceed both $Q^{-1/2}$ and $n^{-\beta}$.  Thus the moment-canceling alternatives in \cref{sec:spectral-middle}, rather than the single rare weak atom, determine the stronger known middle-regime obstruction.

\subsection{Proof of the plug-in bound}
\label{app:plugin-proof}

\begin{proof}[Proof of \cref{thm:plugin-scalar}]
Let $(Z,S,\widehat\xi)$ denote a generic copy of $(Z_j,S_j,\widehat\xi_j)$, where $\widehat\xi=\one\{S>n/2\}-\one\{S<n/2\}$.  Then $\E_G\widehat\theta_{\mathrm{plug}}=\E_G\widehat\xi$, and the bias--fluctuation decomposition gives
\[
\E_G|\widehat\theta_{\mathrm{plug}}-\theta(G)|
\le
|\E_G\widehat\xi-\theta(G)|
+
\E_G|\widehat\theta_{\mathrm{plug}}-\E_G\widehat\theta_{\mathrm{plug}}|.
\]
Because the $\widehat\xi_j$ are iid and bounded by one, the fluctuation term satisfies
\[
\E_G|\widehat\theta_{\mathrm{plug}}-\E_G\widehat\theta_{\mathrm{plug}}|
\le
\sqrt{\Var_G(\widehat\theta_{\mathrm{plug}})}
\le
Q^{-1/2}.
\]

It remains to control the bias.  The margin condition implies $G(\{0\})=0$, so the value of $\sgn$ at zero is immaterial.  For $Z\ne0$, the pointwise inequality $|\widehat\xi-\sgn(Z)|\le2\one\{\widehat\xi\ne\sgn(Z)\}$ gives
\begin{align*}
|\E_G\widehat\xi-\theta(G)|
&\le
\E_G\bigl|\E[\widehat\xi\mid Z]-\sgn(Z)\bigr|\\
&\le
\E_G\E\bigl[|\widehat\xi-\sgn(Z)|\mid Z\bigr]\\
&\le
2\,\E_G\Pp\{\widehat\xi\ne\sgn(Z)\mid Z\}.
\end{align*}
For $z\ne0$, conditional Hoeffding bounds the probability of an incorrect sign or a tie by
\[
\Pp_z\{\widehat\xi\ne\sgn(z)\}
\le
e^{-nz^2/2}.
\]
Consequently,
\begin{equation}
\label{eq:plugin-bias-preintegration}
|\E_G\widehat\xi-\theta(G)|
\le
2\,\E_Ge^{-nZ^2/2}.
\end{equation}

To integrate this bound, put $a=n^{-1/2}$ and split the range of $|Z|$ into
\[
A_{-1}:=\{z\in[-1,1]:|z|\le a\},
\qquad
A_j:=\{z\in[-1,1]:2^ja<|z|\le2^{j+1}a\},
\quad j\ge0.
\]
The margin condition, together with the trivial probability bound once $2^{j+1}a>1$, gives
\[
G(A_{-1})\le Ln^{-\beta/2},
\qquad
G(A_j)\le L2^{(j+1)\beta}n^{-\beta/2}.
\]
On $A_j$, one has $e^{-nZ^2/2}\le e^{-2^{2j-1}}$.  Therefore
\begin{align*}
\E_Ge^{-nZ^2/2}
&\le
Ln^{-\beta/2}
\left[
1+2^\beta\sum_{j\ge0}2^{j\beta}e^{-2^{2j-1}}
\right]
\;\le
C_{\beta,L}n^{-\beta/2},
\end{align*}
because the series converges.  Combining this estimate with \eqref{eq:plugin-bias-preintegration} and the fluctuation bound proves \eqref{eq:plugin-scalar}.  Finally, \eqref{eq:pointwise-decision-reduction} gives the same upper bound for the regret of $\widehat\sigma_{\mathrm{plug}}$.
\end{proof}

\section{Transfer from count-only to observed-question experiments}
\label{app:observed-transfer}

The transfer proof has two auxiliary steps.  First, we approximate each
scalar mixing distribution by an equal-weight discrete measure that remains
in the same margin class while preserving the target and observable count
law asymptotically.  Second, we realize such a measure as the pushforward of
a deterministic question-to-advantage map by assigning its atoms to a fine
equal-mass partition of the question space.  A uniform random permutation of
this assignment produces an observed-question experiment that, after
averaging, is close to the corresponding count-only experiment; the
discrepancy is controlled by the probability of cell collisions.  We apply
these constructions to a reflected pair from the Hellinger modulus and use
the resulting permutation priors to prove the minimax lower bound.  Finally,
embedding the two informative models among $K-2$ neutral models transfers the
bound to the general ensemble problem.

\subsection{Equal-weight discretization}

\begin{lemma}[Equal-weight discretization preserving the margin class]
\label{lem:equal-weight-discretization}
For every $G\in\calG_{\beta,L}$ and every integer $M\ge1$, there is an equal-weight measure
\[
G^{(M)}
=
\frac1M\sum_{r=1}^M\delta_{z_r}
\in\calG_{\beta,L}
\]
such that
\[
|\theta(G^{(M)})-\theta(G)|\le M^{-1},
\qquad
G^{(M)}\Rightarrow G
\quad\text{as }M\to\infty.
\]
Consequently, for every fixed $n$,
\[
\TV(K_nG^{(M)},K_nG)\to0.
\]
If $H$ is the reflection of $G$ under $z\mapsto-z$, the approximations may be chosen so that $H^{(M)}$ is the reflection of $G^{(M)}$.
\end{lemma}

\begin{proof}[Proof of Lemma~\ref{lem:equal-weight-discretization}]
Because $G\in\calG_{\beta,L}$, $G(\{0\})=0$.  Write
\[
p_+=G((0,1]),
\qquad
p_-=G([-1,0)),
\qquad
p_++p_-=1,
\]
and define the one-sided small-ball functions
\[
F_+(t)=G((0,t]),
\qquad
F_-(t)=G([-t,0)),
\qquad 0<t\le1.
\]

The construction treats the two sides of zero separately.  Starting at
zero, divide as much of the mass on each side as possible into blocks of
mass $1/M$, and replace each block by an atom at its endpoint farther from
zero.  Any atom needed to round the number of blocks is placed at $+1$ or
$-1$.  Moving mass outward in this way is what preserves the margin bound.

To make this precise, choose an integer $m_+$ nearest to $Mp_+$ and put
$m_-=M-m_+$.  These are the numbers of atoms assigned to the positive and
negative sides.  They satisfy
\begin{equation}
\label{eq:side-mass-rounding}
\left|\frac{m_+}{M}-p_+\right|\le\frac1{2M},
\qquad
\left|\frac{m_-}{M}-p_-\right|\le\frac1{2M}.
\end{equation}
Let
\[
k_+=\lfloor Mp_+\rfloor,
\qquad
k_-=\lfloor Mp_-\rfloor,
\qquad
e_+=m_+-k_+,
\qquad
e_-=m_--k_-.
\]
Here $k_\pm$ count the complete mass-$1/M$ blocks and $e_\pm$ account for
the possible rounding atom.  Since $Mp_++Mp_-=M$, we have
$e_+,e_-\in\{0,1\}$ and $e_++e_-\in\{0,1\}$.  For
$j=1,\dots,k_+$ and $j=1,\dots,k_-$, respectively, define the outward
quantiles
\[
x_j^+
=
\inf\{t\in(0,1]:F_+(t)\ge j/M\},
\qquad
x_j^-
=
\inf\{t\in(0,1]:F_-(t)\ge j/M\}.
\]
Thus $x_j^+$ and $-x_j^-$ are the endpoints farther from zero of the
$j$th cumulative blocks on the two sides.  Set
\begin{equation}
\label{eq:outward-quantile-measure}
G^{(M)}
=
\frac1M
\left\{
\sum_{j=1}^{k_+}\delta_{x_j^+}
+
\sum_{j=1}^{k_-}\delta_{-x_j^-}
+
e_+\delta_{+1}
+
e_-\delta_{-1}
\right\}.
\end{equation}
For $0<t<1$, the generalized-inverse property gives
$x_j^+\le t$ if and only if $j/M\le F_+(t)$.  Consequently,
$\#\{j:x_j^+\le t\}=\lfloor MF_+(t)\rfloor$; the same argument applies
on the negative side.  Because the possible extra atoms at $\pm1$ are not
counted, we obtain
\begin{equation}
\label{eq:outward-quantile-subdistributions}
G^{(M)}((0,t])
=
\frac{\lfloor MF_+(t)\rfloor}{M},
\qquad
G^{(M)}([-t,0))
=
\frac{\lfloor MF_-(t)\rfloor}{M}.
\end{equation}
In particular, the discretization places no more mass near zero than the
original distribution:
\[
G^{(M)}([-t,t])
\le
F_+(t)+F_-(t)
=
G([-t,t])
\le
Lt^\beta.
\]
At $t=1$, the required inequality is $1\le L$.  Hence $G^{(M)}\in\calG_{\beta,L}$ with exactly the same margin constant.

The positive mass of $G^{(M)}$ is $m_+/M$, so \eqref{eq:side-mass-rounding} gives
\[
|\theta(G^{(M)})-\theta(G)|
=
2\left|\frac{m_+}{M}-p_+\right|
\le M^{-1}.
\]
Equation~\eqref{eq:outward-quantile-subdistributions} shows directly that,
for every $0<t<1$,
\[
0\le F_+(t)-G^{(M)}((0,t])<\frac1M,
\qquad
0\le F_-(t)-G^{(M)}([-t,0))<\frac1M.
\]
At $t=1$, the corresponding errors are at most $1/(2M)$ by
\eqref{eq:side-mass-rounding}.  Thus the one-sided distribution functions
of $G^{(M)}$ converge to those of $G$, which implies
$G^{(M)}\Rightarrow G$.  Each coordinate of $K_nG^{(M)}$ is the
integral of a bounded continuous polynomial in $z$, so weak convergence
implies that $K_nG^{(M)}$ converges coordinatewise to $K_nG$.  Since the
output alphabet is finite, this implies
\[
\TV(K_nG^{(M)},K_nG)\to0.
\]
Finally, if $H$ is the reflection of $G$, define $H^{(M)}$ to be the reflection of $G^{(M)}$.
\end{proof}

\subsection{Random-permutation coupling}

\begin{lemma}[Random-permutation embedding]
\label{lem:permutation-embedding}
Let $G_M=M^{-1}\sum_{r=1}^M\delta_{z_r}$ be an equal-weight measure.  Partition $[0,1]$ into $M$ intervals $I_1,\dots,I_M$ of equal length.  Draw a uniform random permutation $\Pi$ of $[M]$ and define the deterministic map
\[
Z_\Pi(q)=z_{\Pi(r)},
\qquad q\in I_r.
\]
Let $\overline{\mathbb P}^{Q,n}_{G_M,M}$ be the law of the full sample in \eqref{eq:full-observed-scalar} after averaging over $\Pi$.  Then
\begin{equation}
\label{eq:permutation-TV}
\TV\!\left(
\overline{\mathbb P}^{Q,n}_{G_M,M},
\lambda^{\otimes Q}\otimes(K_nG_M)^{\otimes Q}
\right)
\le
3\alpha_{Q,M},
\qquad
\alpha_{Q,M}
:=
1-\frac{(M)_Q}{M^Q}
\le
\frac{Q(Q-1)}{2M},
\end{equation}
where $(M)_Q=M(M-1)\cdots(M-Q+1)$, with $(M)_Q=0$ when $Q>M$.
\end{lemma}

\begin{proof}[Proof of Lemma~\ref{lem:permutation-embedding}]
If $Q>M$, the right side of \eqref{eq:permutation-TV} is at least one and the claim is trivial.  Assume henceforth that $Q\le M$.

We compare two joint laws on the common augmented variables
$(q,R,J,S)=((q_j,R_j,J_j,S_j))_{j=1}^Q$.  Under the
\emph{permutation law}, first draw $R_1,\dots,R_Q$ iid uniformly from
$[M]$ and, conditional on $R$, draw the $q_j$ independently and uniformly
within $I_{R_j}$.  Independently draw a uniform permutation $\Pi$ of
$[M]$, set $J_j=\Pi(R_j)$, and then draw the counts conditionally
independently according to
\[
S_j\mid J_j
\sim
\Bin\!\left(n,\frac{1+z_{J_j}}2\right).
\]
This construction generates $q_1,\dots,q_Q\overset{\mathrm{iid}}{\sim}
\lambda$, and its $(q,S)$-marginal is
$\overline{\mathbb P}^{Q,n}_{G_M,M}$.

Under the \emph{iid law}, generate $(R,q)$ in the same way, but draw
$J_1,\dots,J_Q$ iid uniformly from $[M]$, independently of $(R,q)$, and
then use the same conditional distribution of $S$ given $J$.  Its
$(q,S)$-marginal is
$\lambda^{\otimes Q}\otimes(K_nG_M)^{\otimes Q}$.  Since projection onto
$(q,S)$ cannot increase total variation, it is enough to compare these
two augmented laws.

Let
\[
\mathcal D_{Q,M}
:=
\{(r_1,\dots,r_Q)\in[M]^Q:r_1,\dots,r_Q\text{ are distinct}\}.
\]
The uniform law on $\mathcal D_{Q,M}$ is exactly the law obtained by
drawing $Q$ labels uniformly without replacement from $[M]$, with the
$j$th coordinate recording the label drawn for observation $j$.

Let $A:=\{R\in\mathcal D_{Q,M}\}$ under the permutation law.  Conditional
on $A$, the vector $R$ is uniform on $\mathcal D_{Q,M}$.  For every fixed
$r\in\mathcal D_{Q,M}$, the vector
$\Pi(r):=(\Pi(r_1),\dots,\Pi(r_Q))$ is also uniform on
$\mathcal D_{Q,M}$, and its distribution does not depend on $r$.
Consequently, under the permutation law conditional on $A$, the vectors
$R$ and $J$ are independent and each is uniform on
$\mathcal D_{Q,M}$.

Under the iid law, let
$E:=\{R\in\mathcal D_{Q,M},\ J\in\mathcal D_{Q,M}\}$.  Since $R$ and $J$
are independent and each has iid uniform coordinates before conditioning,
conditional on $E$ they remain independent and each is uniform on
$\mathcal D_{Q,M}$.  Given $(R,J)$, both laws draw $q_j$ uniformly within
$I_{R_j}$ and generate $S_j$ through the same binomial channel indexed by
$J_j$.  Hence their conditional laws agree exactly:
\[
\mathcal L_{\mathrm{perm}}(q,R,J,S\mid A)
=
\mathcal L_{\mathrm{iid}}(q,R,J,S\mid E).
\]

The probability of a collision in one iid sample of $Q$ indices is
\[
\alpha_{Q,M}
=
1-\frac{(M)_Q}{M^Q}
\le
\binom Q2\frac1M.
\]
Using the elementary inequality $\TV(P,P(\cdot\mid B))\le P(B^c)$ and the triangle inequality,
\begin{align*}
\TV(\mathcal L_{\mathrm{perm}}(q,R,J,S),
     \mathcal L_{\mathrm{iid}}(q,R,J,S))
&\le
\pr_{\mathrm{perm}}(A^c)
+
\pr_{\mathrm{iid}}(E^c)\\
&=
\alpha_{Q,M}
+
\{1-(1-\alpha_{Q,M})^2\}\\
&\le3\alpha_{Q,M}.
\end{align*}
Projecting both augmented laws onto $(q,S)$ now proves
\eqref{eq:permutation-TV}.
\end{proof}

\subsection{Transfer theorem and general-ensemble embedding}

\begin{proof}[Proof of Theorem~\ref{thm:observed-transfer}]
The upper inequality is \eqref{eq:obs-le-count}.  We prove the lower inequality.

We first record the Bayes reduction that guides the construction.  Suppose
$\nu_+$ and $\nu_-$ are finitely supported priors on
$\mathfrak Z_{\beta,L}$ such that, for some $b>0$,
$\theta_Z=\pm b$ for every $Z\in\operatorname{supp}(\nu_\pm)$,
respectively.  Let $\overline{\mathbb P}_+$ and
$\overline{\mathbb P}_-$ denote the corresponding prior-averaged laws of
the full observed sample.  Then every decision rule $\widehat\sigma$
satisfies
\begin{align}
\sup_{Z\in\mathfrak Z_{\beta,L}}
\E_Z\mathcal L_Z(\widehat\sigma)
&\ge
\frac b2
\left\{
\overline{\mathbb P}_+(\widehat\sigma=-1)
+
\overline{\mathbb P}_-(\widehat\sigma=+1)
\right\}\notag\\
&\ge
\frac b2
\left\{1-\TV(\overline{\mathbb P}_+,\overline{\mathbb P}_-)\right\}.
\label{eq:transfer-bayes-reduction}
\end{align}
The first inequality bounds the maximum risk below by its average under
the prior that first chooses the sign uniformly and then draws $Z$ from
the corresponding $\nu_\pm$; the second is the usual two-point
testing inequality.  It therefore suffices to construct two such priors
with $b$ comparable to the modulus and whose prior-averaged data laws
have total variation bounded away from one.

Fix $\varepsilon_Q:=(64\sqrt Q)^{-1}$ and take any
$0<\Delta<\omega_{n,\beta,L}(\varepsilon_Q)$.  By the reflection
symmetrization used in the proof of Theorem~\ref{thm:scalar-modulus},
there exist $G_+,G_-\in\calG_{\beta,L}$ such that
\begin{equation}
\label{eq:transfer-hard-pair}
G_-=G_+^\dagger,
\qquad
\theta(G_+)=a,
\qquad
\theta(G_-)=-a,
\qquad
a>\Delta/2,
\qquad
h(K_nG_+,K_nG_-)\le\varepsilon_Q,
\end{equation}
where $\dagger$ denotes reflection under $z\mapsto-z$.

Apply Lemma~\ref{lem:equal-weight-discretization} to $G_+$ and reflect the
result to obtain equal-weight measures $G^{(M)}_+$ and $G^{(M)}_-$.  Write
\[
P^{(M)}_\pm:=K_nG^{(M)}_\pm,
\qquad
a_M:=\theta(G^{(M)}_+).
\]
Because the targets and channel laws converge, and because $M$ may be
taken arbitrarily large, choose $M$ so that
\begin{equation}
\label{eq:transfer-M-choices}
a_M\ge\Delta/4,
\qquad
h(P^{(M)}_+,K_nG_+)\le\varepsilon_Q,
\qquad
6\alpha_{Q,M}\le\frac18.
\end{equation}
Reflection gives the analogous channel bound for the negative hypothesis
and $\theta(G^{(M)}_-)=-\theta(G^{(M)}_+)$.  The triangle inequality gives
\[
h(P^{(M)}_+,P^{(M)}_-)
\le3\varepsilon_Q.
\]
Consequently, the product-Hellinger and total-variation inequalities
\eqref{eq:product-hellinger} and \eqref{eq:TV-le-hellinger} give
\begin{equation}
\label{eq:transfer-product-TV}
\TV\!\left(
(P^{(M)}_+)^{\otimes Q},
(P^{(M)}_-)^{\otimes Q}
\right)
\le
\frac{3}{64}.
\end{equation}

Each discretized alternative is now lifted to a family of deterministic
question-to-advantage maps.  For each sign, let $\nu_\pm$ be the uniform
permutation prior from Lemma~\ref{lem:permutation-embedding}, constructed
using $G^{(M)}_\pm$.  Every map in the support of $\nu_\pm$ belongs to
$\mathfrak Z_{\beta,L}$ and has target $\pm a_M$.  Denote the corresponding
prior-averaged full-data laws by $\overline{\mathbb P}_+$ and
$\overline{\mathbb P}_-$.  The permutation lemma gives
\[
\TV\!\left(
\overline{\mathbb P}_\pm,
\lambda^{\otimes Q}\otimes(P^{(M)}_\pm)^{\otimes Q}
\right)
\le3\alpha_{Q,M}.
\]
Since adjoining the common factor $\lambda^{\otimes Q}$ does not change
total variation, the triangle inequality and
\eqref{eq:transfer-product-TV} yield
\[
\TV(\overline{\mathbb P}_+,\overline{\mathbb P}_-)
\le
6\alpha_{Q,M}+\frac{3}{64}
<\frac12.
\]

Applying \eqref{eq:transfer-bayes-reduction} with $b=a_M$ and then taking
the infimum over $\widehat\sigma$ gives
\[
\calR^{\mathrm{obs}}_{Q,n}(\beta,L)
\ge
\frac{a_M}{4}
\ge
\frac{\Delta}{16}.
\]
Since this holds for every
$0<\Delta<\omega_{n,\beta,L}(\varepsilon_Q)$,
this proves \eqref{eq:observed-transfer} with the stated universal constants.
\end{proof}

\begin{proof}[Proof of Corollary~\ref{cor:general-lower-transfer}]
In this submodel the clean margin of $w$ is
\[
M_q(w)=(w_1-w_2)Z(q).
\]
Put $p_+:=\Pp\{Z>0\}$ and $p_-:=\Pp\{Z<0\}$.  The margin condition
implies that $Z$ has no mass at zero, so $p_++p_-=1$ and
$\theta_Z=p_+-p_-$.  Accordingly, the risk of $w$ is $p_-$ when
$w_1>w_2$, $p_+$ when $w_1<w_2$, and one when $w_1=w_2$.  Hence
$\inf_{w\in\Delta_K}R_Z(w)=(1-|\theta_Z|)/2$, attained at $e_1$ when
$\theta_Z>0$ and at $e_2$ when $\theta_Z<0$.  At either choice,
\[
\Pp\{0<M_q(w^\star)\le t\}
\le
(Z_\#\lambda)([-t,t])
\le Lt^\beta.
\]
Thus the embedded environments satisfy the one-sided margin condition at
$w^\star$ used in the general theorem.

Given any estimator $\widehat w$, define the induced scalar decision
$\widehat\sigma$ to be $+1$ when $\widehat w_1\ge\widehat w_2$ and $-1$
otherwise.  If $\widehat w_1\ne\widehat w_2$, its clean excess risk equals
$|\theta_Z|$ when $\widehat\sigma$ chooses the wrong side and is zero
otherwise.  If $\widehat w_1=\widehat w_2$, the ensemble margin is zero
on every question, so $R_Z(\widehat w)=1$ and its excess risk is
$(1+|\theta_Z|)/2\ge|\theta_Z|$.
Therefore, pointwise in the data,
\[
R_Z(\widehat w)-\inf_wR_Z(w)
\ge
|\theta_Z|\one\{\widehat\sigma\theta_Z<0\}.
\]
For models 1 and 2, $S_j$ is sufficient for their raw generations; the
generations of models $i\ge3$ have a law independent of $Z$.  Thus the
induced $\widehat\sigma$ is statistically equivalent to a possibly
randomized rule based on $((q_j,S_j))_{j=1}^Q$, and such randomization is
already allowed in \eqref{eq:minimax-observed}.  Taking expectations,
suprema, and infima proves the first inequality; the second follows from
Corollary~\ref{cor:observed-modulus}.
\end{proof}

\section{Proof of the identification floor}
\label{app:identification}

Recall that $\Pi_n$ denotes the space of real algebraic polynomials of degree at most $n$.

\subsection{Localized polynomial approximation and the upper bound}

\begin{lemma}[Localized Jackson polynomial]
\label{lem:jackson}
For every $s>0$, there is $C_s$ such that, for every integer $n\ge1$, there is $P_n\in\Pi_n$ satisfying
\begin{equation}
\label{eq:jackson-bound}
\norm{P_n}_{L_\infty[-1,1]}\le1,
\qquad
|P_n(x)-\sgn(x)|
\le
C_s\min\{1,(n|x|)^{-s}\}
\end{equation}
for every $x\in[-1,1]\setminus\{0\}$.
\end{lemma}

\begin{proof}
The key idea is to smooth the discontinuous sign function with a positive, normalized, finite-frequency Jackson kernel.  The resulting convolution is a bounded trigonometric polynomial that closely follows the sign function away from its jumps; evenness then converts it into an algebraic polynomial in $x=\cos\vartheta$.

Let $f(\vartheta):=\sgn(\cos\vartheta)$ on the circle $\R/(2\pi\mathbb Z)$, and choose an integer $r$ such that $2r-1>s$.  For $m\ge2$, define
\[
J_{m,r}(u)
:=c_{m,r}
\left(\frac{\sin(mu/2)}{\sin(u/2)}\right)^{2r},
\]
where the quotient is interpreted continuously at zero.  The constant $c_{m,r}$ is chosen to give normalization $\frac1{2\pi} \int_{-\pi}^{\pi} J_{m,r}(u)\,\dd u=1$.

\emph{Kernel properties.}
Recall that the Fej\'er kernel satisfies
\[
F_m(u)
:=\frac1m\left(\frac{\sin(mu/2)}{\sin(u/2)}\right)^2
=
\sum_{k=-(m-1)}^{m-1}
\left(1-\frac{|k|}{m}\right)e^{iku}.
\]
Since $J_{m,r}=c_{m,r}m^rF_m^r$, it is nonnegative, even, and a trigonometric polynomial of degree $d:=r(m-1)$.

To determine its normalization scale, put $K_{m,r}^0(u):=(\sin(mu/2)/\sin(u/2))^{2r}$.  On $|u|\le1/m$, one has $K_{m,r}^0(u)\asymp_r m^{2r}$, while for $0<|u|\le\pi$, the inequalities $|\sin(u/2)|\ge|u|/\pi$ and $|\sin(mu/2)|\le1$ give $K_{m,r}^0(u)\lesssim_r|u|^{-2r}$.  Splitting the integral at $1/m$ therefore yields
\[
\frac1{2\pi}\int_{-\pi}^{\pi}K_{m,r}^0(u)\,\dd u
\asymp_r m^{2r-1},
\qquad
c_{m,r}\asymp_r m^{-(2r-1)}.
\]
Only the resulting upper bound on $c_{m,r}$ is needed below.  We next bound the tail.  If $a\le1/m$, normalized total mass bounds it by one; if $a>1/m$, the preceding pointwise and normalization bounds give
\[
\frac1{2\pi}\int_{|u|\ge a}J_{m,r}(u)\,\dd u
\lesssim_r
m^{-(2r-1)}\int_a^\pi u^{-2r}\,\dd u
\lesssim_r
(ma)^{-(2r-1)}.
\]
Thus, for every $0<a\le\pi$,
\begin{equation}
\label{eq:jackson-tail}
\frac1{2\pi}\int_{|u|\ge a}J_{m,r}(u)\,\dd u
\le
C_r\min\{1,(ma)^{-(2r-1)}\},
\qquad 0<a\le\pi.
\end{equation}

\emph{Finite-frequency smoothing.}
Define
\[
T_m(\vartheta)
:=
\frac1{2\pi}\int_{-\pi}^{\pi}
f(\vartheta-u)J_{m,r}(u)\,\dd u.
\]
Since $J_{m,r}\ge0$ and has normalized mass one, $T_m(\vartheta)$ is a weighted average of values in $[-1,1]$, so $|T_m(\vartheta)|\le1$.  Write $J_{m,r}(u)=\sum_{|k|\le d}j_ke^{iku}$.  The convolution identity gives
\[
T_m(\vartheta)
=
\sum_{|k|\le d}j_k\widehat f(k)e^{ik\vartheta},
\]
where $\widehat f(k)$ is the $k$th Fourier coefficient of $f$.  Thus $T_m$ is a trigonometric polynomial of degree at most $d$.  Because both $f$ and $J_{m,r}$ are even, so is $T_m$, and hence it has a cosine expansion $T_m(\vartheta)=a_0+\sum_{k=1}^da_k\cos(k\vartheta)$.  If $\mathsf T_k$ denotes the $k$th Chebyshev polynomial, define $P(x):=a_0+\sum_{k=1}^da_k\mathsf T_k(x)$.  Then $P$ has degree at most $d$ and
\[
T_m(\vartheta)=P(\cos\vartheta),
\qquad
\norm P_{L_\infty[-1,1]}\le1.
\]

\emph{Localized approximation error.}
Let $a(\vartheta)$ denote the circular distance from $\vartheta$ to the nearest discontinuity of $f$, namely $\pi/2$ or $3\pi/2$.  Whenever $|u|<a(\vartheta)$, the values $f(\vartheta-u)$ and $f(\vartheta)$ agree.  Therefore \eqref{eq:jackson-tail} implies
\[
\begin{aligned}
|T_m(\vartheta)-f(\vartheta)|
&\le
\frac2{2\pi}\int_{|u|\ge a(\vartheta)}J_{m,r}(u)\,\dd u\\
&\le
C_r\min\{1,(m a(\vartheta))^{-(2r-1)}\}.
\end{aligned}
\]
If $x=\cos\vartheta$, then $|x|=\sin a(\vartheta)\le a(\vartheta)$.  For $n\ge2r$, take $m=1+\lfloor n/r\rfloor$.  Then $d=r(m-1)\le n$ and $m\ge n/r$.  Since $2r-1>s$, the last display implies \eqref{eq:jackson-bound}, after enlarging the constant.  For the finitely many integers $1\le n<2r$, take $P_n\equiv0$ and enlarge $C_s$.
\end{proof}

The estimate in Lemma~\ref{lem:jackson} is intentionally nonuniform: it gives only a constant bound in the boundary layer $|x|\lesssim n^{-1}$, but outside that layer the error decays at any prescribed fixed polynomial order in $n|x|$.

\begin{lemma}[Integrated approximation under a margin condition]
\label{lem:integrated-approx}
If $G\in\calG_{\beta,L}$ and $s>\beta$ in \cref{lem:jackson}, then, for every $n\ge1$,
\begin{equation}
\label{eq:integrated-approx}
\int|P_n(z)-\sgn(z)|\,\dd G(z)
\le C_{\beta,L,s}n^{-\beta}.
\end{equation}
\end{lemma}

\begin{proof}
For $|z|\le n^{-1}$,
 the integrand is bounded by two and the margin condition gives probability at most $Ln^{-\beta}$.  On the dyadic shell
\[
\frac{2^j}{n}<
|z|
\le \min\left\{\frac{2^{j+1}}n,1\right\},
\]
the pointwise error is at most $C_s2^{-js}$ and the shell probability is at most $L(2^{j+1}/n)^\beta$ until the upper endpoint reaches one.  Consequently,
\[
\int|P_n-\sgn|\,\dd G
\le
Ln^{-\beta}
\biggl[2+C_s2^\beta\sum_{j\ge0}2^{j(\beta-s)}\biggr]
\le
C_{\beta,L,s}n^{-\beta},
\]
because $s>\beta$, proving \eqref{eq:integrated-approx}.
\end{proof}

\begin{proof}[Proof of the upper bound in \cref{thm:identification-floor}]
Suppose $K_nG=K_nH$.  By \cref{lem:moment-equivalence}, $G$ and $H$ integrate every polynomial of degree at most $n$ equally.  Fix $s=\beta+1$ and let $P_n$ be the corresponding polynomial from \cref{lem:jackson}.  Then
\begin{align*}
|\theta(G)-\theta(H)|
&=
\left|
\int(\sgn-P_n)\,\dd G
-
\int(\sgn-P_n)\,\dd H
\right|\\
&\le
\int|\sgn-P_n|\,\dd G
+
\int|\sgn-P_n|\,\dd H \;\le\;
C_{\beta,L}n^{-\beta}
\end{align*}
by \cref{lem:integrated-approx}.  Taking the supremum over all pairs with $K_nG=K_nH$ proves the upper bound on $\omega_{n,\beta,L}(0)$.
\end{proof}

\subsection{Weighted \texorpdfstring{$L_1$}{L1} approximation of the jump}

Let
\[
\mu_\beta(\dd x)
=
\frac\beta2|x|^{\beta-1}\one_{[-1,1]}(x)\,\dd x.
\]
The next lemma gives the matching approximation rate.  Its upper bound follows from the localized Jackson construction above; the lower bound uses a weighted inverse theorem of \citet{Kopotun}.

\begin{lemma}[Weighted $L_1$ approximation of a jump]
\label{lem:weighted-sign}
There are $0<c_\beta<C_\beta<\infty$ such that
\begin{equation}
\label{eq:weighted-sign-rate}
c_\beta n^{-\beta}
\le
E_n(\beta)
:=
\inf_{P\in\Pi_n}
\int|\sgn(x)-P(x)|\,\dd\mu_\beta(x)
\le
C_\beta n^{-\beta}.
\end{equation}
\end{lemma}

\begin{proof}
The upper bound follows from \cref{lem:jackson,lem:integrated-approx}, because $\mu_\beta([-t,t])=t^\beta$.

For the lower bound, we use only the inverse theorem of \citet[Theorem~9.1]{Kopotun}.  The power weight
$w_\beta(x):=|x|^{\beta-1}$
is the generalized-Jacobi weight with one designated point $z_1=0$ and exponent $\beta-1>-1$, and belongs to the class treated there.  Fix an integer $r>\beta$, which will be the order of the modulus, and a sufficiently large constant $A$.  Specializing Kopotun's theorem to $p=1$, and absorbing the normalization of $\mu_\beta$ into the constant, gives
\begin{equation}
\label{eq:kopotun-inverse}
\omega_\varphi^r(f,A,m^{-1})_{1,w_\beta}
\le
C m^{-r}\sum_{k=1}^m k^{r-1}E_k(f)_{1,w_\beta},
\end{equation}
Here $\norm{g}_{1,w_\beta}:=\int_{-1}^1|g(x)|w_\beta(x)\,\dd x$, $E_k(f)_{1,w_\beta}:=\inf_{P\in\Pi_k}\norm{f-P}_{1,w_\beta}$, and $\omega_\varphi^r(f,A,m^{-1})_{1,w_\beta}$ is the complete weighted modulus of smoothness defined in \citet[Section~3]{Kopotun}.  The normalization of $\mu_\beta$ gives $E_k(\sgn)_{1,w_\beta}=(2/\beta)E_k(\beta)$.  Kopotun uses $\Pi_k$ for polynomials of degree less than $k$, whereas our $\Pi_k$ contains polynomials of degree at most $k$; this one-degree shift only affects the constants.

For $f=\sgn$, the definition of the complete modulus simplifies substantially.  Its main-part finite-difference term is evaluated away from a neighborhood of zero and vanishes because $\sgn$ is constant on each remaining component.  Only the local best-approximation term remains:
\[
\omega_\varphi^r(\sgn,A,t)_{1,w_\beta}
=
\inf_{P\in\Pi_{r-1}}
\int_{-b_t}^{b_t}|\sgn(x)-P(x)|\,|x|^{\beta-1}\,\dd x,
\qquad
b_t:=2A\rho(t,0)=2A(t+t^2),
\]
for all sufficiently small $t$. The zero polynomial gives an upper bound of order $t^\beta$.  For the reverse bound, the interval $[-b_t,b_t]$ contains $[-2At,2At]$, and the substitution $x=tu$ gives
\[
\inf_{P\in\Pi_{r-1}}
\int_{-2At}^{2At}|\sgn(x)-P(x)|\,|x|^{\beta-1}\,\dd x
=
t^\beta
\inf_{Q\in\Pi_{r-1}}
\int_{-2A}^{2A}|\sgn(u)-Q(u)|\,|u|^{\beta-1}\,\dd u.
\]
The last infimum is strictly positive: $\Pi_{r-1}$ is a finite-dimensional closed subspace of the weighted $L_1$ space, whereas $\sgn$ is not equal almost everywhere to a polynomial.  We have therefore shown that
\begin{equation}
\label{eq:sign-complete-modulus}
\omega_\varphi^r(\sgn,A,t)_{1,w_\beta}
\asymp
t^\beta,
\qquad 0<t\le t_\beta.
\end{equation}
Together, \eqref{eq:kopotun-inverse} and \eqref{eq:sign-complete-modulus} give a lower bound of order $m^{-\beta}$ on a weighted average of $E_1(\beta),\ldots,E_m(\beta)$ for all sufficiently large $m$.  It remains to extract a pointwise lower bound for $E_n(\beta)$.

Write $E_k:=E_k(\beta)$.  By the upper bound already established, $E_k\lesssim k^{-\beta}$.  Apply \eqref{eq:kopotun-inverse} with $m=Mn$, where $M\ge2$ is a fixed integer to be chosen.  By \eqref{eq:sign-complete-modulus},
\[
c(Mn)^{-\beta}
\le
C(Mn)^{-r}\sum_{k=1}^{Mn}k^{r-1}E_k.
\]
Split the sum at $n$.  Using $E_k\lesssim k^{-\beta}$ on the first part and monotonicity $E_k\le E_n$ for $k\ge n$ on the second gives
\[
\sum_{k=1}^{n}k^{r-1}E_k
\lesssim n^{r-\beta},
\qquad
\sum_{k=n+1}^{Mn}k^{r-1}E_k
\lesssim (Mn)^rE_n.
\]
Consequently,
\[
cM^{-\beta}n^{-\beta}
\le
C_3M^{-r}n^{-\beta}+C_4E_n.
\]
Since $r>\beta$, choose the fixed $M$ large enough that $C_3M^{-r}\le(c/2)M^{-\beta}$.  The first term on the right can then be absorbed into the left,
giving
$E_n\gtrsim n^{-\beta}$.  Decreasing the implicit constant handles the finitely many small values of $n$ not covered by \eqref{eq:sign-complete-modulus}.
\end{proof}

\subsection{Duality and moment-matched alternatives}

\begin{lemma}[$L_1$ duality produces a moment-canceling family]
\label{lem:duality-moment-pair}
For every integer $m\ge1$, there is an odd function $h_m\in L_\infty(\mu_\beta)$ satisfying
\begin{equation}
\label{eq:hm-properties}
\norm{h_m}_\infty\le1,
\qquad
\int P(z)h_m(z)\,\dd\mu_\beta(z)=0
\quad\text{for every }P\in\Pi_m,
\end{equation}
and
\begin{equation}
\label{eq:hm-target}
\int\sgn(z)h_m(z)\,\dd\mu_\beta(z)
=E_m(\beta).
\end{equation}
Define the raw alternatives $G_{m,+},G_{m,-}$ by
\begin{equation}
\label{eq:raw-dual-pair}
\dd G_{m,+}:=(1+h_m)\dd\mu_\beta,
\qquad
\dd G_{m,-}:=(1-h_m)\dd\mu_\beta.
\end{equation}
They are probability measures on $[-1,1]$ such that:
\begin{enumerate}[label=(\alph*)]
\item their moments agree through degree $m$;
\item $G_{m,\pm}([-t,t])\le2t^\beta$ for every $0<t\le1$;
\item $G_{m,-}=G_{m,+}^\dagger$ and
$
\theta(G_{m,+})
=-\theta(G_{m,-})
=E_m(\beta).
$
\end{enumerate}
Moreover, for every $L\ge1$, let
\begin{equation}
\label{eq:adjusted-dual-pair}
\rho_L:=\min\{1,L/2\},
\qquad
H_0:=\frac12(\delta_{-1}+\delta_{+1}),
\qquad
H_{m,\pm}:=\rho_LG_{m,\pm}+(1-\rho_L)H_0.
\end{equation}
Then $H_{m,\pm}\in\calG_{\beta,L}$, their moments agree through degree $m$, $H_{m,-}=H_{m,+}^\dagger$, and
\begin{equation}
\label{eq:adjusted-dual-target}
\theta(H_{m,+})
=-\theta(H_{m,-})
=\rho_LE_m(\beta).
\end{equation}
\end{lemma}

\begin{proof}
The polynomial-approximation error admits a dual certificate $h_m$ that annihilates every polynomial in $\Pi_m$ while being maximally correlated with $\sgn$.  Perturbing the reference measure in the directions $\pm h_m$ will then produce moment-matched distributions with oppositely separated targets.

Let $f:=\sgn$ and use the dual pairing
\[
\langle g,h\rangle
:=
\int g(x)h(x)\,\dd\mu_\beta(x),
\qquad
g\in L_1(\mu_\beta),\quad h\in L_\infty(\mu_\beta).
\]
Since $E_m(\beta)$ is the $L_1$ distance from $f$ to the subspace $\Pi_m$, it has the following dual representation; see \cref{rem:l1-duality}.
\begin{equation}
\label{eq:l1-duality}
E_m(\beta)
=
\max_{\substack{h\in L_\infty(\mu_\beta),\ \norm h_\infty\le1\\
\langle P,h\rangle=0\ \text{for every }P\in\Pi_m}}
\langle f,h\rangle.
\end{equation}
The maximum is attained by weak-$^*$ compactness.  We may choose a maximizer that is odd.  Indeed, if $h$ is feasible, then
\[
h_{\mathrm{odd}}(x):=\frac{h(x)-h(-x)}2
\]
is feasible, has $L_\infty$ norm at most one, and has the same objective value because $\mu_\beta$ is symmetric, $\Pi_m$ is reflection invariant, and $\sgn$ is odd.

Let $h_m$ be such an odd maximizer and define $G_{m,\pm}$ by \eqref{eq:raw-dual-pair}.  These measures are nonnegative because $\norm{h_m}_\infty\le1$.  Orthogonality to the constant polynomial makes their total masses one; orthogonality to $x^j$, $j=0,\dots,m$, makes their moments equal.  Moreover,
\[
G_{m,\pm}([-t,t])
\le
2\mu_\beta([-t,t])
=
2t^\beta.
\]
Oddness gives $G_{m,-}=G_{m,+}^\dagger$.  Finally, symmetry of $\mu_\beta$ gives $\int\sgn\,\dd\mu_\beta=0$, so \eqref{eq:l1-duality} gives
\[
\theta(G_{m,+})
=
\int\sgn(x)h_m(x)\,\dd\mu_\beta(x)
=
E_m(\beta),
\]
and reflection gives the target of $G_{m,-}$.

It remains to verify the claims for the adjusted pair.  For every $0<t<1$, the outer atoms do not contribute and
\[
H_{m,\pm}([-t,t])
\le
2\rho_Lt^\beta
\le
Lt^\beta;
\]
at $t=1$, the required inequality is $1\le L$.  Thus $H_{m,\pm}\in\calG_{\beta,L}$.  Adding the same symmetric measure to the two raw alternatives preserves their moment matching, and symmetry of $H_0$ together with reflection of the raw pair gives $H_{m,-}=H_{m,+}^\dagger$.  Finally, $\theta(H_0)=0$, so
\[
\theta(H_{m,+})
=-\theta(H_{m,-})
=\rho_LE_m(\beta).
\]
\end{proof}

\begin{remark}[Duality interpretation]
\label{rem:l1-duality}
Identity~\eqref{eq:l1-duality} is the infinite-dimensional analogue of linear-programming duality.  Indeed, put $E:=\inf_{P\in\Pi_m}\norm{f-P}_{L_1(\mu_\beta)}$.  For every feasible $h$ and every $P\in\Pi_m$,
\[
\langle f,h\rangle
=
\langle f-P,h\rangle
\le
\norm{f-P}_{L_1(\mu_\beta)},
\]
so the dual value is at most $E$.  Hahn--Banach supplies a norm-one functional on $L_1(\mu_\beta)$ that annihilates $\Pi_m$ and takes the value $E$ at $f$.  Identifying $L_1(\mu_\beta)^*$ with $L_\infty(\mu_\beta)$ represents this functional by a feasible $h$ attaining equality, and hence establishes strong duality.
\end{remark}

\begin{proof}[Proof of the lower bound in \cref{thm:identification-floor}]
Let $H_{n,+},H_{n,-}$ be the adjusted degree-$n$ alternatives supplied by \cref{lem:duality-moment-pair}.  Their targets are opposite with
\[
\theta(H_{n,+})
=
-\theta(H_{n,-})
=
\rho_LE_n(\beta)
\ge
c_{\beta,L}n^{-\beta}.
\]
By \cref{lem:moment-equivalence}, $K_nH_{n,+}=K_nH_{n,-}$, so this pair proves the lower bound on $\omega_{n,\beta,L}(0)$.
\end{proof}

\begin{remark}[Direct minimax consequence]
The same pair also gives a direct proof of~\eqref{eq:N-only-lower} by a two-point argument.  Let
\[
\Pp_\pm
:=
(K_nH_{n,\pm})^{\otimes Q}
\]
denote the laws of the full $Q$-question experiment.  Moment equivalence gives $\Pp_+=\Pp_-$.  Hence, for every decision rule $\widehat\sigma$,
\begin{align*}
\sup_{G\in\calG_{\beta,L}}
\E_G\mathcal L_G(\widehat\sigma)
&\ge
\frac{\rho_LE_n(\beta)}2
\left[
\Pp_+(\widehat\sigma=-1)
+
\Pp_-(\widehat\sigma=+1)
\right]\\
&=
\frac{\rho_LE_n(\beta)}2.
\end{align*}
The equality uses $\Pp_+=\Pp_-$ and $\widehat\sigma\in\{-1,+1\}$.  The prefactor $\rho_LE_n(\beta)$ is the cost of a wrong decision under either alternative, since $|\theta(H_{n,\pm})|=\rho_LE_n(\beta)$.  Taking the infimum over $\widehat\sigma$ gives $\calR^{\mathrm{dec}}_{Q,n}(\beta,L)\gtrsim_{\beta,L}n^{-\beta}$, independently of $Q$.
\end{remark}

\section{Coefficient bounds and the Durrmeyer calculation}
\label{app:coefficients}

\subsection{Proof of the explicit exponential coefficient bound}

\begin{proof}[Proof of \cref{lem:B0}]
Expand $P(z)=c_0+\sum_{j=1}^dc_jT_j(z)$ in Chebyshev polynomials.  The cosine coefficient formula and $\norm P_\infty\le1$ give $|c_0|\le1$ and $|c_j|\le2$ for $j\ge1$.  Let $L_j$ be the sum of the absolute values of the monomial coefficients of $T_j$.  The recurrence $T_{j+1}(z)=2zT_j(z)-T_{j-1}(z)$ gives $L_{j+1}\le2L_j+L_{j-1}$; since $L_0=L_1=1$, it follows that $L_j\le C(1+\sqrt2)^j$.  After substituting $z=2p-1$, each monomial $z^\ell$ has absolute coefficient sum $3^\ell$, so the corresponding sum for $T_j(2p-1)$ is at most $C\{3(1+\sqrt2)\}^j$.

Write $P(2p-1)=\sum_{j=0}^d r_jp^j$, so the $r_j$ are its coefficients in the ordinary power basis.  The preceding bounds and a geometric sum give $\sum_{j=0}^d|r_j|\le C_0B_0^d$, with $B_0=7.5>3(1+\sqrt2)$.

For any $n\ge d$, the identity $p^j=\sum_{k=j}^n\frac{\binom kj}{\binom nj}b_{k,n}(p)$ gives the degree-$n$ Bernstein coefficients directly:
\[
a_k(P)=\sum_{j=0}^{\min(k,d)}r_j\frac{\binom kj}{\binom nj},
\qquad 0\le k\le n.
\]
Every ratio in this sum lies in $[0,1]$, so $|a_k(P)|\le\sum_{j=0}^d|r_j|\le C_0B_0^d$.
\end{proof}

\subsection{Proofs for the Durrmeyer spectrum and coefficient bound}
\label{app:durrmeyer}

\begin{proof}[Proof of \cref{prop:durrmeyer}]
The operator $D_n=T_nT_n^*$ is self-adjoint and positive.  It also preserves each polynomial space $\Pi_j$.  Indeed, beta integration gives
\[
(T_n^*u^j)(k)
=
\frac{(n+1)!}{(n+j+1)!}(k+1)(k+2)\cdots(k+j),
\]
which is a degree-$j$ polynomial in $k$.  Averaging this polynomial under $k\sim\Bin(n,p)$ shows that $D_nu^j\in\Pi_j$ and that its leading coefficient is
\[
\lambda_{j,n}
=
\frac{n!(n+1)!}{(n-j)!(n+j+1)!}.
\]
Here we use the fact that the expectation of a monic degree-$j$ polynomial in a $\Bin(n,p)$ variable has leading coefficient $n!/(n-j)!$ as a polynomial in $p$.

For $j=0$, the claim follows directly from $D_n1=1$.  For $j\ge1$, because $L_0,\dots,L_j$ are obtained by orthogonalizing $1,u,\dots,u^j$,
\[
\Pi_j
=
\Pi_{j-1}\oplus\operatorname{span}\{L_j\}.
\]
Now $D_nL_j\in\Pi_j$ by degree preservation.  For every $q\in\Pi_{j-1}$, self-adjointness and degree preservation give
\[
\inner{D_nL_j}{q}_{L_2}
=
\inner{L_j}{D_nq}_{L_2}
=0,
\]
because $D_nq\in\Pi_{j-1}$ and $L_j\perp\Pi_{j-1}$.  Hence $D_nL_j$ lies in the one-dimensional space $\Pi_j\cap\Pi_{j-1}^\perp=\operatorname{span}\{L_j\}$.  Comparing leading coefficients proves \eqref{eq:durrmeyer-eigen}.

Define
$
v_j=\lambda_{j,n}^{-1/2}T_n^*L_j.
$
Then
\[
\inner{v_i}{v_j}_{\nu_n}
=
\frac{1}{\sqrt{\lambda_{i,n}\lambda_{j,n}}}
\inner{L_i}{D_nL_j}_{L_2}
=\one\{i=j\},
\]
and
$
T_nv_j=\sqrt{\lambda_{j,n}}L_j.
$
There are $n+1$ orthonormal vectors $v_j$ in the $(n+1)$-dimensional space $\mathsf H_n$, so they form a basis.  Thus $(v_j,L_j,\sqrt{\lambda_{j,n}})$ is a singular system and \eqref{eq:durrmeyer-svd} follows.
\end{proof}

\begin{proof}[Proof of \cref{prop:coefficient-spectral}]
If $P\in\Pi_d$, write $\widetilde P=\sum_{j=0}^dc_jL_j$.  Since $T_na(P)=\widetilde P$, comparison of the Legendre coefficients in \eqref{eq:durrmeyer-svd} gives
\[
\inner{a(P)}{v_j}_{\nu_n}
=
\begin{cases}
c_j/\sqrt{\lambda_{j,n}},&0\le j\le d,\\
0,&d<j\le n.
\end{cases}
\]
Consequently,
\[
a(P)=\sum_{j=0}^d\frac{c_j}{\sqrt{\lambda_{j,n}}}v_j.
\]
Parseval's identity in $\mathsf H_n$ now gives
\[
\frac1{n+1}\sum_{k=0}^na_k(P)^2
=
\norm{a(P)}_{\nu_n}^2
=
\sum_{j=0}^d\frac{c_j^2}{\lambda_{j,n}},
\]
which proves \eqref{eq:coefficient-spectral}.  Since $j\mapsto\lambda_{j,n}$ is decreasing, Parseval in $L_2[0,1]$ also gives
\[
\norm{a(P)}_{\nu_n}
\le
\kappa_{d,n}
\left(\sum_{j=0}^dc_j^2\right)^{1/2}
=
\kappa_{d,n}\norm{\widetilde P}_{L_2[0,1]}.
\]
Finally,
\[
\max_k|a_k(P)|
\le
\sqrt{\sum_{k=0}^na_k(P)^2}
=
\sqrt{n+1}\norm{a(P)}_{\nu_n}.
\]
Applying the same bound to the maximum and minimum proves \eqref{eq:range-spectral}.
\end{proof}

\paragraph{Asymptotics of the inverse singular values.}
For $d\le n/2$,
\[
\log\kappa_{d,n}
=\frac12\sum_{r=0}^{d-1}
\log\left(1+\frac{2r+2}{n-r}\right).
\]
Using $x/(1+x)\le\log(1+x)\le x$ and $n-r\asymp n$ gives \eqref{eq:kappa-asymp}.  At $d=n$,
\[
\kappa_{n,n}^2
=\frac{(2n+1)!}{n!(n+1)!}
=\binom{2n+1}{n},
\]
and Stirling's formula gives the endpoint asymptotic.

\section{Proofs for the spectral lower bound and rate extraction}
\label{app:spectral-lower}

We prove \cref{thm:spectral-modulus-lower,prop:explicit-spectral-rates}.  The lower-bound argument combines the degree-$d$ dual alternatives from \cref{app:identification} with the Durrmeyer singular system from \cref{app:durrmeyer}.

\begin{lemma}[Adjoint tail identity for integrable perturbations]
\label{lem:adjoint-tail}
Let $g\in L_1[0,1]$ and define
\[
u_k:=\int_0^1b_{k,n}(p)g(p)\,\dd p,
\qquad k=0,\dots,n.
\]
Then
\begin{equation}
\label{eq:adjoint-tail-identity}
(n+1)\sum_{k=0}^nu_k^2
=
\sum_{j=0}^n
\lambda_{j,n}
\left(\int_0^1g(p)L_j(p)\,\dd p\right)^2.
\end{equation}
\end{lemma}

\begin{proof}
Define the coefficient vector $T_n^*g$ algebraically by
\[
(T_n^*g)(k)=(n+1)u_k.
\]
The vectors $v_0,\dots,v_n$ in \cref{app:durrmeyer} form an orthonormal basis of $(\R^{n+1},\nu_n)$ and satisfy $T_nv_j=\sqrt{\lambda_{j,n}}L_j$.  Although $g$ need not belong to $L_2$, every integral below is finite because $g\in L_1$ and $L_j$ is bounded.  Hence
\[
\langle T_n^*g,v_j\rangle_{\nu_n}
=
\int_0^1g(p)(T_nv_j)(p)\,\dd p
=
\sqrt{\lambda_{j,n}}
\int_0^1g(p)L_j(p)\,\dd p.
\]
Parseval in the finite-dimensional coefficient space gives \eqref{eq:adjoint-tail-identity}, because
\[
\|T_n^*g\|_{\nu_n}^2
=
\frac1{n+1}\sum_{k=0}^n(n+1)^2u_k^2
=(n+1)\sum_{k=0}^nu_k^2.
\]
\end{proof}

\begin{proof}[Proof of \cref{thm:spectral-modulus-lower}]
Let $h_d$, $G_{d,\pm}$, and $H_{d,\pm}$ be, respectively, the degree-$d$ dual certificate, raw alternatives, and margin-adjusted alternatives supplied by \cref{lem:duality-moment-pair}.  By \cref{eq:adjusted-dual-target,lem:weighted-sign},
\begin{equation}
\label{eq:spectral-pair-target}
|\theta(H_{d,+})-\theta(H_{d,-})|
=2\rho_LE_d(\beta)
\ge c_{\beta,L}d^{-\beta}.
\end{equation}

We now bound the Hellinger distance between their count laws.  Write $p=(1+z)/2$.  Under this transformation, $\mu_\beta$ has density
\[
w_\beta(p)=\beta|2p-1|^{\beta-1},
\qquad 0<p<1,
\]
with respect to Lebesgue measure.  Put
\[
g_d(p):=h_d(2p-1)w_\beta(p),
\]
and define
\[
q_0(k):=\int_0^1b_{k,n}(p)w_\beta(p)\,\dd p,
\qquad
u_d(k):=\int_0^1b_{k,n}(p)g_d(p)\,\dd p.
\]
Then
\[
K_nG_{d,+}=q_0+u_d,
\qquad
K_nG_{d,-}=q_0-u_d.
\]
If $r_0:=K_nH_0$ and $m_0:=\rho_Lq_0+(1-\rho_L)r_0$, the final pair has probability mass functions
\[
P_{d,+}:=K_nH_{d,+}=m_0+\rho_Lu_d,
\qquad
P_{d,-}:=K_nH_{d,-}=m_0-\rho_Lu_d.
\]
Because $m_0\ge\rho_Lq_0$ coordinatewise,
\begin{align}
 h^2(P_{d,+},P_{d,-})
&=
\sum_{k=0}^n
\left(\sqrt{m_0(k)+\rho_Lu_d(k)}-
\sqrt{m_0(k)-\rho_Lu_d(k)}\right)^2\notag\\
&\le
2\rho_L^2\sum_{k=0}^n\frac{u_d(k)^2}{m_0(k)}
\le
2\rho_L\sum_{k=0}^n\frac{u_d(k)^2}{q_0(k)}.
\label{eq:hellinger-baseline}
\end{align}
For $0<\beta\le1$, $w_\beta(p)\ge\beta$ for almost every $p\in(0,1)$, and
$
\int_0^1b_{k,n}(p)\,\dd p=\frac1{n+1}.
$
Thus
\begin{equation}
\label{eq:q0-lower}
q_0(k)\ge\frac\beta{n+1}.
\end{equation}
Combining \eqref{eq:hellinger-baseline} and \eqref{eq:q0-lower} gives
\begin{equation}
\label{eq:hellinger-adjoint}
h^2(P_{d,+},P_{d,-})
\le
\frac{2\rho_L(n+1)}\beta
\sum_{k=0}^nu_d(k)^2
=
\frac{2\rho_L}\beta
\|T_n^*g_d\|_{\nu_n}^2.
\end{equation}

The cancellation in \eqref{eq:hm-properties} removes the first $d$ singular modes.  Indeed, $L_j((1+z)/2)$ is a polynomial of degree $j$ in $z$, so
\[
\int_0^1g_d(p)L_j(p)\,\dd p=0,
\qquad j=0,\dots,d.
\]
By \cref{lem:adjoint-tail} and the monotonicity of $j\mapsto\lambda_{j,n}$,
\begin{align*}
\|T_n^*g_d\|_{\nu_n}^2
&=
\sum_{j=d+1}^n
\lambda_{j,n}
\left(\int_0^1g_d(p)L_j(p)\,\dd p\right)^2\\
&\le
\lambda_{d+1,n}
\sum_{j=d+1}^n
\left(\int_0^1g_d(p)L_j(p)\,\dd p\right)^2.
\end{align*}
Since $\|g_d\|_1\le1$ and the orthonormal shifted Legendre polynomial satisfies $\|L_j\|_\infty=\sqrt{2j+1}$,
\[
\left|\int_0^1g_d(p)L_j(p)\,\dd p\right|
\le\sqrt{2j+1}.
\]
Therefore
\begin{equation}
\label{eq:L1-tail}
\|T_n^*g_d\|_{\nu_n}^2
\le
\lambda_{d+1,n}
\sum_{j=0}^n(2j+1)
=
\lambda_{d+1,n}(n+1)^2.
\end{equation}
Equations \eqref{eq:hellinger-adjoint} and \eqref{eq:L1-tail} imply
\[
h(P_{d,+},P_{d,-})
\le
C_{\beta,L}
\frac{n+1}{\kappa_{d+1,n}}.
\]
Together with \eqref{eq:spectral-pair-target}, this proves \eqref{eq:spectral-modulus-lower}.

If $\beta>1/2$, then $g_d\in L_2[0,1]$ and
$
\|g_d\|_2^2
\le
\int_0^1w_\beta(p)^2\,\dd p
=
\frac{\beta^2}{2\beta-1}.
$
Bessel's inequality replaces \eqref{eq:L1-tail} by
\[
\|T_n^*g_d\|_{\nu_n}^2
\le
\lambda_{d+1,n}\|g_d\|_2^2
\le
C_\beta\lambda_{d+1,n}.
\]
Substitution into \eqref{eq:hellinger-adjoint} proves \eqref{eq:spectral-modulus-lower-L2}.
\end{proof}

\begin{proof}[Proof of the lower assertion in \cref{prop:explicit-spectral-rates}]
For $j\le n/2$, the lower half of \eqref{eq:kappa-asymp} gives
\[
\log\kappa_{j,n}\ge c_\kappa j(j+1)/n.
\]
Let
\[
\Lambda_{Q,n}:=\log\!\bigl(A_{\beta,L}(n+1)\sqrt Q\bigr),
\]
where $A_{\beta,L}$ is large enough to absorb the constant in \eqref{eq:spectral-modulus-lower}, and choose
\[
d+1
=
\left\lceil
\sqrt{\frac{2n\Lambda_{Q,n}}{c_\kappa}}
\right\rceil.
\]
Because $y_{Q,n}\ge1$ and $A_{\beta,L}$ is fixed, $\Lambda_{Q,n}\asymp_{\beta,L}y_{Q,n}$.  If the constant $c_0$ in \cref{prop:explicit-spectral-rates}(a) is sufficiently small and $n$ is sufficiently large, then $d+1\le n/2$ and
\[
\kappa_{d+1,n}
\ge
A_{\beta,L}(n+1)\sqrt Q.
\]
Consequently, the Hellinger radius in \eqref{eq:spectral-modulus-lower} is at most $Q^{-1/2}$.  Monotonicity of the modulus, \cref{thm:scalar-modulus}, and $d\asymp\sqrt{n\Lambda_{Q,n}}$ yield
\[
\calR^{\mathrm{dec}}_{Q,n}(\beta,L)
\gtrsim_{\beta,L}
(n\Lambda_{Q,n})^{-\beta/2}.
\]
The comparison $\Lambda_{Q,n}\asymp_{\beta,L}y_{Q,n}$ proves \eqref{eq:spectral-explicit-lower}.  When $\beta>1/2$, apply \eqref{eq:spectral-modulus-lower-L2} instead and set $\Lambda_{Q}:=\log(A_{\beta,L}\sqrt Q)$.  Then $\Lambda_Q\asymp_{\beta,L}y_Q^{(2)}$, and the same argument proves the refined assertion.
\end{proof}

\begin{proof}[Proof of the upper assertion in \cref{prop:explicit-spectral-rates}]
For $d\le n/2$, the upper half of \eqref{eq:kappa-asymp} gives $\log\kappa_{d,n}\le C_\kappa d(d+1)/n$.  Choose
\[
d=\left\lfloor a\sqrt{nx_{Q,n}}\right\rfloor\vee1
\]
with $a>0$ small enough that $C_\kappa d(d+1)/n\le x_{Q,n}/4$.  Reducing $c_0$ if necessary ensures $d\le n/2$, while $x_{Q,n}\ge1$ gives $d\asymp\sqrt{nx_{Q,n}}$.  The spectral upper bound \eqref{eq:spectral-upper} therefore yields a bias term of order $(nx_{Q,n})^{-\beta/2}$ and a stochastic term bounded by
\[
\frac{\sqrt{n+1}\,\kappa_{d,n}}{\sqrt Q}
\lesssim
\frac{\sqrt n\,e^{x_{Q,n}/4}}
{n^{(\beta+1)/2}e^{x_{Q,n}/2}}
=
n^{-\beta/2}e^{-x_{Q,n}/4}.
\]
Since $e^{-x/4}\le C_\beta x^{-\beta/2}$ for $x\ge1$, the stochastic term has the same stated order, proving \eqref{eq:intermediate-upper}.
\end{proof}

\end{document}